\documentclass[letterpaper,notitlepage,11pt,twoside]{article} 

\usepackage{amsmath} 
\usepackage{amssymb} 
\usepackage{amsthm} 
\usepackage{bbm} 
\usepackage[backend=bibtex,maxnames=15, doi= false, eprint=false]{biblatex} 
\bibliography{biblio2}
\AtEveryBibitem{%
    \clearfield{urldate}%
   \clearfield{doi}%
		\clearfield{eprint}%
		\clearfield{ISSN}
		\clearfield{ISBN}
}
\renewbibmacro{in:}{}
\usepackage{graphicx} 

\numberwithin{equation}{section}

\usepackage{enumerate} 
\usepackage{enumitem}
\setlist[enumerate]{itemsep = -0.2em}

\usepackage{authblk} 

\usepackage[explicit]{titlesec}  
\titleformat{\subsection}[runin]
  {\normalfont\normalsize\bfseries}{\thesubsection}{0.3em}{#1.}
	
\titleformat{\section}
  {\normalfont\large}{\thesection}{0.2em.}{\centering \textsc{#1}}

\usepackage{caption}  
\usepackage{hyperref} 
\usepackage{color} 
\definecolor{MyDarkBlue}{rgb}{0,0.08,0.50}  
\definecolor{BrickRed}{rgb}{0.65,0.08,0}

\hypersetup{  
colorlinks=true,       
   linkcolor=MyDarkBlue,          
    citecolor=BrickRed,        
    filecolor=red,      
    urlcolor=cyan           
}

\usepackage[headheight=0.6cm,headsep=0.6cm,margin=1in,bottom=1.1cm]{geometry} 

\usepackage{fancyhdr}
\newtheorem{Lemma}{Lemma}[section]

\newtheorem{Proposition}[Lemma]{Proposition}

\newtheorem{Theorem}[Lemma]{Theorem}

\newtheorem{Remark}[Lemma]{Remark}

\newtheorem{Corollary}[Lemma]{Corollary}
\newtheorem{Definition}[Lemma]{Definition}

\newcommand{\sub}[1]{\boldsymbol{#1}}
\newcommand{\R}{\mathbb{R}}

\newcommand{\N}{\mathbb{N}}
\newcommand{\I}{\mathbbm{1}}
\newcommand{\pr}{\mathbb{P}}
\newcommand{\E}{\mathbb{E}}

\newcommand{\emp}{\varnothing}

\newcommand{\eqn}[1]{\begin{equation} #1 \end{equation}}

\newcommand{\e}{\mathrm{e}}

\newcommand{\La}{\mathcal{L}}
\newcommand{\CTBP}{\mathrm{CTBP}}
\newcommand{\CTBPs}{\mathrm{CTBPs}}
\newcommand{\de}{\partial}

\begin{document}

\title{\bfseries\uppercase{\large The dynamics of power laws:\\
 Fitness and aging in preferential attachment trees} }


\author[a,1]{Alessandro Garavaglia}
\author[a,2]{Remco van der Hofstad}
\author[b,3]{Gerhard Woeginger}
\affil[a]{\footnotesize Department of Mathematics and
    Computer Science, Eindhoven University of Technology, 5600 MB Eindhoven, The Netherlands}
\affil[b]{\footnotesize Department of Computer Science, RWTH Aachen,
Ahornstrasse 55, D-52074, Aachen, Germany}

\vspace{0.2cm}
\affil[$ $]{{\itshape email address}: $^1$a.garavaglia@tue.nl, $^2$rhofstad@win.tue.nl, $^3$woeginger@cs.rwth-aachen.de}

\date{}
\maketitle

\vspace{-1cm}
\begin{abstract}
Continuous-time branching processes describe the evolution of a population whose individuals generate a random number of children according to a birth process. Such branching processes can be used to understand preferential attachment models in which the birth rates are linear functions.

We are motivated by citation networks, where power-law citation counts are observed as well as aging in the citation patterns. To model this, we introduce fitness and age-dependence in these birth processes. The multiplicative fitness moderates the rate at which children are born, while the aging is integrable, so that individuals receives a {\em finite} number of children in their lifetime. We show the existence of a limiting degree distribution for such processes. In the preferential attachment case, where fitness and aging are absent, this limiting degree distribution is known to have power-law tails. We show that the limiting degree distribution has exponential tails for bounded fitnesses in the presence of integrable aging, while the power-law tail is restored when integrable aging is combined with fitness with unbounded support with at most exponential tails. In the absence of integrable aging, such processes are explosive. 
\end{abstract}

\maketitle

\thispagestyle{plain}

\section{Introduction}
\label{sec01-intr}

Preferential attachment models (PAMs) aim to describe dynamical networks. As for many real-world networks, PAMs present power-law degree distributions that arise directly from the dynamics, and are not artificially imposed as, for instance, in configuration models or inhomogeneous random graphs.

PAMs were first proposed by Albert and Barab{\'a}si \cite{ABrB}, who defined a random graph model where, at every discrete time step, a new vertex is added with one or more edges, that are attached to existing vertices with probability proportional to the degrees, i.e.,
	$$
	\pr\left(\mbox{vertex}~(n+1)~\mbox{is attached to vertex}~i\mid \mbox{graph at time}~n\right)\propto D_i(n),
	$$
where $D_i(n)$ denotes the degree of a vertex $i\in\{1,\ldots,n\}=[n]$ at time $n$. In general, the dependence of the attachment probabilities on the degree can be through a { \itshape  preferential attachment function} of the degree, also called {\itshape preferential attachment weights}. Such models are called PAMs with {\itshape general weight function}. According to the asymptotics of the weight function $w(\cdot)$, the limiting degree distribution of the graph can behave rather differently. There is an enormous body of literature showing that PAMs present power-law decay in the limiting degree distribution precisely when the weight function is affine, i.e., it is a constant plus a linear function. See e.g., \cite[Chapter 8]{vdH1} and the references therein. In addition, these models show the so-called {\em old-get-richer} effect, meaning that the vertices of highest degrees are the vertices present early in the network formation.  An extension of this model is called preferential attachment models with a {\itshape random number of edges} \cite{Dei}, where new vertices are added to the graph with a different number of edges according to a fixed  distribution, and again power-law degree sequences arise. A generalization that also gives younger vertices the chance to have high degrees is given by PAMs with {\em fitness} as studied in \cite{der2014},\cite{der16}. Borgs et al.\ \cite{Borgs} present a complete description of the limiting degree distribution of such models, with different regimes according to the distribution of the fitness, using {\itshape generalized Poly\'a's urns}. An interesting variant of a multi-type PAM is investigated in \cite{Rosen}, where the author consider PAMs where fitnesses are not i.i.d.\ across the vertices, but they are sampled according to distributions depending on the fitnesses of the ancestors. 

This work is motivated by {\em citation networks}, where vertices denote papers and the directed edges correspond to citations. For such networks, other models using preferential attachment schemes and adaptations of them have been proposed mainly in the physics literature. Aging effects, i.e., considering the {\itshape age of a vertex} in its likelihood to obtain children, have been extensively considered as the starting point to investigate their dynamics  \cite{WaMiYu}, \cite{WangYu}, \cite{Hajra}, \cite{Hajra2}, \cite{Csardi}. Here the idea is that old papers are less likely to be cited than new papers. Such aging has been observed in many citation network datasets and makes PAMs with weight functions depending only on the degree ill-suited for them. As mentioned above, such models could more aptly be called {\itshape old-get-richer} models, i.e., in general {\em old} vertices have the highest degrees. In citation networks, instead, papers with many citations appear all the time. Barab\'asi, Wang and Song \cite{BarWang} investigate a model that incorporates these effects. On the basis of empirical data, they suggest a model where the aging function follows a lognormal distribution with paper-dependent parameters, and the preferential attachment function is the identity. In \cite{BarWang}, the fitness function is estimated rather than the more classical approach where it is taken to be i.i.d.. Hazoglou, Kulkarni, Skiena Dill in \cite{Hazo} propose a similar dynamics for citation evolution , but only considering the presence of aging and cumulative advantage without fitness. 

Tree models, arising when new vertices are added with only one edge, have been analyzed in \cite{Athr}, \cite{Athr2}, \cite{RudValko}, \cite{Rudas} and lead to continuous-time branching processes (CTBP). The degree distributions in tree models show identical qualitative behavior as for the non-tree setting, while their analysis is much simpler. Motivated by this and the wish to understand the qualitative behavior of PAMs with general aging and fitness, the starting point of our model is the CTBP or tree setting. Such processes have been intensively studied, due to their applications in other fields, such as biology. Detailed and rigorous analysis of CTBPs can be found in \cite{athrBook}, \cite{Jagers}, \cite{Nerman}, \cite{RudValko}, \cite{Athr}, \cite{Athr2}, \cite{Bhamidi}. A CTBP consists of individuals, whose children are born according to certain birth processes, these processes being i.i.d.\ across the individuals in the population. The birth processes $(V_t)_{t\geq0}$ are defined in term of point or jump processes on $\N$ \cite{Jagers}, \cite{Nerman}, where the birth times of children are the jump times of the process, and the number of children of an individual at time $t\in\R^+$ is given by $V_t$. 

\begin{figure}[b]
\centering
	\begin{minipage}[c]{.4\textwidth}
					\includegraphics[width=0.9\textwidth]{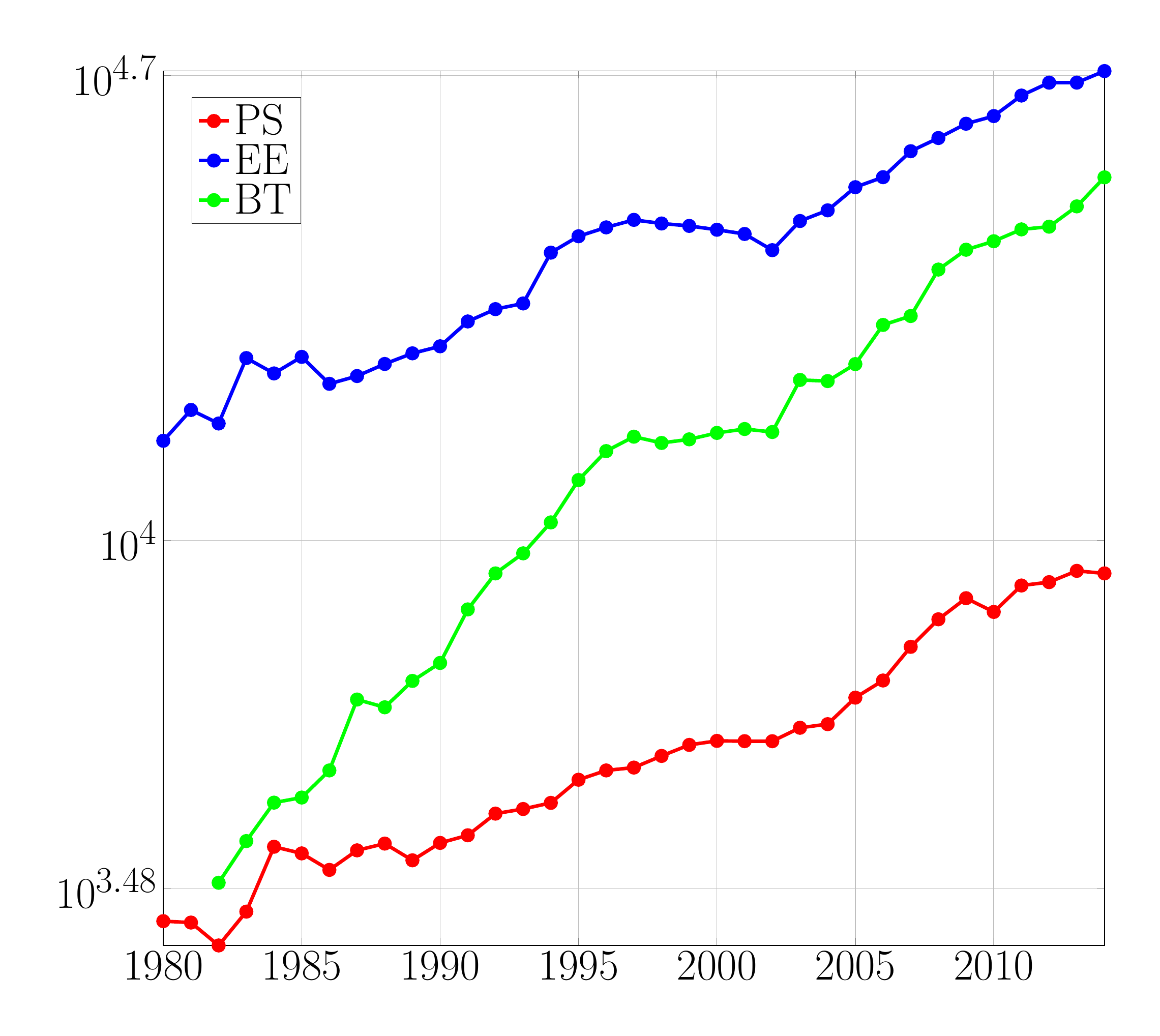} 
				\vspace{-0.3cm}
						\caption{} \label{fig-numberpublic}
						Number of publication per year (logarithmic Y axis).
	\end{minipage}%
		\hspace{15mm}%
	\begin{minipage}[c]{.45\textwidth}
				\includegraphics[width=0.98\textwidth]{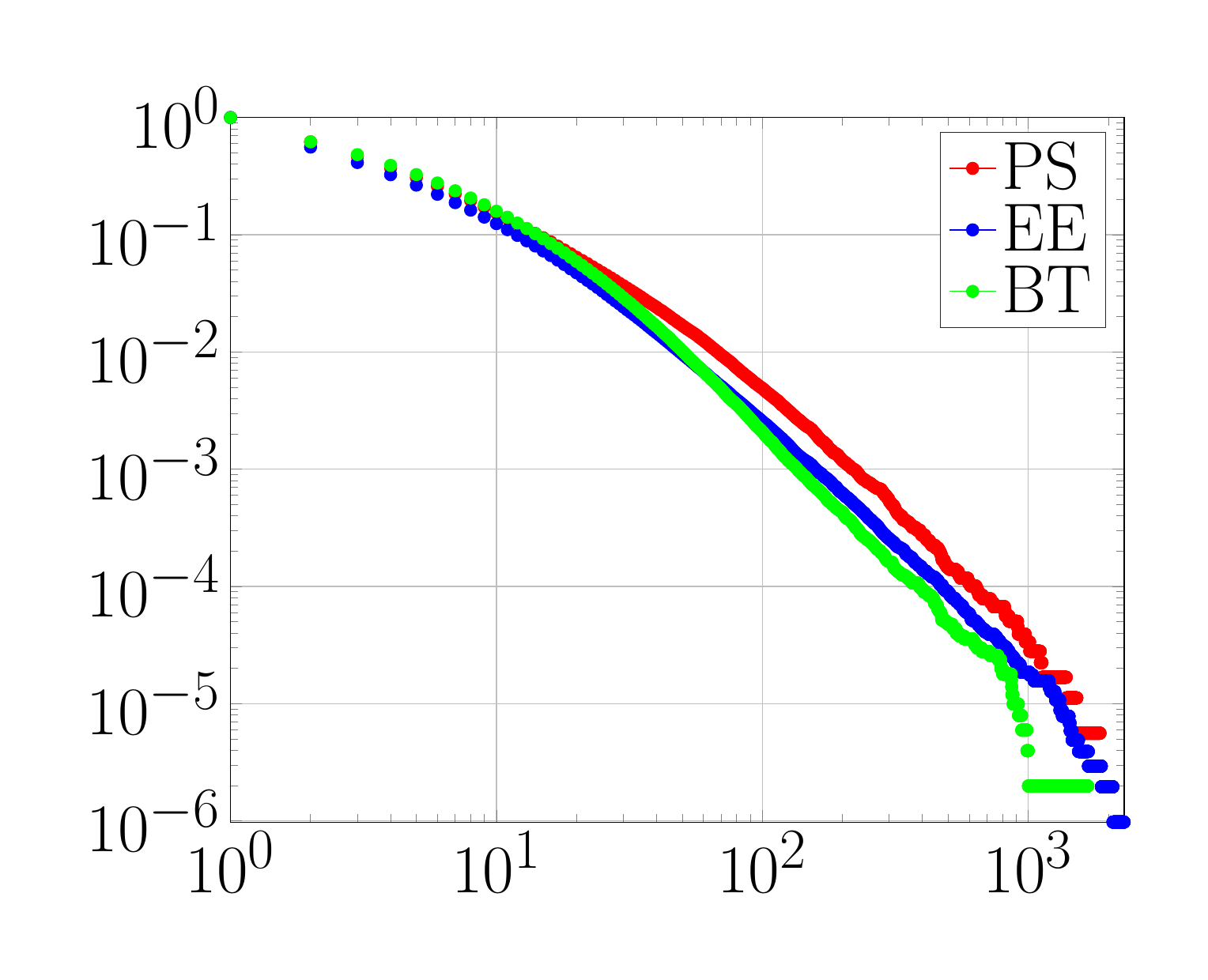} 
		\vspace{-0.3cm}
				\caption{} 		\label{fig-tailtogether}
		Loglog plot for the in-degree distribution tail in citation networks
	\end{minipage}
\end{figure}

In the literature, the CTBPs are used as a technical tool to study PAMs \cite{Athr2}, \cite{RudValko}, \cite{Rosen}. Indeed, the CTBP at the $n$th birth time follows the same law as the PAM consisting of $n$ vertices. In \cite{Athr2}, \cite{RudValko}, the authors prove an embedding theorem between branching processes and preferential attachment trees, and give a description of the degree distribution in terms of the asymptotic behavior of the weight function $w(\cdot)$. In particular, a power-law degree distribution is present in the case of (asymptotically) linear weight functions \cite{Rudas}. In the sub-linear case, instead, the degree distribution is {\itshape stretched-exponential}, while in the super-linear case it collapses, in the sense that one of the first vertices will receive all the incoming new edges after a certain step \cite{OliSpe05}. Due to the apparent exponential growth of the number of nodes in citation networks, we view the continuous-time process as the real network, which deviates from the usual perspective. Because of its motivating role in this paper, let us now discuss the empirical properties of citation networks in detail.

\subsection{Citation networks data}
\label{sec-citnet}
Let us now discuss the empirical properties of citation networks in more detail. We analyze the Web Of Science database, focusing on three different fields of science: {\em Probability and Statistics} (PS), {\em Electrical Engineering} (EE) and {\em Biotechnology and Applied Microbiology} (BT). We first point out some characteristics of citation networks that we wish to replicate in our models.

\begin{figure}[t]
	\centering
	\includegraphics[width = 0.3\textwidth]{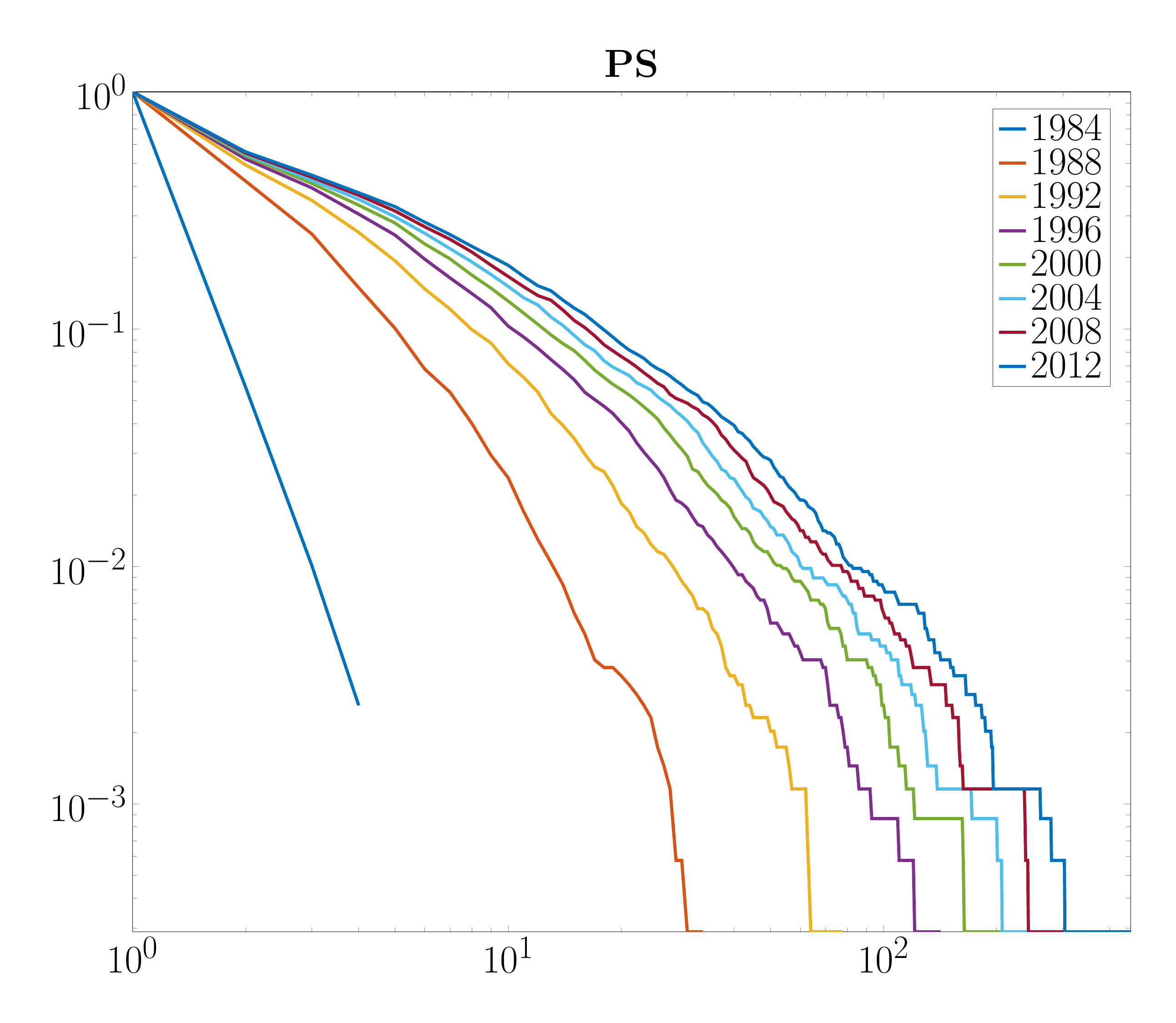}
	\includegraphics[width = 0.3\textwidth]{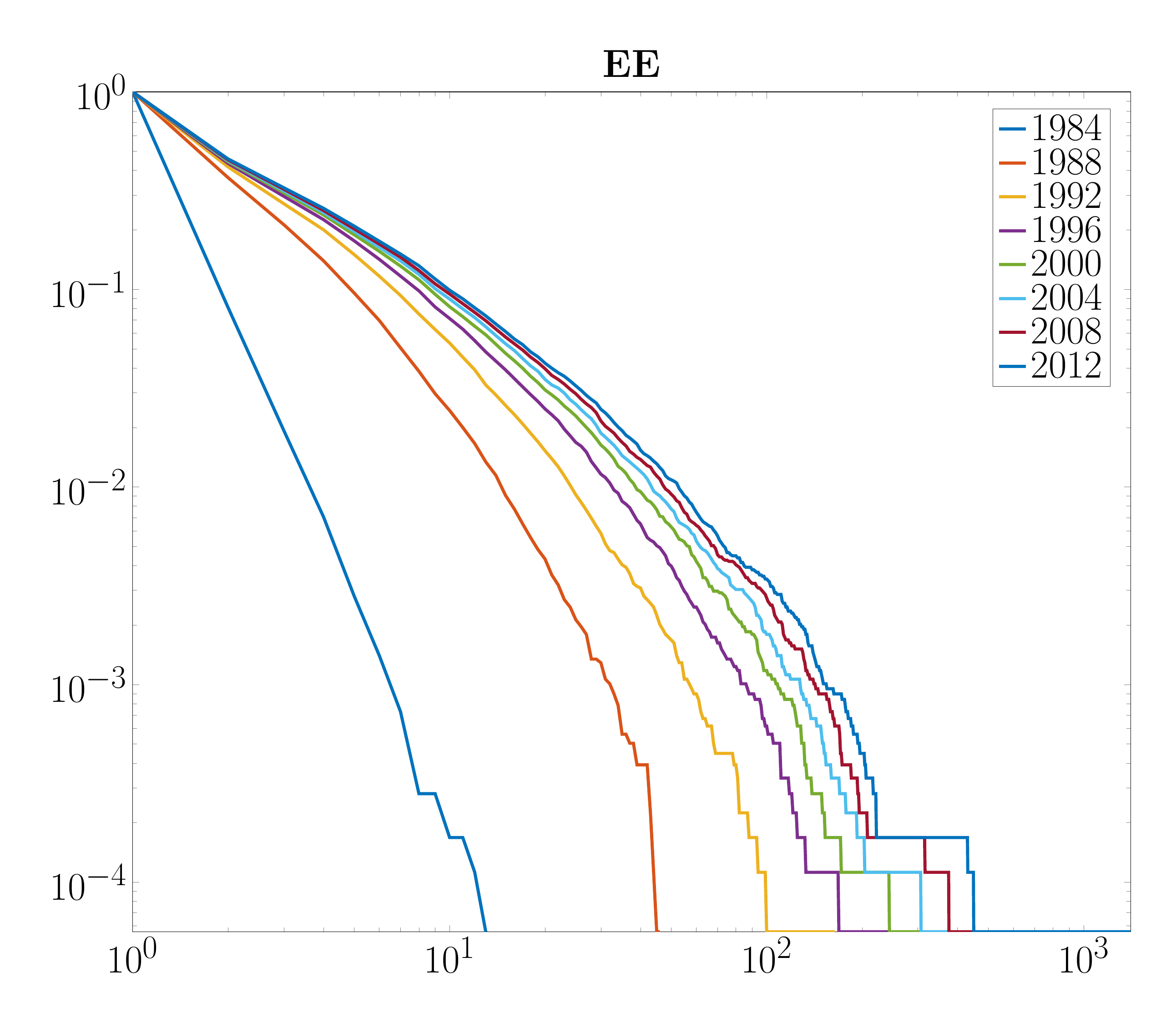}
	\includegraphics[width = 0.3\textwidth]{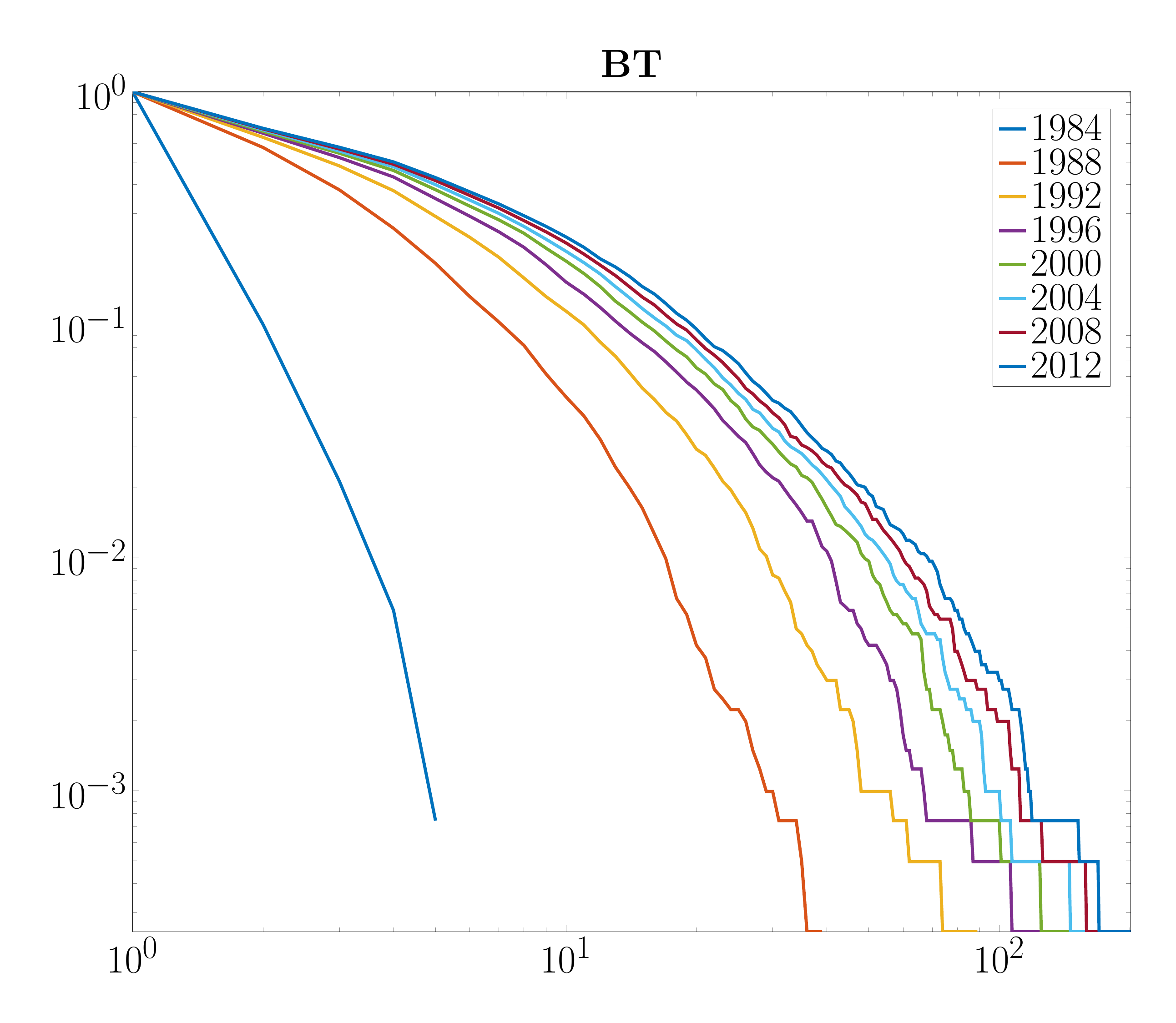}
\caption{Degree distribution for papers from 1984 over time.}
\label{fig-dunamycpowerlaw}
\end{figure}

Real-world citation networks possess five main characteristics:
\begin{figure}[b]
	\centering
	\includegraphics[width = 0.3\textwidth]{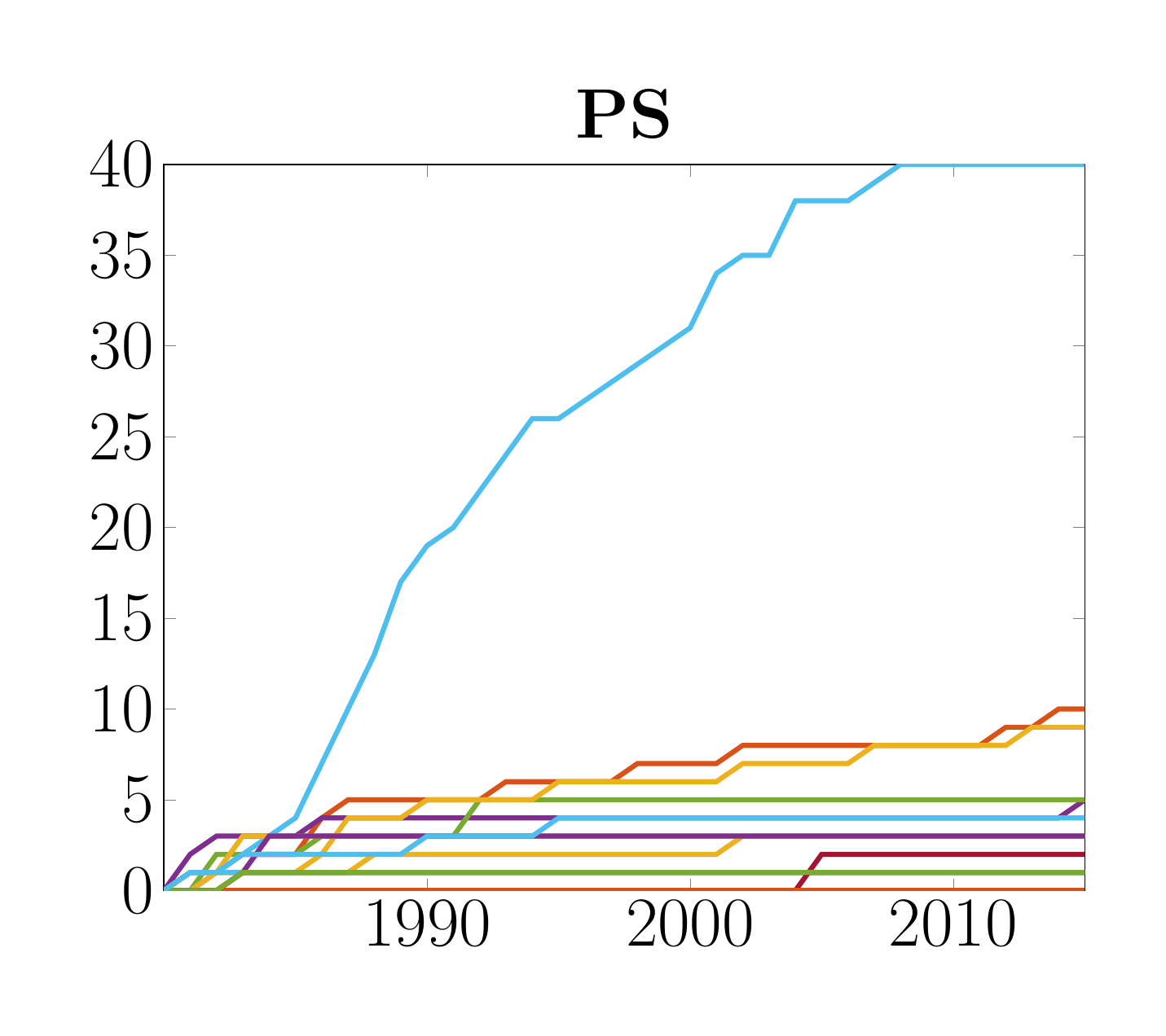}
	\includegraphics[width = 0.304\textwidth]{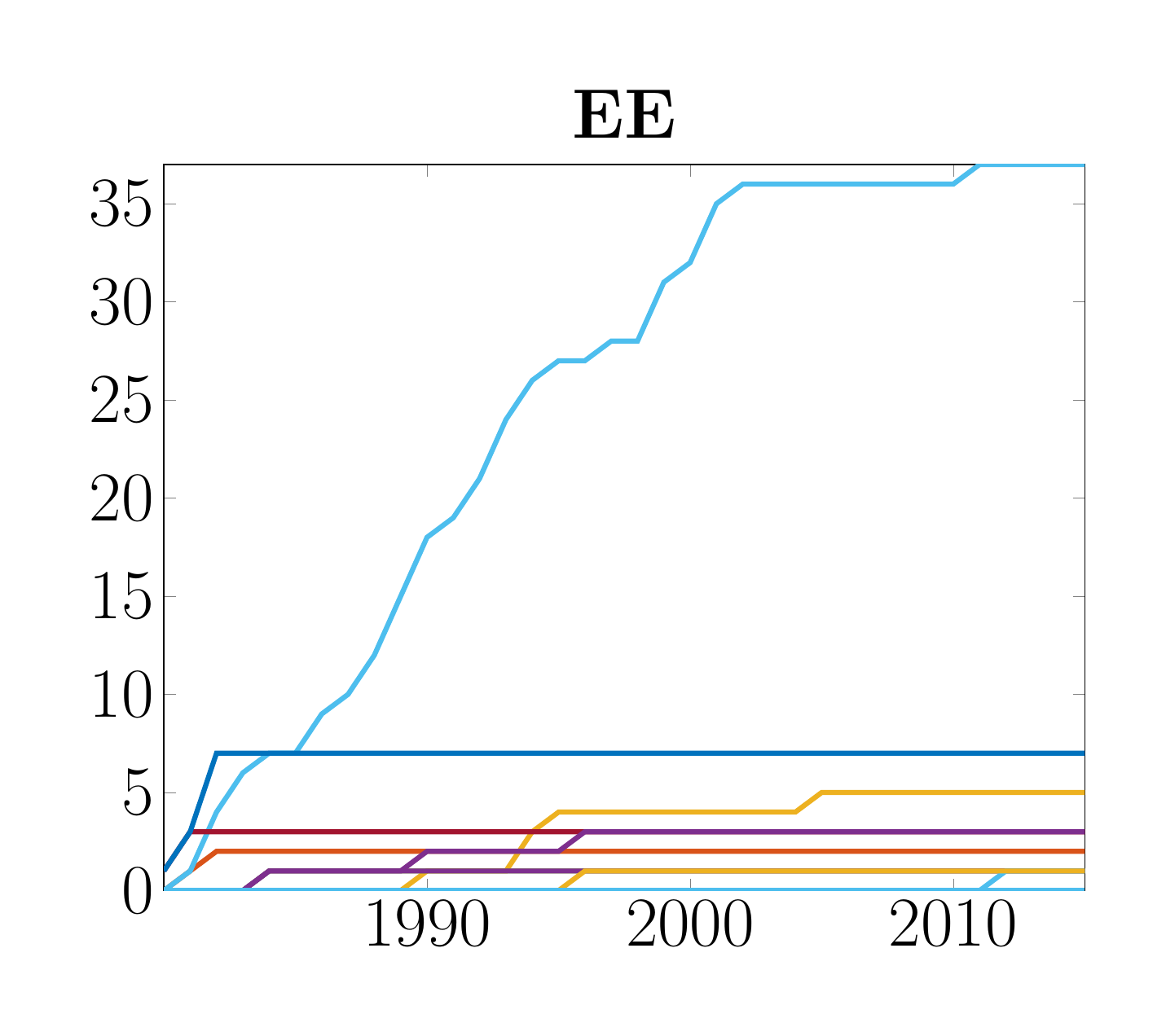}
	\includegraphics[width = 0.31\textwidth]{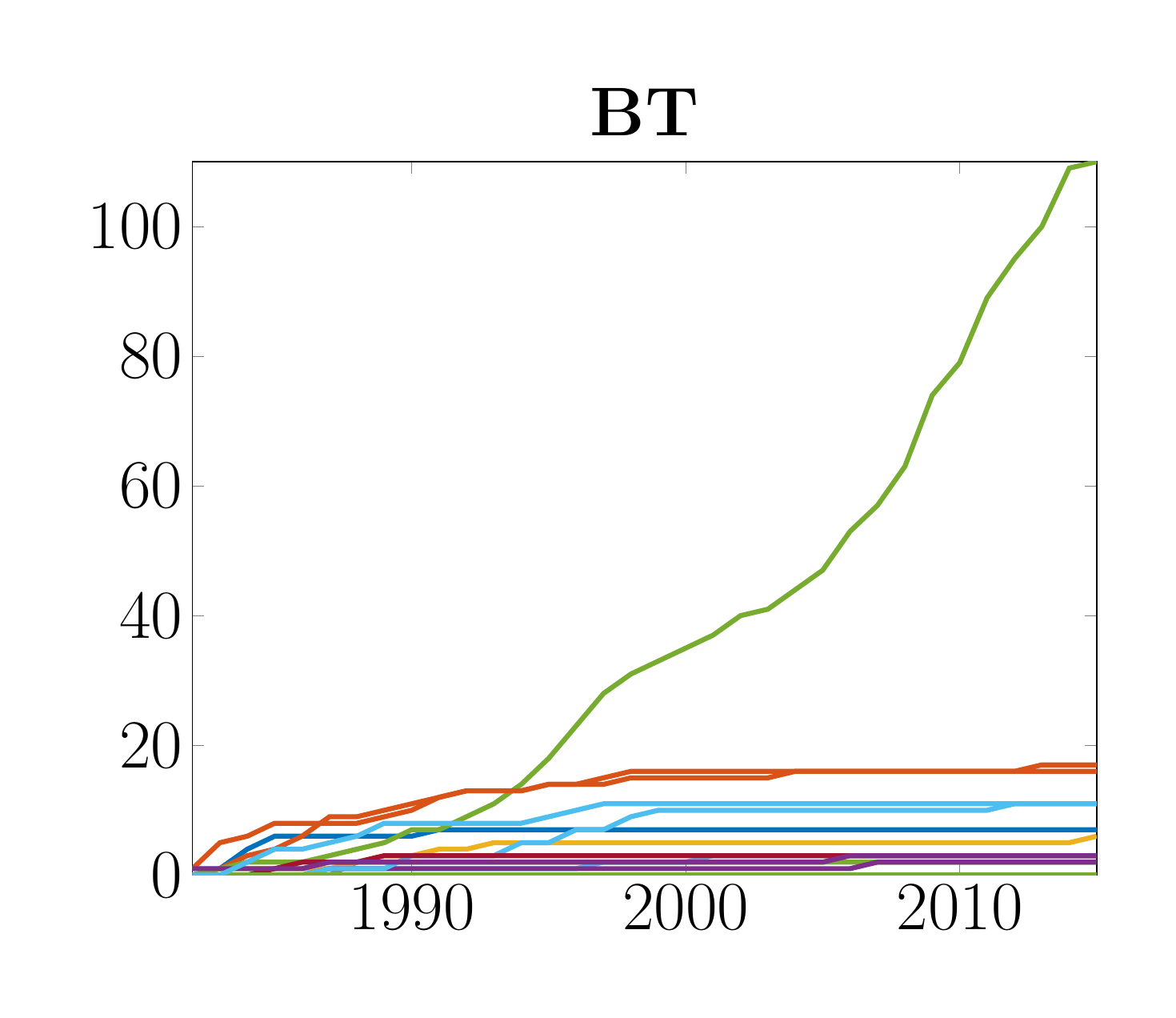}
	\caption{Time evolution for the number of citations of samples of 20 randomly chosen papers from 1980 for PS and EE, and from 1982 for BT.}
	\label{fig-randomsample}
\end{figure}

\begin{enumerate}
\item  In Figure \ref{fig-numberpublic}, we see that the number of scientific publications grows exponentially in time. While this is quite prominent in the data, it is unclear how this exponential growth arises. This could either be due to the fact that the number of journals that are listed in Web Of Science grows over time, or that journals contain more and more papers. 

\item In Figure \ref{fig-tailtogether}, we notice that these datasets have empirical power-law citation distributions. Thus, most papers attract few citations, but the amount of variability in the number of citations is rather substantial.
We are also interested in the dynamics of the citation distribution of the papers published in a given year, as time proceeds. This can be observed in Figure \ref{fig-dunamycpowerlaw}. We see a {\em dynamical power law}, meaning that at any time the degree distribution is close to a power law, but the exponent changes over time (and in fact decreases, which corresponds to heavier tails). When time grows quite large, the power law approaches a fixed value.

\begin{figure}[t]
	\centering
	\includegraphics[width = 0.31\textwidth]{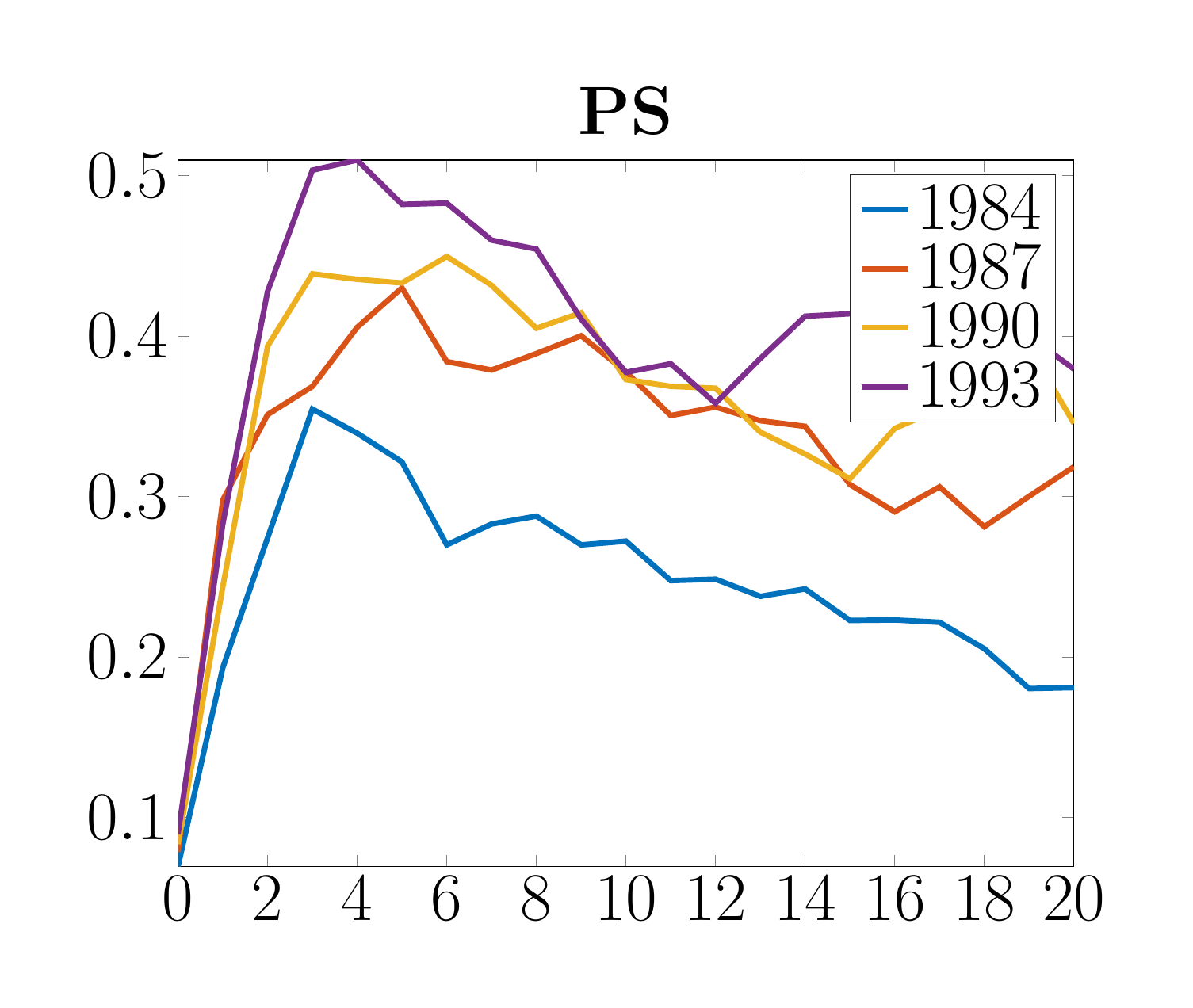}
	\includegraphics[width = 0.31\textwidth]{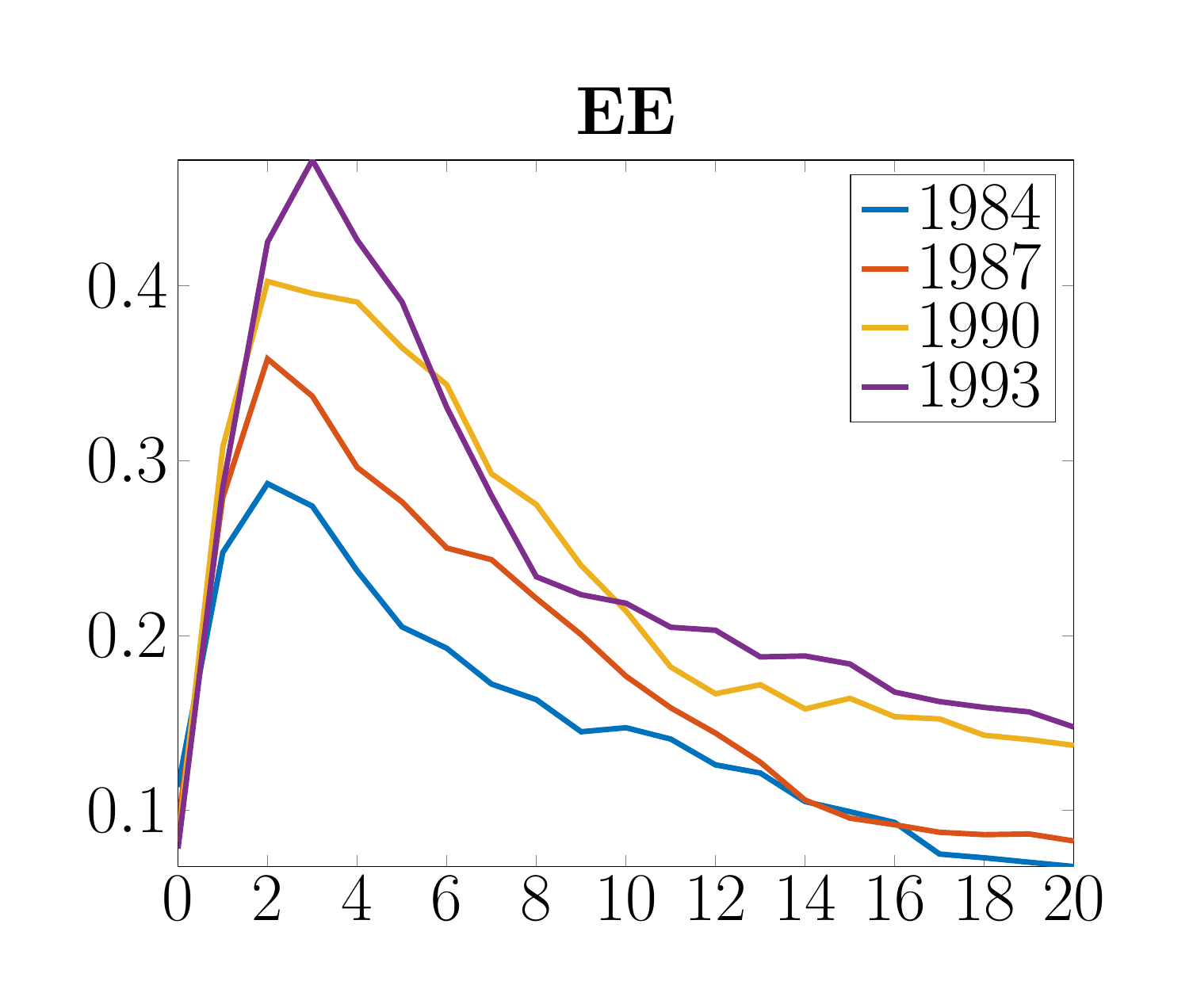}
	\includegraphics[width = 0.31\textwidth]{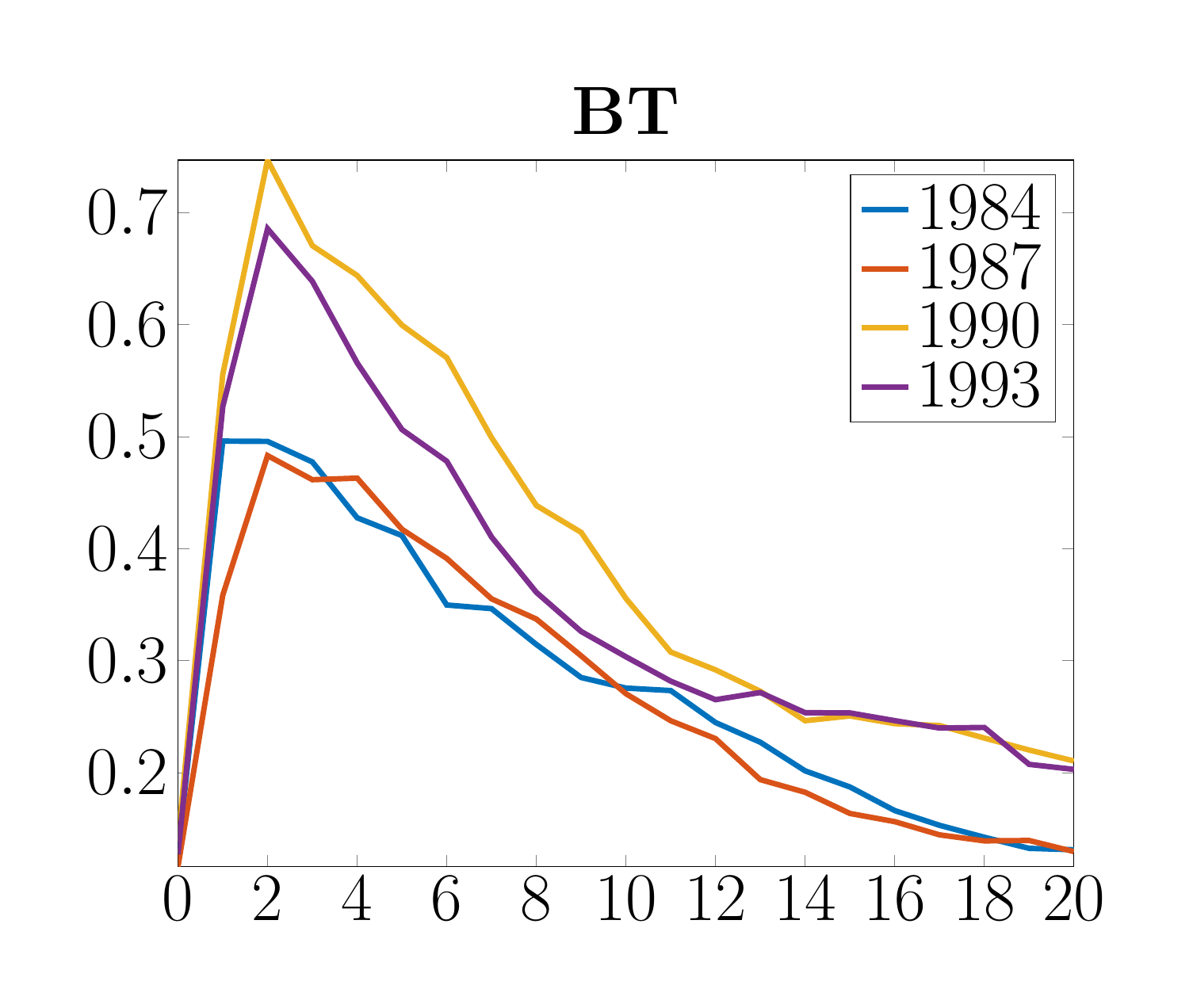}
\caption{Average degree increment over a 20-years time window for papers published in different years. PS presents an aging effect different from EE and BT, showing that papers in PS receive citations longer than papers in EE and BT.}
\label{fig-average_degree_increment}
\end{figure}

\begin{figure}[b]
	\centering
	\includegraphics[width = 0.31\textwidth]{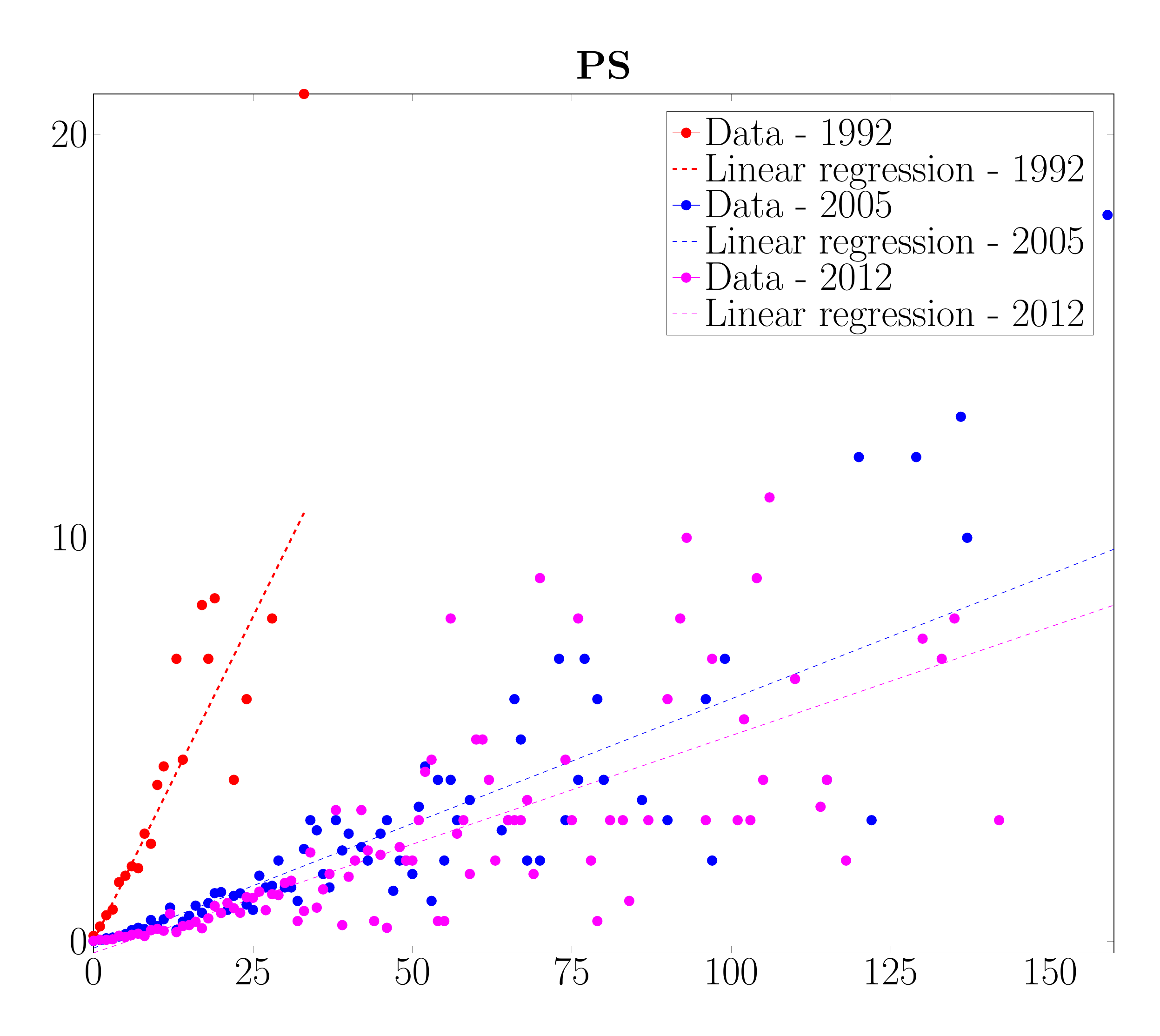}
	\includegraphics[width = 0.31\textwidth]{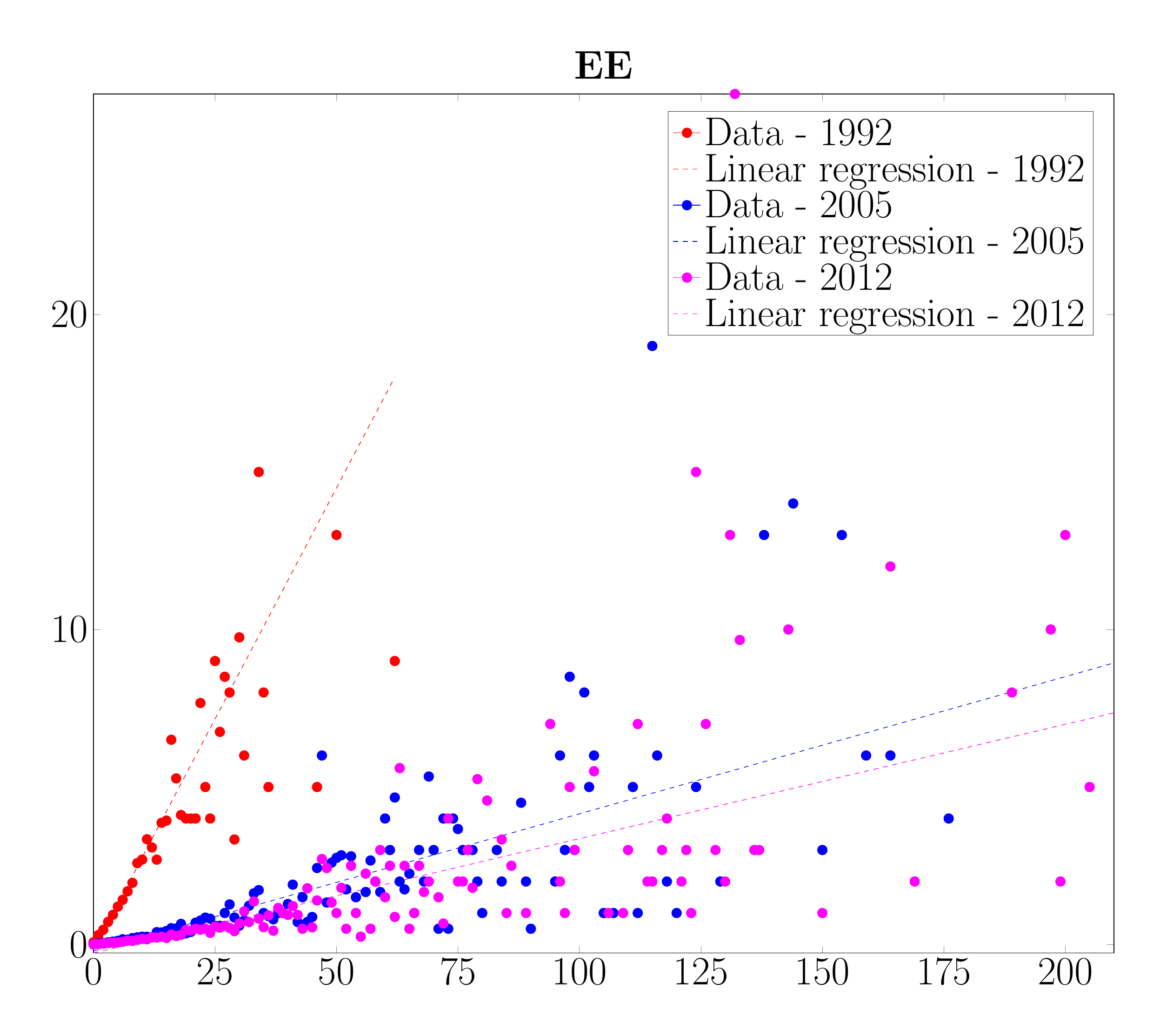}
	\includegraphics[width = 0.31\textwidth]{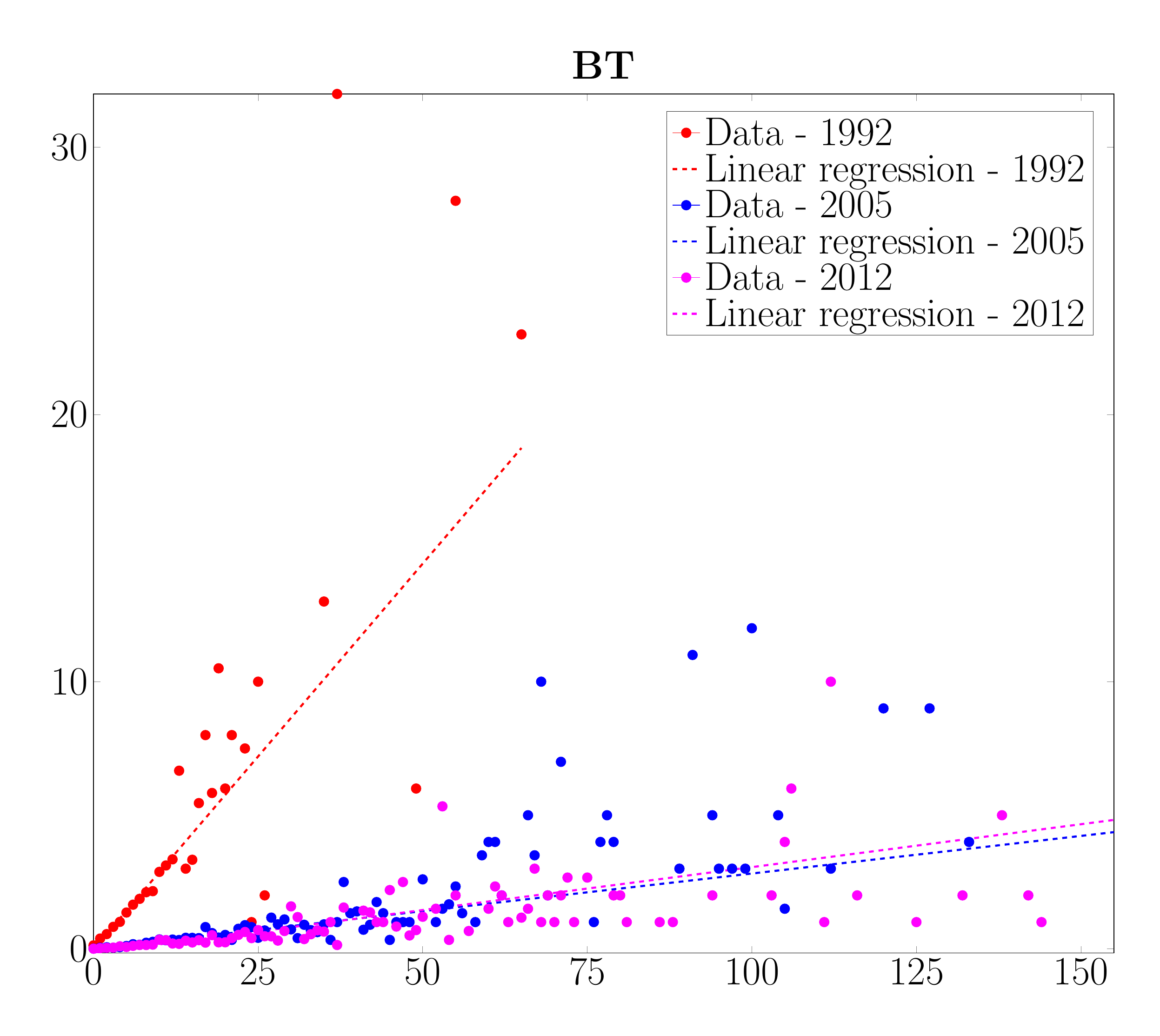}
	\caption{Linear dependence between past and future number of citations for papers from 1988.}
	\label{fig-lindep}
\end{figure}

\item In Figure \ref{fig-randomsample}, we see that the majority of papers stop receiving citations after some time, while few others keep being cited for longer times. This inhomogeneity in the evolution of node degrees is not present in classical PAMs, where the degree of {\em every} fixed vertex grows as a positive power of the graph size. Figure \ref{fig-randomsample} shows that the number of citations of papers published in the same year can be rather different, and the majority of papers actually stop receiving citations quite soon. In particular, after a first increase, the average increment of citations decreases over time (see Figure \ref{fig-average_degree_increment}). We observe a difference in this aging effect between the PS dataset and the other two datasets, due to the fact that in PS, scientists tend to cite older papers than in EE or BT. Nevertheless the average increment of citations received by papers in different years tends to decrease over time for all three datasets.


\item Figure \ref{fig-lindep} shows the linear dependence between the past number of citations of a paper and the future ones. Each plot represents the average number of citations received by papers published in 1984 in the years 1993, 2006 and 2013 according to the initial number of citations in the same year. At least for low values of the starting number of citations, we see that the average number of citations received during a year grows linearly. This suggests that the attractiveness of a paper depends on the past number of citations through an affine function.

\item A last characteristic that we observe is the lognormal distribution of the age of cited papers. In Figure \ref{fig-ageCited}, we plot the distribution of cited papers, looking at references made by papers in different years. We have used a 20 years time window in order to compare different citing years. Notice that this lognormal distribution seems to be very similar within different years, and the shape is similar over different fields.

\end{enumerate}
\begin{figure}[t]
	\centering
	\includegraphics[width = 0.31\textwidth]{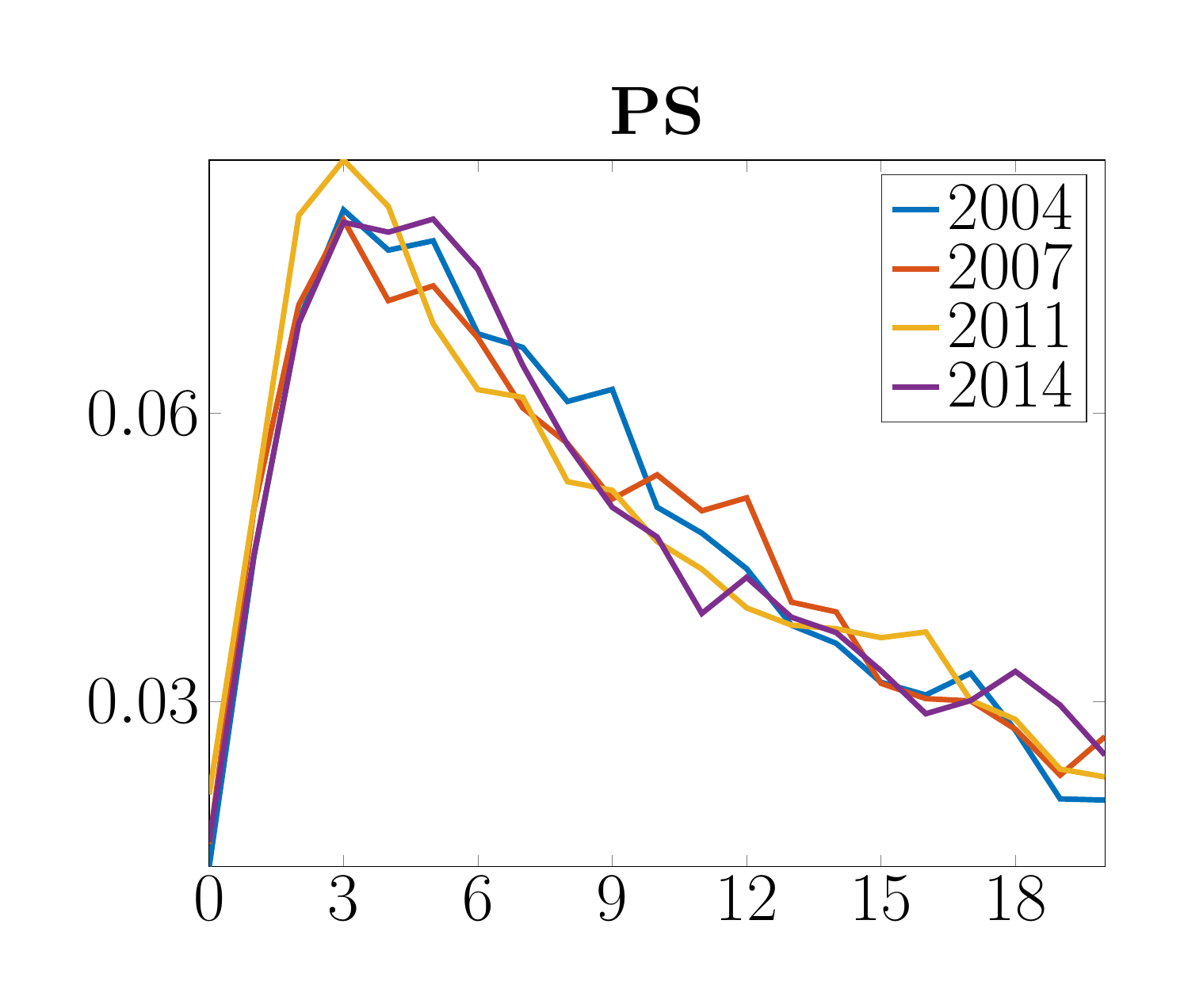}
	\includegraphics[width = 0.31\textwidth]{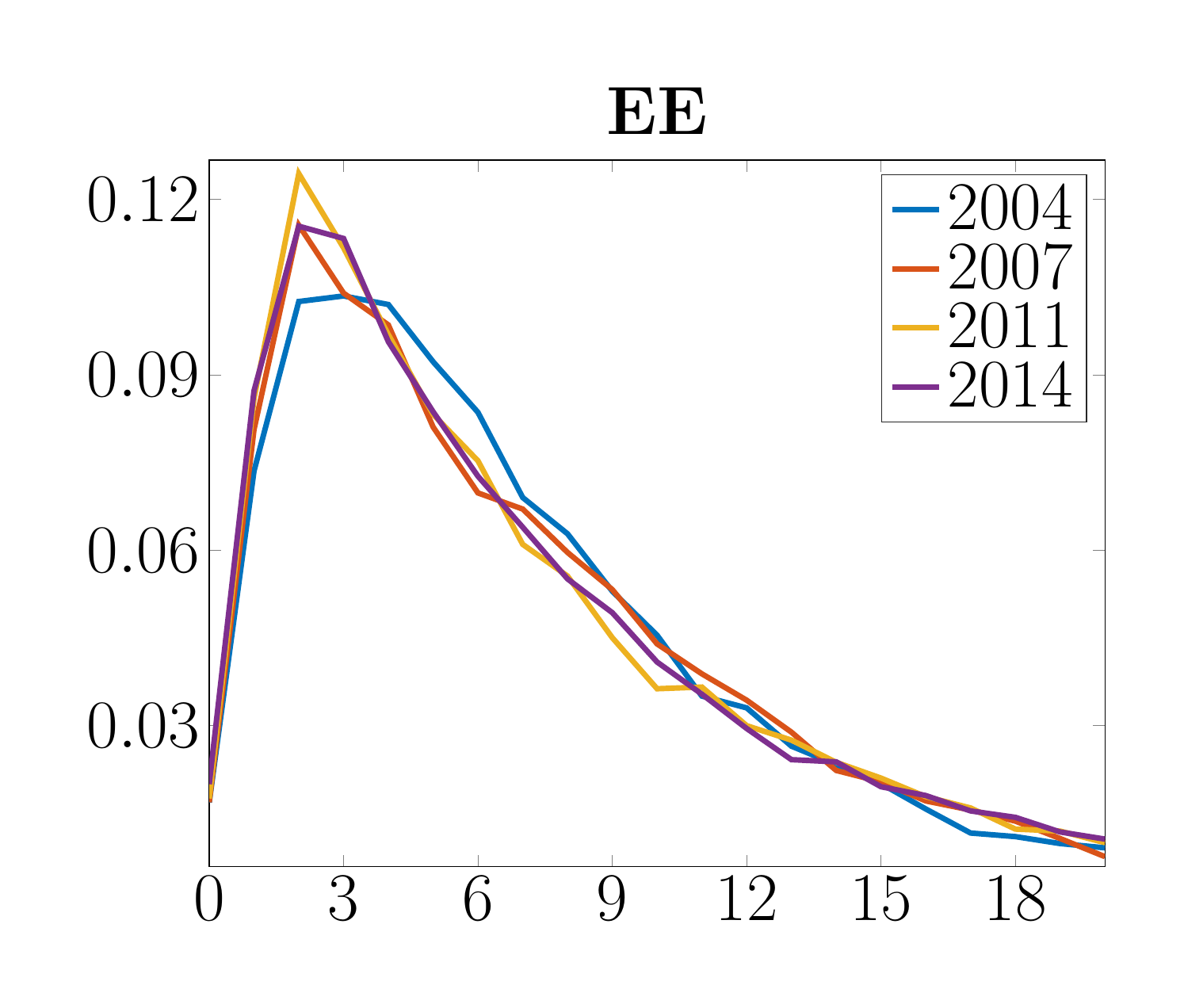}
	\includegraphics[width = 0.31\textwidth]{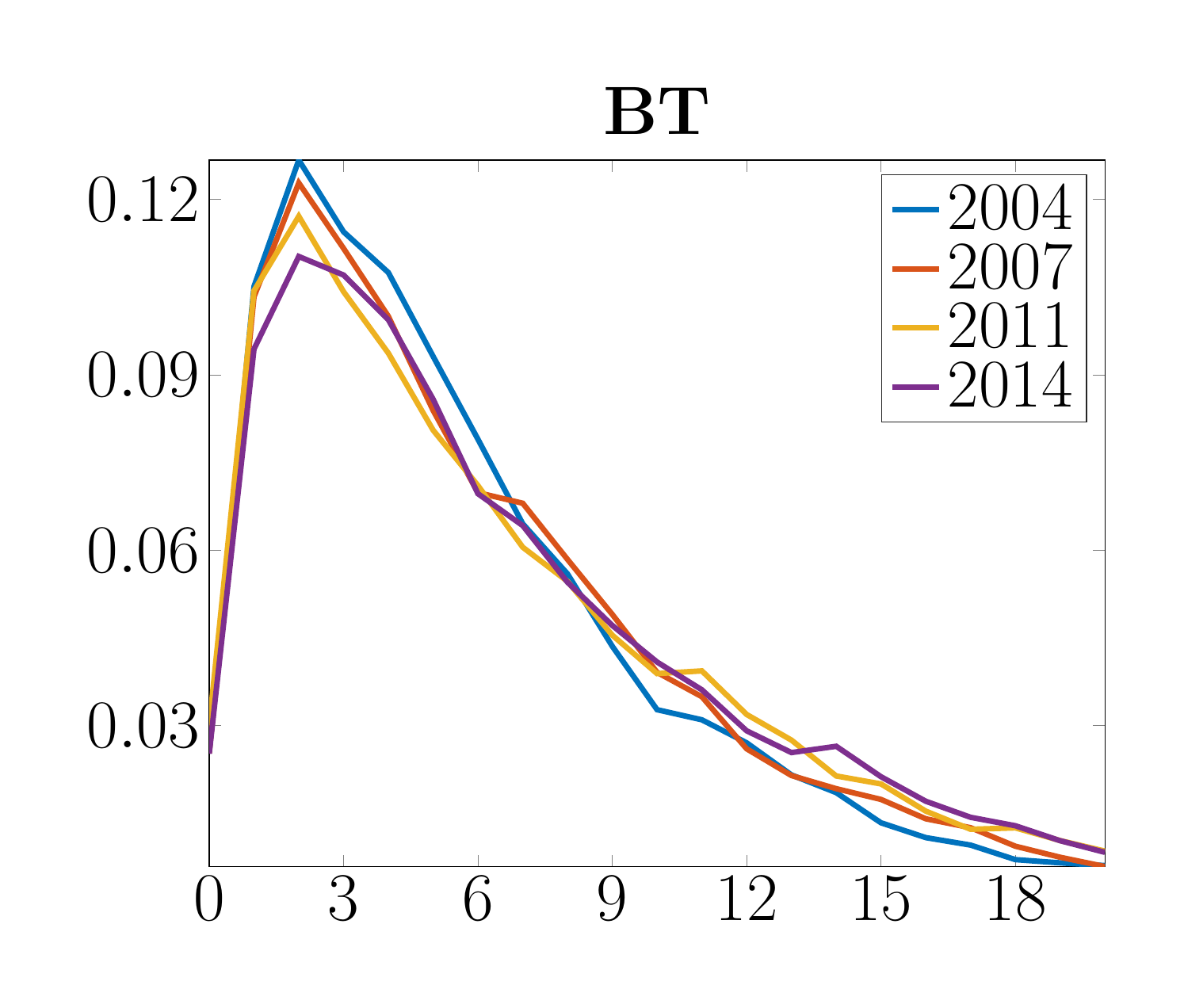}
\caption{Distribution of the age of cited papers for different citing years.}
\label{fig-ageCited}
\end{figure}

Let us now explain how we translate the above empirical characteristics into our model. First, CTBPs grow exponentially over time, as observed in citation networks. Secondly, the aging present in citation networks, as seen both in  Figures \ref{fig-randomsample} and \ref{fig-average_degree_increment}, suggests that citation rates become smaller for large times, in such a way that typical papers stop receiving citations at some (random) point in time. The hardest characteristic to explain is the power-law degree sequence. 
For this, we note that citations of papers are influenced by many {\itshape external factors} that affect the attractiveness of papers (the journal, the authors, the topic,\ldots). Since this cannot be quantified explicitly, we introduce another source of randomness in our birth processes that we call  {\itshape fitness}. This appears in the form of multiplicative factors of the attractiveness of a paper, and for lack of better knowledge, we take these factors  to be i.i.d.\ across papers, as often assumed in the literature. These assumptions are similar in spirit as the ones by Barab\'asi et al.\ \cite{BarWang}, which were also motivated by citation data, and we formalize and extend their results considerably. In particular, we give the precise conditions under which power-law citation counts are observed in this model.

Our main goal is to define CTBPs with both aging as well as random fitness that keep having a power-law decay in the in-degree distribution. Before discussing our model in detail in Section \ref{sec-mainres}, we present the heuristic ideas behind it as well as the main results of this paper.

\subsection{Our main contribution}
\label{sec-maincontribution}
The crucial point of this work is to show that it is possible to obtain power-law degree distributions in preferential attachment trees where the birth process {\em is not just depending on an asymptotically linear weight sequence}, in the presence of {\em integrable aging} and {\em fitness}. Let us now briefly explain how these two effects change the behavior of the degree distribution. 
\medskip

\paragraph{\bf Integrable aging and affine preferential attachment without fitness.}
In the presence of aging but without fitness, we show that the aging effect substantially slows down the birth process. In the case of affine weights, aging destroys the power-law of the stationary regime, generating a limiting distribution that consists of a power law with exponential truncation. We prove this under reasonable conditions on the underlying aging function (see Lemma \ref{Lem-adaptLap-age}).
\medskip

\paragraph{\bf Integrable aging and super-linear preferential attachment without fitness.}
Since the aging destroys the power-law of the affine PA case, it is natural to ask whether the combination of integrable aging and {\em super-linear} weights restores the power-law limiting degree distribution. Theorem \ref{th-explosive} states that this is not the case, as super-linear weights imply explosiveness of the branching process, which is clearly unrealistic in the setting of citation networks  (here, we call a weight sequence $k\mapsto f_k$ {\em super-linear} when $\sum_{k\geq 1} 1/f_k<\infty$). This result is quite general, because it holds for  {\em any} integrable aging function. Due to this, it is impossible to obtain power-laws from super-linear preferential attachment weights. This suggests that (apart from slowly-varying functions), affine preferential attachment weights have the strongest possible growth, while maintaining exponential (and thus, in particular, non-explosive) growth.
\medskip

\paragraph{\bf Integrable aging and affine preferential attachment with unbounded fitness.}
In the case of aging and fitness, the asymptotic behavior of the limiting degree distribution is rather involved. We estimate the asymptotic decay of the limiting degree distribution with affine weights in Proposition \ref{prop-pkasym_fitage}. With the example fitness classes analyzed in Section \ref{sec-fitexamples}, we prove that power-law tails are possible in the setting of aging and fitness, at least when the fitness has roughly exponential tail. So far, PAMs with fitness required the support of the fitness distribution to be {\em bounded}. The addition of aging allows the support of the fitness distribution to be unbounded, a feature that seems reasonable to us in the context of citation networks. Indeed, the relative attractivity of one paper compared to another one can be enormous, which is inconsistent with a bounded fitness distribution. While we do not know precisely what the necessary and sufficient conditions are on the aging and the fitness distribution to assure a power-law degree distribution, our results suggests that affine PA weights with integrable aging and fitnesses with at most an exponential tail in general do so, a feature that was not observed before.

\medskip

\paragraph{\bf Dynamical power laws.} In the case of fitness with exponential tails, we further observe that the number of citations of a paper of age $t$ has a power-law distribution with an exponent that depends on $t$. We call this a {\em dynamical power law}, and it is a possible explanation of the dynamical power laws observed in citation data (see Figure \ref{fig-dunamycpowerlaw}).
\medskip

\paragraph{\bf Universality.} An interesting and highly relevant observation in this paper is that the limiting degree distribution of preferential attachment trees with aging and fitness shows a high amount of {\em universality}. Indeed, for integrable aging functions, the dependence on the precise choice of the aging function seems to be minor, except for the total integral of the aging function. Further, the dependence on fitness is quite robust as well.

\section{Our model and main results}
\label{sec-mainres}
In this paper we introduce the effect of aging and fitness in $\CTBP$ populations, giving rise to directed trees. Our model is motivated by the study of {\em citation networks}, which can be seen as directed graphs. Trees are the simplest case in which we can see the effects of aging and fitness. Previous work has shown that PAMs can be obtained from PA trees by collapsing, and their general degree structure can be quite well understood from those in trees.
For example, PAMs with fixed out-degree $m\geq 2$ can be defined through a collapsing procedure, where a vertex in the multigraph is formed by $m\in\N$ vertices in the tree (see \cite[Section 8.2]{vdH1}). In this case, the limiting degree distribution of the PAM preserve the structure of the tree case (\cite[Section 8.4]{vdH1}, \cite[Section 5.7]{Bhamidi}).
This explains the relevance of the tree case results for the study of the effect of aging and fitness in PAMs.
It could be highly interesting to prove this rigorously.

\subsection{Our CTBP model}
\label{se-mainres}
CTBPs represent a population made of individuals producing children independently from each other, according to i.i.d. copies of a birth process on $\N$. We present the general theory of CTBPs in Section \ref{sec-generalth}, where we define such processes in detail and we refer to general results that are used throughout the paper. In general, considering a birth process $(V_t)_{t\geq0}$ on $\N$, every individual in the population has an i.i.d. copy of the process $(V_t)_{t\geq0}$, and the number of children of individual $x$ at time $t$ is given by the value of the process $V^x_t$. We consider birth processes defined by a sequence of weights $(f_k)_{k\in\N}$ describing the birth rates. Here, the time between the $k$th and the $(k+1)$st jump is exponentially distributed with parameter $f_k$. The behavior of the whole population is determined by this sequence. 

The fundamental theorem for the CTBPs that we study is Theorem \ref{th-expogrowth} quoted in Section \ref{sec-generalth}. It states that, under some hypotheses on the birth process $(V_t)_{t\geq0}$, the population grows exponentially in time, which nicely fits the exponential growth of scientific publications as indicated in Figure \ref{fig-numberpublic}. Further, using a so-called {\em random vertex characteristic} as introduced in \cite{Jagers}, a complete class of properties of the population can be described, such as the fraction of individuals having $k$ children, as we investigate in this paper. The two main properties are stated in Definitions \ref{def-supercr} and \ref{def-malthus}, and are called {\em supercritical} and {\em Malthusian} properties. These properties require that there exists a positive value $\alpha^*$ such that 
	$$
	\E\left[V_{T_{\alpha^*}}\right]=1, \quad \quad \mbox{ and }\quad \quad
	 -\left.\frac{d}{d\alpha}\E\left[V_{T_{\alpha}}\right]\right|_{\alpha=\alpha^*}<\infty,
	$$
where $T_\alpha$ denotes an exponentially distributed random variable with rate $\alpha$ independent of the process $(V_t)_{t\geq0}$.
The unique value $\alpha^*$ that satisfies both conditions is called the {\em Malthusian parameter}, and it describes the exponential growth rate of the population size. The aim is to investigate the ratio
	$$
	\frac{\mbox{number of individuals with}~k~\mbox{children at time}~t}{\mbox{size total population at time}~t}.
	$$
According to Theorem \ref{th-expogrowth}, this ratio converges almost surely to a deterministic limiting value $p_k$. The sequence $(p_k)_{k\in\N}$, which we refer to as the limiting degree distribution of the CTBP (see Definition \ref{def-limitdistr}), is given by 
	$$
	p_k = \E\left[\pr\left(V_u=k\right)_{u=T_{\alpha^*}}\right].
	$$
The starting idea of our model of citation networks is that, given the history of the process up to time $t$,
	\eqn{
	\label{for-heuristic}
	\mbox{the rate of an individual of age}~t~\mbox{and}~k~\mbox{children to generate a new child is}~Yf_kg(t),
	}
where $f_k$ is a non-decreasing PA function of the degree, $g$ is an integrable function of time, and $Y$ is a positive random variable called fitness. Therefore, the likelihood to generate children increases by having many children and/or a high fitness, while it is reduced by age. 

Recalling Figure \ref{fig-lindep}, we assume that the PA function $f$ is affine, so $f_k = ak+b$. In terms of a PA scheme, this implies
$$
	\pr\left(\mbox{a paper cites another with past}~k~\mbox{citations}~|~\mbox{past}\right) \approx \frac{n(k) (ak+b)}{A},
$$
where $n(k)$ denotes the number of papers with $k$ past citations, and $A$ is the normalization factor. Such behavior has already been observed by Redner \cite{Redner3}  and Barab\'asi et al.\ \cite{BarJeoNed}).

We assume throughout the paper that the aging function $g$ is integrable. In fact, we start by the fact that the age of cited papers is lognormally distributed (recall Figure \ref{fig-ageCited}).
By normalizing such a distribution by the average increment in the number of citations of papers in the selected time window, we identify a universal function $g(t)$. Such function can be approximated by a lognormal shape of the form
$$
	g(t) \approx c_1\e^{-c_2(\log(t+1)-c_3)^2},
$$
for $c_1$, $c_2$ and $c_3$ field-dependent parameters. In particular, from the procedure used to define $g(t)$, we observe that
$$
	g(t)\approx \frac{\mbox{number of references to year}~t}{\mbox{number of papers of age}~t}~
		\frac{\mbox{total number of papers considered}}{\mbox{total number of references considered}},
$$
which means in terms of PA mechanisms that
$$
	\pr\left(\mbox{a paper cites another of age}~t~|~\mbox{past}\right)\approx
		\frac{n(t)g(t)}{B},
$$
where $B$ is the normalization factor, while this time $n(t)$ is the number of paper of age $t$. This suggests that the citing probability depends on age through a lognormal aging function $g(t)$, which is integrable. This is one of the main assumptions in our model, as we discuss in Section \ref{sec-maincontribution}.

It is known from the literature (\cite{Rudas}, \cite{RudValko}, \cite{Athr}) that CTBPs show power-law limiting degree distributions when the infinitesimal rates of jump depend only on a sequence $(f_k)_{k\in\N}$ that is asymptotically {\em linear}. Our main aim is to investigate whether power-laws can also arise in branching processes that include aging and fitness. The results are organized as follows. In Section \ref{sec-res-aging}, we discuss the results for CTBPs with aging in the absence of fitness. In Section \ref{sec-res-aging-fitness}, we present the results with aging and fitness. In Section \ref{sec-res-aging-fitness-exp}, we specialize to fitness with distributions with exponential tails, where we show that the limiting degree distribution is a power law with a {\em dynamic} power-law exponent. 

\subsection{Results with aging without fitness}
\label{sec-res-aging}
In this section, we focus on aging in PA trees in the absence of fitness. The aging process can then be viewed as a time-changed stationary birth process (see Definition \ref{def-statnonfit}). A stationary birth process is a stochastic process $(V_t)_{t\geq0}$ such that, for $h$ small enough, 
$$
	\pr\left(V_{t+h}=k+1 \mid V_t=k\right) = f_kh+o(h).
$$
In general, we assume that $k\mapsto f_k$ is increasing. The {\em affine case} arises when $f_k = ak+b$ with $a,b>0$. By our observations in Figure \ref{fig-lindep}, as well as related works (\cite{Redner3}, \cite{BarJeoNed}), the affine case is a reasonable approximation for the attachment rates in citation networks.

For a stationary birth process $(V_t)_{t\geq0}$, under the assumption that it is supercritical and Malthusian, the limiting degree distribution $(p_k)_{k\in\N}$ of the corresponding branching process is given by
	\eqn{
	\label{deg-distr-PA-tree}
	p_k = \frac{\alpha^*}{\alpha^* + f_k}\prod_{i=0}^{k-1}\frac{f_i}{\alpha^* + f_i}.
	}
For a more detailed description, we refer to Section \ref{sec-stat-nonfit}. Branching processes defined by stationary processes (with no aging effect) have a so-called {\itshape old-get-richer} effect. As this is not what we observe in citation networks (recall Figure \ref{fig-randomsample}), we want to introduce {\em aging} in the reproduction process of individuals. The aging process arises by adding 
age-dependence in the infinitesimal transition probabilities:

\begin{Definition}[Aging birth processes]
\label{def-nonstatbirth}
Consider a non-decreasing PA sequence $(f_k)_{k\in\N}$ of positive real numbers and an aging function $g\colon \R^+\rightarrow\R^+$. We call a stochastic process $(N_t)_{t\geq 0}$ an {\em aging birth process} (without fitness) when
\begin{enumerate}
	\item $N_0=0$, and $N_t\in\N$ for all $t\in\N$;
	\item $N_t\leq N_s$ for every $t\leq s$;
	\item for fixed $k\in\N$ and $t\geq0$, as $h\rightarrow0$,
	$$
		\pr\left(N_{t+h}=k+1 \mid N_t=k\right) = f_k g(t)h + o(h).
	$$
\end{enumerate}
\end{Definition}
Aging processes are time-rescaled versions of the corresponding stationary process defined by the same sequence $(f_k)_{k\in\N}$. In particular, for any $t\geq0$, $N_t$ has the same distribution as $V_{G(t)}$, where $G(t) = \int_0^tg(s)ds$.
In general, we assume that the aging function is {\em integrable}, which means that $G(\infty) := \int_0^\infty g(s)ds<\infty$. This implies that the number of children of a single individual in its entire lifetime has distribution $V_{G(\infty)}$, which is finite in expectation. In terms of citation networks, this assumption is reasonable since we do not expect papers to receive an infinite number of citations ever (recall Figure \ref{fig-average_degree_increment}). Instead, for the stationary process $(V_t)_{t\geq0}$ in Definition \ref{def-statnonfit}, we have that $\pr$-a.s.\ $V_t\rightarrow\infty$, so that also the aging process diverges $\pr$-a.s.\ when $G(\infty) = \infty$. 

For aging processes, the main result is the following theorem, proven in Section \ref{sec-existence}. In its statement, we rely on the Laplace transform of a function. For a precise definition of this notion, we refer to Section \ref{sec-generalth}:

\begin{Theorem}[Limiting distribution for aging branching processes]
\label{th-limitdist-nonstat}
Consider an integrable aging function and  a PA sequence $(f_k)_{k\in\N}$. Denote the corresponding aging birth process by $(N_t)_{t\geq0}$. Then, assuming that $(N_t)_{t\geq0}$ is supercritical and Malthusian, the limiting degree distribution of the branching process $\sub{N}$ defined by the birth process $(N_t)_{t\geq0}$ is given by
	\eqn{
	\label{for-nonstatdist}
	p_k = \frac{\alpha^*}{\alpha^*+f_k\hat{\mathcal{L}}^g(k,\alpha^*)}\prod_{i=0}^{k-1}\frac{f_i\hat{\mathcal{L}}^g(i,\alpha^*)}
																																										{\alpha^*+f_{i}\hat{\mathcal{L}}^g(i,\alpha^*)},
	}
where $\alpha^*$ is the Malthusian parameter of $\sub{N}$. Here, the sequence of coefficients $(\hat{\mathcal{L}}^g(k,\alpha^*))_{k\in\N}$ appearing in \eqref{for-nonstatdist} is given by
	\eqn{
	\label{for-lkratio}
	\hat{\mathcal{L}}^g(k,\alpha^*) = \frac{\mathcal{L}(\pr\left(N_\cdot=k\right)g(\cdot))(\alpha^*)}
	{\mathcal{L}(\pr\left(N_\cdot=k\right))(\alpha^*)},
	}
where, for $h\colon \R^+\rightarrow\R$, $\mathcal{L}(h(\cdot))(\alpha)$ denotes the Laplace transform of $h$.\\
Further, considering a fixed individual in the branching population, the total number of children in its entire lifetime is distributed as $V_{G(\infty)}$, where $G(\infty)$ is the $L^1$-norm of $g$. 
\end{Theorem}


The limiting degree distribution maintains a product structure as in the stationary case (see \eqref{deg-distr-PA-tree} for comparison). Unfortunately, the analytic expression for the probability distribution $(p_k)_{k\in\N}$ in \eqref{for-nonstatdist} given by the previous theorem is not explicit. In the stationary case, the form reduces to the simple expression in \eqref{deg-distr-PA-tree}.

\begin{figure}[t]
		\centering
		\includegraphics[width = 0.35\textwidth]{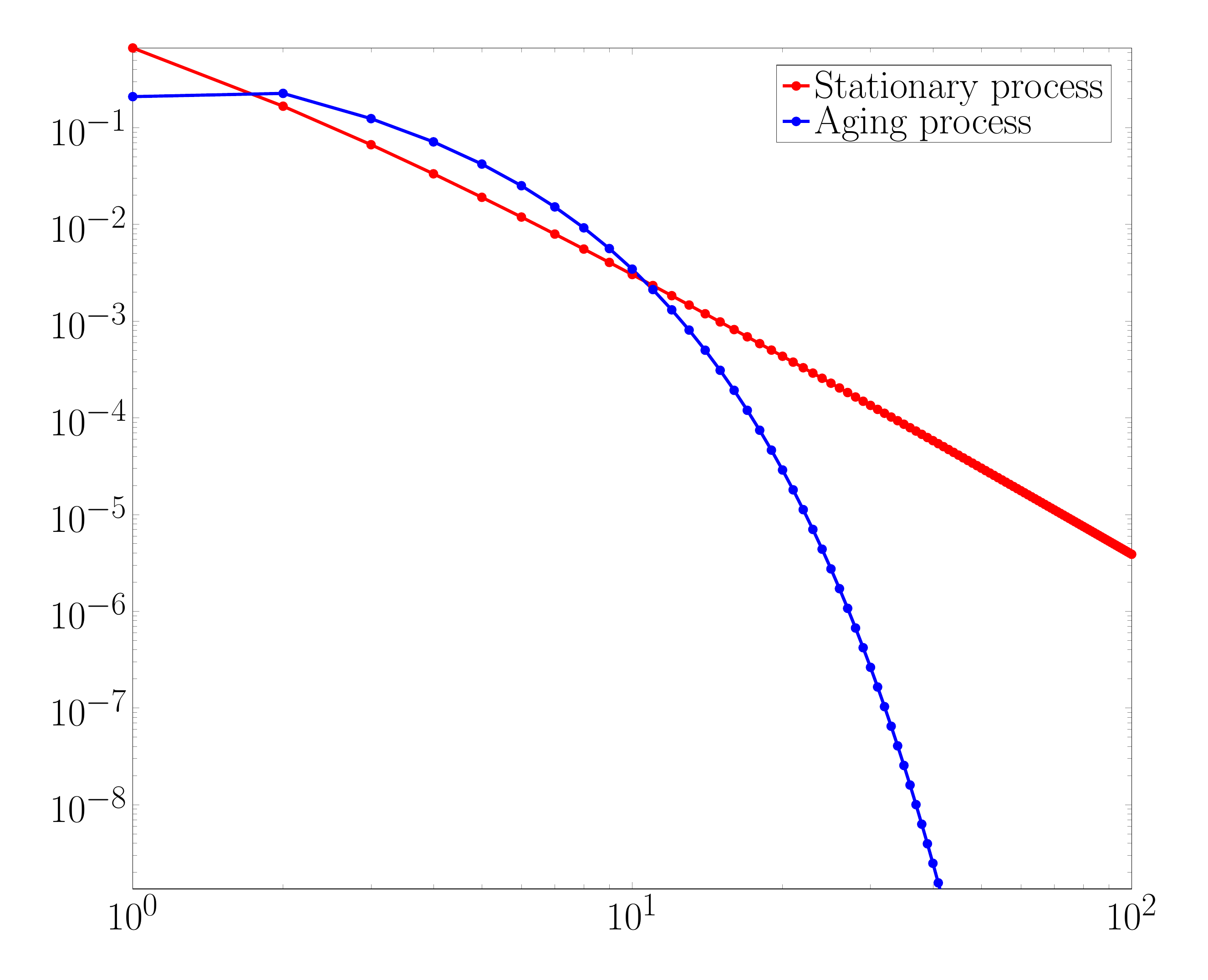}
		\includegraphics[width = 0.35\textwidth]{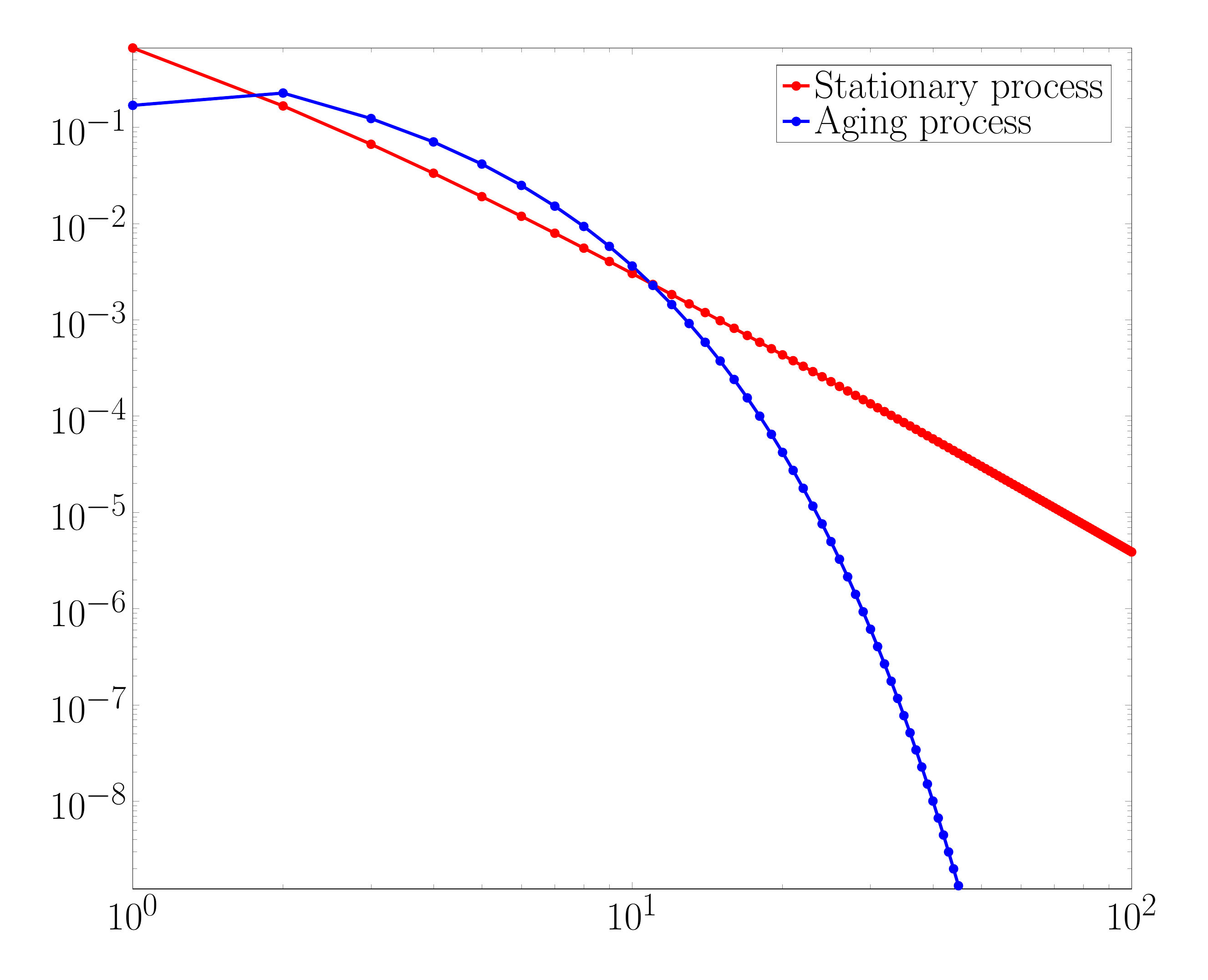}
\caption{ Examples of stationary and aging limit degree distributions}
\label{fig-distr-agepower}
\end{figure}

In general, the asymptotics of the coefficients $(\hat{\mathcal{L}}^g(k,\alpha^*))_{k\in\N}$ is unclear, since it depends both on the aging function $g$ as well as the PA weight sequence $(f_k)_{k\in\N}$ itself in an intricate way. In particular, we have no explicit expression for the ratio in \eqref{for-lkratio}, except in special cases. In this type of birth process, the cumulative advantage given by $(f_k)_{k\in\N}$ and the aging effect given by $g$ cannot be separated from each other. 

Numerical examples in Figure \ref{fig-distr-agepower} show how  aging destroys the power-law degree distribution. In each of the two plots, the limiting degree distribution of a stationary process with affine PA weights  gives a power-law degree distribution, while the process with two different integrable aging functions does not. 
In the examples we have used $g(t) = \e^{-\lambda t}$ and $g(t) = (1+t)^{-\lambda}$ for some $\lambda>1$, and we observe the insensitivity of the limiting degree distribution with respect to $g$. The distribution given by \eqref{for-nonstatdist} can be seen as the limiting degree distribution of a CTBP defined by preferential attachment weight $(f_k\hat{\mathcal{L}}^g(k,\alpha^*))_{k\in\N}$. This suggests that $f_k\hat{\mathcal{L}}^g(k,\alpha^*)$ is not asymptotically linear in $k$. 

In Section \ref{sec-examples-age}, we investigate the two examples in Figure \ref{fig-distr-agepower}, showing that the limiting degree distribution has exponential tails, a fact that we know in general just as an upper bound (see Lemma \ref{lem-exp-tails-aging-bd-fitness}).
\medskip

In order to apply the general CTBP result in Theorem \ref{th-expogrowth} below, we need to prove that an aging process $(N_t)_{t\geq0}$ is supercritical and Malthusian. We show in Section \ref{sec-existence} that, for an integrable aging function $g$, the corresponding process is supercritical if and only if
	\eqn{
	\label{cond-eg_inf1}
	\lim_{t\rightarrow\infty}\E\left[V_{G(t)}\right] = \E\left[V_{G(\infty)}\right]>1.
	}
Condition \eqref{cond-eg_inf1} heuristically suggests that the process $(N_t)_{t\geq0}$ has a Malthusian parameter if and only if the expected number of children in the entire lifetime of a fixed individual is larger than one, which seems quite reasonable. In particular, such a result follows from the fact that if $g$ is integrable, then the Laplace transform is always finite for every $\alpha>0$. In other words, since $N_{T_{\alpha^*}}$ has the same distribution as $V_{G(T_{\alpha^*})}$, $\E[N_{T_{\alpha^*}}]$ is always bounded by $\E[V_{G(\infty)}]$. This implies that $G(\infty)$ cannot be too small, as otherwise the Malthusian parameter would not exist, and the CTBP would die out $\pr$-a.s..

The aging effect obviously slows down the birth process, and makes the limiting degree distribution have exponential tails for affine preferential attachment weights. One may wonder whether the power-law degree distribution could be restored when $(f_k)_{k\in\N}$ grows super-linearly instead.  Here, we say that a sequence of weights $(f_k)_{k\in\N}$ grows super-linearly when $\sum_{k\geq1}1/f_k<\infty$ (see Definition \ref{def-superlin}). In the super-linear case, however, the branching process is {\em explosive}, i.e., for every individual the probability of generating an infinite number of children in finite time is $1$. In this situation, the Malthusian parameter does not exist, since the Laplace transform of the process is always infinite. One could ask whether, by using an integrable aging function, this explosive behavior is destroyed. The answer to this question is given by the following theorem:

\begin{Theorem}[Explosive aging branching processes for super-linear attachment weights]
\label{th-explosive}
Consider a stationary process $(V_t)_{t\geq0}$ defined by super-linear PA weights $(f_k)_{k\in\N}$. For any aging function $g$, the corresponding non-stationary process $(N_t)_{t\geq0}$ is explosive.
\end{Theorem}

The proof of Theorem \ref{th-explosive} is rather simple, and is given in Section \ref{sec-aging-gen-PA}. 
We investigate the case of affine PA weights $f_k = ak+b$ in more detail in Section \ref{sec-adaptedLap}. Under a hypothesis on the regularity of the integrable aging function, in Proposition \ref{prop-pkage_asym}, we give the asymptotic behavior of the corresponding limiting degree distribution. In particular, as $k\rightarrow\infty$, 
$$
	p_k = C_1\frac{\Gamma(k+b/a)}{\Gamma(k+1)}\e^{-C_2k}\mathcal{G}(k,g)(1+o(1)),
$$
for some positive constants $C_1,C_2$. The term $\mathcal{G}(k,g)$ is a function of $k$, the aging function $g$ and its derivative. The precise behavior of such term depends crucially on the aging function. Apart from this, we notice that aging generates an exponential term in the distribution, which explains the two examples in Figure \ref{fig-distr-agepower}. In Section \ref{sec-examples-age}, we prove that the two limiting degree distributions in Figure \ref{fig-distr-agepower} indeed have exponential tails.

\subsection{Results with aging and fitness}
\label{sec-res-aging-fitness} 
The analysis of birth processes becomes harder when we also consider fitness. First of all, we define the birth process with aging and fitness as follows:

\begin{Definition}[Aging birth process with fitness]
\label{def-nonstatfit}
Consider a birth process $(V_t)_{t\geq0}$. Let $g\colon \R^+\rightarrow\R^+$ be an aging function, and $Y$ a positive random variable. The process
$M_t := V_{YG(t)}$
is called a birth process with {\em aging and fitness}.
\end{Definition}

Definition \ref{def-nonstatfit} implies that the infinitesimal jump rates of the process $(M_t)_{t\geq0}$ are as in \eqref{for-heuristic}, so that the birth probabilities of an individual depend on the PA weights, the age of the individual and on its fitness. Assuming that the process $(M_t)_{t\geq0}$ is supercritical and Malthusian, we can prove the following theorem:

\begin{Theorem}[Limiting degree distribution for aging and fitness]
\label{th-degagefit}
Consider a process $(M_t)_{t\geq0}$ with integrable aging function $g$, fitnesses that are i.i.d. across the population, and assume that it is supercritical and Malthusian with Malthusian parameter $\alpha^*$. Then, the limiting degree distribution for the corresponding branching process is given by
	$$
	p_k = \E\left[\frac{\alpha^*}{\alpha^*+f_kY\hat{\mathcal{L}}(k,\alpha^*,Y)}\prod_{i=0}^{k-1}
			\frac{f_{i}Y\hat{\mathcal{L}}(i,\alpha^*,Y)}{\alpha^*+f_iY\hat{\mathcal{L}}(i,\alpha^*,Y)}\right].
	$$
For a fixed individual, the distribution $(q_k)_{k\in\N}$ of the number of children it generates over its entire lifetime is given by
	$$
	q_k = \pr\left(V_{YG(\infty)}=k\right).
	$$
\end{Theorem}
Similarly to Theorem \ref{th-limitdist-nonstat}, the sequence $(\hat{\mathcal{L}}(k,\alpha^*,Y))_{k\in\N}$ is given by
	$$
	\hat{\mathcal{L}}(k,\alpha^*,Y) = \left(\frac{\mathcal{L}(\pr\left(V_{uG(\cdot)}=k\right)g(\cdot))(\alpha^*)}
	{\mathcal{L}(\pr\left(V_{uG(\cdot)}=k\right))(\alpha^*)}\right)_{u=Y},
	$$
where again $\mathcal{L}(h(\cdot))(\alpha)$ denotes the Laplace transform of a function $h$. Notice that in this case, with the presence of the fitness $Y$, this sequence is no longer deterministic but random instead. We still have the product structure for $(p_k)_{k\in\N}$ as in the stationary case, but now we have to average over the fitness distribution.

We point out that Theorem \ref{th-limitdist-nonstat} is a particular case of Theorem \ref{th-degagefit}, when we consider $Y\equiv 1$. We state the two results as separate theorems to improve the logic of the presentation. We prove Theorem \ref{th-degagefit} in Section \ref{sec-pf-aging-fitness-gen}. In Section \ref{sec-aging-gen-PA} we show how Theorem \ref{th-limitdist-nonstat} can be obtained from Theorem  \ref{th-degagefit}, and in particular how Condition \eqref{cond-eg_inf1} is obtained from the analogous Condition \eqref{cond-eg_inf2} stated below for general fitness distributions.

With affine PA weights, in Proposition \ref{prop-pkasym_fitage}, we can identify the asymptotics of the limiting degree distribution we obtain. This is proved by similar techniques as in the case of aging only, even though the result cannot be expressed so easily. In particular, we prove
	$$
	p_k =\frac{\Gamma(k+b/a)}{\Gamma(b/a)\Gamma(k+1)}\frac{2\pi}{\sqrt{\mathrm{det}(kH_k(t_k,s_k))}}
	\e^{-k\Psi_k(t_k,s_k)}\pr\left(\mathcal{N}_1\geq -t_k,\mathcal{N}_2\geq -s_k\right)(1+o(1)),
	$$ 
where the function $\Psi_k(t,s)$ depends on the aging function, the density $\mu$ of the fitness and $k$. The point $(t_k,s_k)$ is the absolute minimum of $\Psi_k(t,s)$, $H_k(t,s)$ is the Hessian matrix of $\Psi_k(t,s)$, and $(\mathcal{N}_1,\mathcal{N}_2)$ is a bivariate normal vector with covariance matrix related to $H_k(t,s)$. We do not know the necessary and sufficient conditions for the existence of such a minimum $(t_k,s_k)$. However, in Section \ref{sec-fitexamples}, we consider two examples where we can apply this result, and we show that it is possible to obtain power-laws for them.

In the case of aging and fitness, the supercriticality condition in \eqref{cond-eg_inf1} is replaced by the analogous condition that 
	\eqn{
	\label{cond-eg_inf2}
	\E\left[V_{YG(t)}\right]<\infty \quad\mbox{for every }t\geq0 \quad \quad \mbox{and}
	\quad \lim_{t\rightarrow\infty}\E\left[V_{YG(t)}\right]>1.
	}

Borgs et al.\ \cite{Borgs} and Dereich \cite{der16}, \cite{der2014} prove results on stationary CTBPs with fitness. In these works, the authors investigate models with affine dependence on the degree and bounded fitness distributions. This is necessary since unbounded distributions with affine weights are explosive and thus {\em do not have Malthusian parameter}. We refer to Section \ref{sec-fitconditions} for a more precise discussion of the conditions on fitness distributions. 

In the case of integrable aging and fitness, it is possible to consider affine PA weights, even with unbounded fitness distributions, as exemplified by \eqref{cond-eg_inf2}. In particular, for $f_k = ak+b$,
$$
	\E[V_t] = \frac{b}{a}\left(\e^{at}-1\right).
$$
As a consequence, Condition \eqref{cond-eg_inf2} can be written as
	\eqn{
	\label{for-AgeFitLap}
	\forall t\geq0 \quad \E\left[\e^{aYG(t)}\right]<\infty \quad\quad \mbox{and}\quad\quad \lim_{t\rightarrow\infty}\E\left[\e^{aYG(t)}\right]>1+		\frac ab.
	}
The expected value $\E\left[\e^{aYG(t)}\right]$ is the moment generating function of $Y$ evaluated in $aG(t)$. In particular, a necessary condition to have a Malthusian parameter is that the moment generating function is finite on the interval $[0,aG(\infty))$. As a consequence, denoting $\E[\e^{sY}]$ by $\varphi_Y(s)$, we have effectively moved from the condition of having bounded distributions to the condition
\eqn{
\label{for-varpYcond}
	\varphi_Y(x)<+\infty \quad 
			\mbox{on}\quad [0,aG(\infty)),\quad\quad \mbox{and}\quad \lim_{x\rightarrow aG(\infty)}\varphi_Y(x)>\frac{a+b}{a}.
}
Condition \eqref{for-varpYcond} is weaker than assuming a bounded distribution for the fitness $Y$, which means we can consider a larger class of distributions for the aging and fitness birth processes. Particularly for citation networks, it seems reasonable to have unbounded fitnesses, as the relative popularity of papers varies substantially.

\subsection{Dynamical power-laws for exponential fitness and integrable aging}
\label{sec-res-aging-fitness-exp}
In Section \ref{sec-fitexamples} we introduce three different classes of fitness distributions, for which we give the asymptotics for the limiting degree distribution of the corresponding $\CTBP$. 

The first class is called {\em heavy-tailed}. Recalling \eqref{for-varpYcond}, any distribution $Y$ in this class satisfies, for any $t>0$, 
\eqn{
	\label{def-powerlawfit}
		\varphi_Y(t) = \E\left[\e^{tY}\right] = +\infty.
}
These distributions have a tail that is thicker than exponential. For instance, power-law distributions belong to this first class. Similarly to unbounded distributions in the stationary regime, such distributions generate {\em explosive} birth processes, independent of the choice of the integrable aging functions.

The second class is called {\em sub-exponential}.  The density $\mu$ of a distribution $Y$ in this class satisfies 
\eqn{
	\label{def-subexpfitness}
		\forall ~\beta>0, \quad \quad \lim_{s\rightarrow+\infty}\mu(s)\e^{\beta s}=0.
}
An example of this class is the density $\mu(s) = C\e^{-\theta s^{1+\varepsilon}}$, for some $\varepsilon,C,\theta>0$. For such density, we show in Proposition \ref{prop-subexpfit} that the corresponding limiting degree distribution  has a thinner tail than a power-law. 

The third class is called {\em general-exponential}. The density $\mu$ of a distribution $Y$ in this class is of the form
\eqn{
\label{def-generalexpfit}
		\mu(s) = Ch(s)\e^{-\theta s},
}
where $h(s)$ is a twice differentiable function such that $h'(s)/h(s)\rightarrow0$ and $h''(s)/h(s)\rightarrow0$ as $s\rightarrow\infty$, and $C$ is a normalization constant. For instance, exponential and Gamma distributions belong to this class. From \eqref{for-varpYcond}, we know that in order to obtain a non-explosive process, it is necessary to consider the exponential rate $\theta>aG(\infty)$. We will see that the limiting degree distribution obeys a power law as $\theta>aG(\infty)$ with tails becoming thinner when $\theta$ increases.

For a distribution in the general exponential class, as proven in Proposition \ref{prop-expfit_general}, the limiting degree distribution of the corresponding $\CTBP$ has a power-law term, with slowly-varying corrections given by the aging function $g$ and the function $h$.  We do not state Propositions \ref{prop-expfit_general} and \ref{prop-subexpfit} here, as these need notation and results from Section \ref{sec-adaptedLap}. For this reason, we only state the result for the special case of purely exponential fitness distribution:

\begin{Corollary}[Exponential fitness distribution]
\label{th-degexpfitness}
Let the fitness distribution $Y$ be exponentially distributed with parameter $\theta$, and let $g$ be an integrable aging function. Assume that the corresponding birth process $(M_t)_{t\geq0}$ is supercritical and Malthusian. Then, the limiting degree distribution $(p_k)_{k\in\N}$ of the corresponding CTBP $\sub{M}$ is
$$
	p_k = \E\left[\frac{\theta}{\theta+f_kG(T_{\alpha^*})}\prod_{i=0}^{k-1}\frac{f_iG(T_{\alpha^*})}{\theta+f_iG(T_{\alpha^*})}\right].
$$
The distribution $(q_k)_{k\in\N}$ of the number of children of a fixed individual in its entire lifetime is given by
$$
	q_k = \frac{\theta}{\theta+G(\infty) f_k}\prod_{i=0}^{k-1}\frac{G(\infty) f_i}{\theta+G(\infty) f_i}.
$$
\end{Corollary}
Using exponential fitness makes the computation of the Laplace transform and the limiting degree distribution easier. We refer to Section \ref{sec-expfitness} for the precise proof. In particular, the sequence defined in Corollary \ref{th-degexpfitness} is very similar to the limiting degree distribution of a stationary process with a bounded fitness. Let $(\xi^Y_t)_{t\geq0}$ be a birth process with PA weights $(f_k)_{k\in\N}$ and fitness $Y$ with bounded support.  As proved in \cite[Corollary 2.8]{der2014}, and as we show in Section \ref{sec-fitconditions}, the limiting degree distribution of the corresponding branching process, assuming that $(\xi^Y_t)_{t\geq0}$ is supercritical and Malthusian, has the form
$$
	p_k = \E\left[\frac{\alpha^*}{\alpha^* +Yf_k}\prod_{i=0}^{k-1}\frac{Yf_i}{\alpha^* + Yf_i}\right] = \pr\left(\xi^Y_{T_{\alpha^*}} = k\right).
$$
We notice the similarities with the limiting degree sequence given by Corollary \ref{th-degexpfitness}. When $g$ is integrable, the random variable $G(T_{\alpha^*})$ has bounded support. In particular, we can rewrite the sequence of the Corollary \ref{th-degexpfitness} as
$$
	p_k = \pr\left(\xi^{G(T_{\alpha^*})}_{T_{\theta}}=k\right).
$$
As a consequence, the limiting degree distribution of the process $(M_t)_{t\geq0}$ equals that of a stationary process with fitness $G(T_{\alpha^*})$ and Malthusian parameter $\theta$. 

\begin{figure}[t]
		\centering
		\includegraphics[width = 0.35\textwidth]{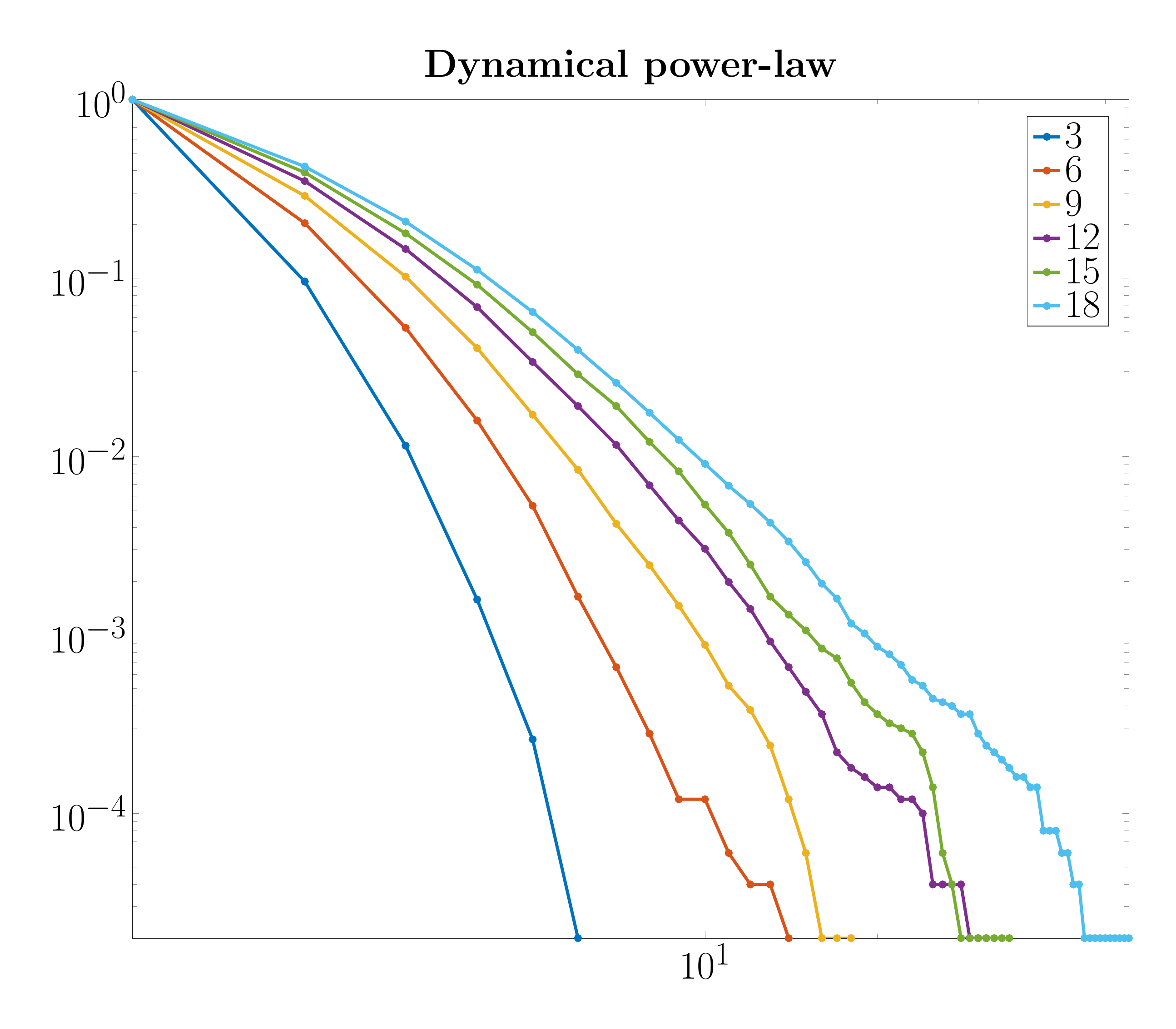}
		\includegraphics[width = 0.35\textwidth]{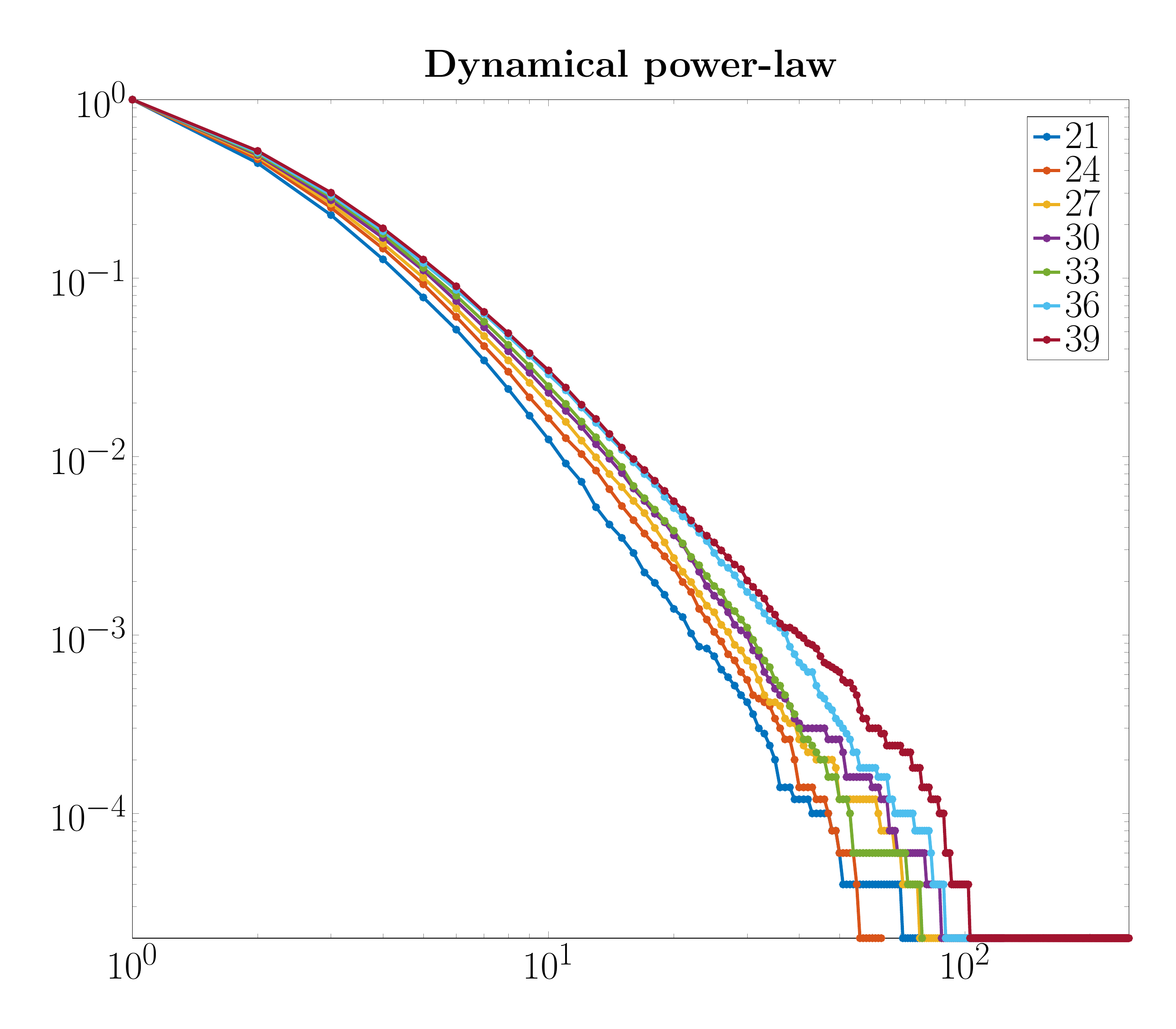}
	\caption{Degree distribution for simulated processes with fitness, aging and affine weights. We considered $a=1$, $b=3.1567$, exponentially distributed fitness with parameter $\theta = 2.3866$, and normalized lognormal aging function. Simulations made at different times show the change of the power-law exponent.}
	\label{fig-pwrlwtime}
\end{figure}

In the case where $Y$ has exponential distribution and the PA weights are affine, we can also investigate the occurrence of {\em dynamical power laws}. In fact, with $(M_t)_{t\geq0}$ such a process, the exponential distribution $Y$ leads to 
\eqn{
\begin{split}
	P_k[M](t) = \pr\left(M_t=k\right) & = \frac{\theta}{\theta+f_kG(t)}\prod_{i=0}^{k-1}\frac{f_iG(t)}{\theta+f_iG(t)}\\
	& = \frac{\theta}{aG(t)}\frac{\Gamma((b+\theta)/(aG(t))}{\Gamma(aG(t))}\frac{\Gamma(k+b/(aG(t)))}{\Gamma(k+b/(aG(t))+ 1+\theta/(aG(t)))}.
	\label{pkt-formula}
\end{split}
}
Here, $M_t$ describes the number of children of an individual of age $t$. 
In other words, $(\pr(M_t=k))_{k\in\N}$ is a distribution such that, as $k\rightarrow\infty$, 
$$
	P_k[M](t) = \pr\left(M_t = k\right) = k^{-(1+\theta/aG(t))}(1+o(1)).
$$
This means that for every time $t\geq0$, the random variable $M_t$ has a power-law distribution with exponent $\tau(t) = 1+\theta/aG(t)>2$. In particular, for every $t\geq0$,  $M_t$ has finite expectation.  We call this behavior where power laws occur that vary with the age of the individuals a {\em dynamical power law}. This occurs not only in the case of pure exponential fitness, but in general for every distribution as in \eqref{def-generalexpfit}, as shown in Proposition \ref{prop-expfit_general} below.

Further, we see that when $t\rightarrow \infty$, the dynamical power-law exponent coincides with the power-law exponent of the entire population.
Indeed, the limiting degree distribution equals
	\eqn{
	\label{for-expfit-pk}
	p_k = \E\left[\theta/(aG(T_{\alpha^*})))\frac{\Gamma(\theta/(aG(T_{\alpha^*}))+b/(aG(T_{\alpha^*}))}{\Gamma(b/(aG(T_{\alpha^*})))}	
	\frac{\Gamma(k+b/(aG(T_{\alpha^*})))}{\Gamma(k+b/(aG(T_{\alpha^*}))+1+\theta/(aG(T_{\alpha^*})))}\right].
	}

In Figure \ref{fig-pwrlwtime}, we show a numerical example of the dynamical power-law for a process with exponential fitness distribution and affine weights. When time increases, the power-law exponent monotonically decreases to the limiting exponent $\tau\equiv \tau(\infty)>2$, which means that the limiting distribution still has finite first moment. Note the similarity to the case of citation networks in Figure \ref{fig-dunamycpowerlaw}.

\begin{figure}[b]
		\centering
		\includegraphics[width = 0.35\textwidth]{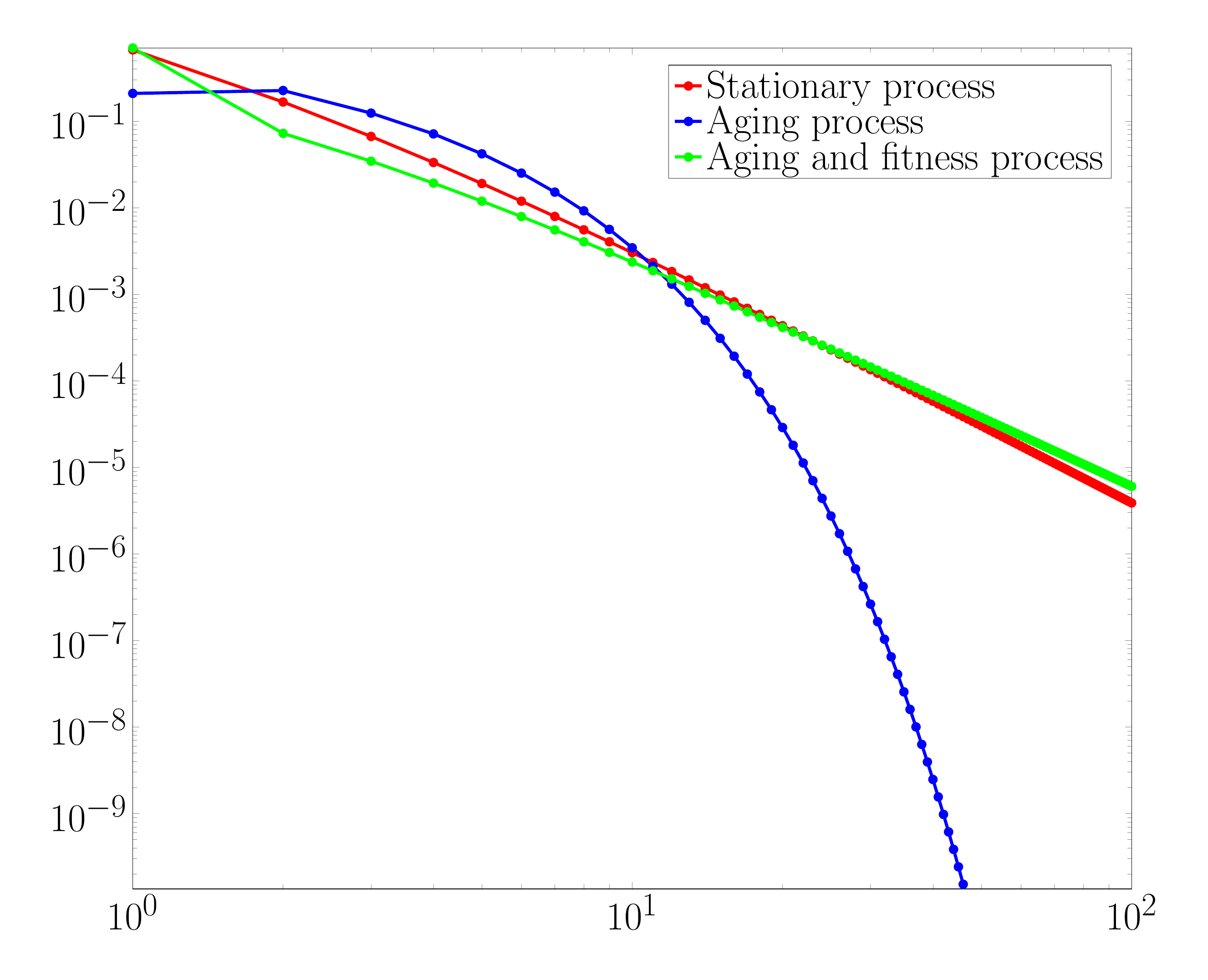}
		\includegraphics[width = 0.35\textwidth]{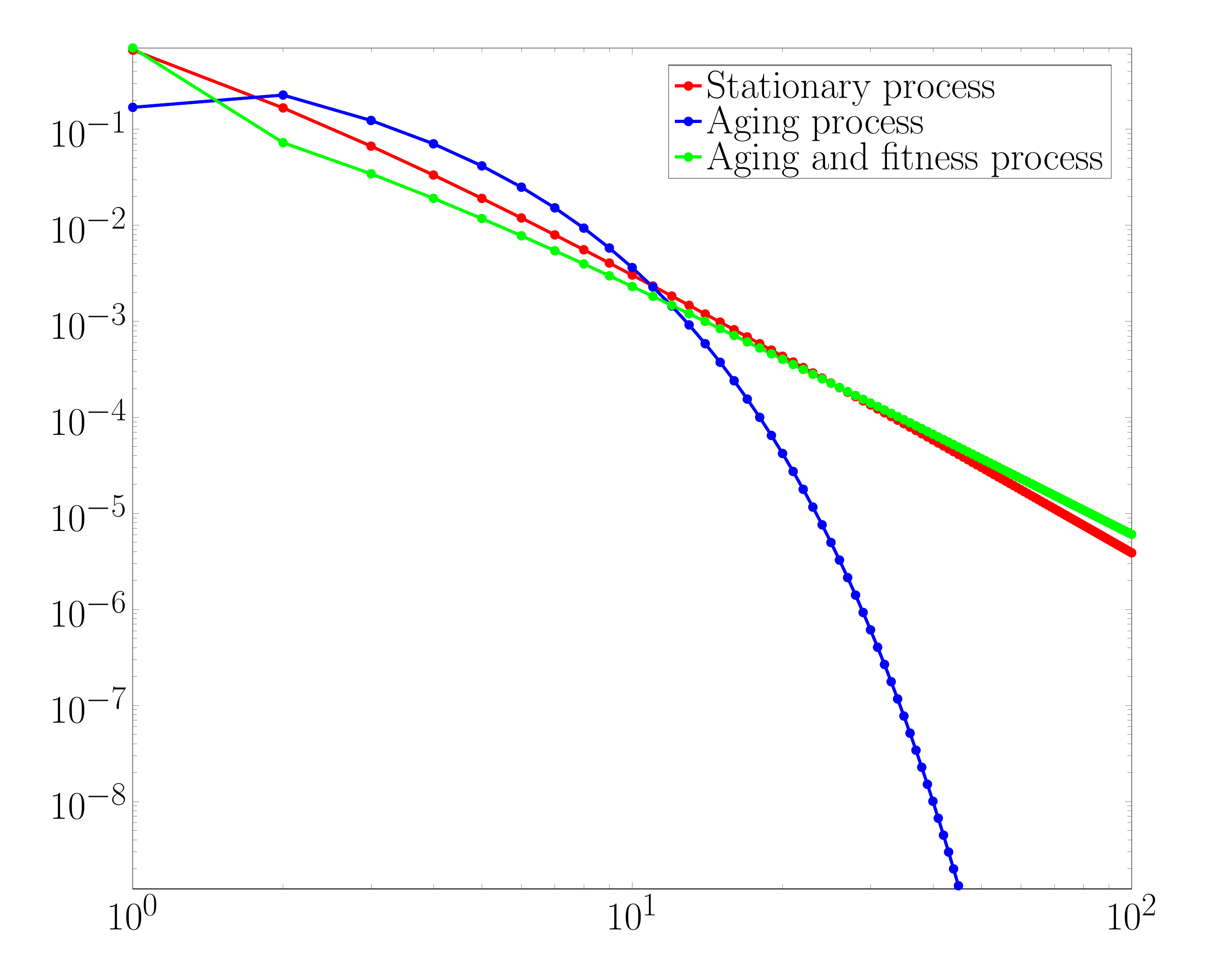}
\caption{Example of limiting degree distribution for branching processes.}
\label{fig-stat-nonstat-fit}
\end{figure}

When $t\rightarrow\infty$, the power-law exponent converges, and also $M_t$ converges in distribution to a limiting random distribution $M_\infty$ given by
	\eqn{
	\label{exp-fitness-entire}
	q_k = \pr\left(M_\infty=k\right) = \frac{\theta}{aG(\infty)}\frac{\Gamma((b+\theta)/(aG(\infty))}{\Gamma(b/(aG(\infty)))}\frac{\Gamma(k+b/(aG(\infty)))}{\Gamma(k+b/(aG(\infty))+ 1+\theta/(aG(\infty)))}.
	}
$M_\infty$ has a power-law distribution, where the power-law exponent is 
$$
	\tau = \lim_{t\rightarrow\infty}\tau(t) = 1+ \theta/(aG(\infty))>2.
$$
In particular, since $\tau>2$, a fixed individual has finite expected number of children also in its entire lifetime, unlike the stationary case with affine weights. In terms of citation networks, this type of processes predicts that papers do not receive an infinite number of citations after they are published (recall Figure \ref{fig-average_degree_increment}).

Figure \ref{fig-distr-agepower} shows the effect of aging on the stationary process with affine weights, where the power-law is lost due to the aging effect. Thus,  aging {\em slows down} the stationary process, and it is not possible to create the amount of high-degree vertices that are present in power-law distributions. Fitness can {\em speed up} the aging process to gain high-degree vertices, so that the power-law distribution is restored. This is shown in Figure \ref{fig-stat-nonstat-fit}, where aging is combined with exponential fitness for the same aging functions as in Figure \ref{fig-distr-agepower}. 

In the stationary case, it is not possible to use unbounded distributions for the fitness to obtain a Malthusian process if the PA weights $(f_k)_{k\in\N}$ are affine. In fact, using unbounded distributions, the expected number of children at exponential time $T_{\alpha}$ is not finite {\em for any} $\alpha>0$, i.e., the branching process is {\em explosive}. The aging effect allows us to relax the condition on the fitness, and the restriction to bounded distributions is relaxed to a condition on its moment generating function.

\subsection{Conclusion and open problems}
\label{sec-struc}

\paragraph{\bf Beyond the tree setting.}
In this paper, we only consider the {\em tree setting}, which is clearly unrealistic for citation networks. However, the analysis of PAMs has  shown that the qualitative features of the degree distribution for PAMs are identical to those in the tree setting. Proving this remains an open problem that we hope to address hereafter. Should this indeed be the case, then we could summarize our findings in the following simple way:
The power-law tail distribution of PAMs is destroyed by integrable aging, and cannot be restored either by super-linear weights or by adding bounded fitnesses. However, it {\em is} restored by {\em unbounded} fitnesses with at most an exponential tail. Part of these results are example based, while we have general results proving that the limiting degree distribution exists.

\paragraph{\bf Structure of the paper.} The present paper is organized as follows. In Section \ref{sec-generalth}, we quote general results on CTBPs, in particular Theorem \ref{th-expogrowth} that we use throughout our proofs. In Section \ref{sec-stat-nonfit}, we describe known properties of the stationary regime. In Section \ref{sec-fitconditions}, we briefly discuss the Malthusian parameter, focusing on conditions on fitness distributions to obtain supercritical processes. In Section \ref{sec-existence}, we prove Theorem \ref{th-explosive} and  \ref{th-degagefit}, and we show how Theorem \ref{th-limitdist-nonstat} is a particular case of Theorem  \ref{th-degagefit}. In Section \ref{sec-laplaceSection} we specialize to the case of affine PA function, giving precise asymptotics.

\section{General theory of Continuous-Time Branching Processes}
\label{sec-generalth}

\subsection{General set-up of the model}
\label{sec-gen-set-up}
In this section we present the general theory of continuous-time branching processes ($\mathrm{CTBPs}$). In such models, individuals produce children according to i.i.d. copies of the same birth process. We now define birth processes in terms of point processes:

\begin{Definition}[Point process]
A {\em point process} $\xi$ is a random variable from a probability space $(\Omega,\mathcal{A},\pr)$ to the space of integer-valued measures on $\R^+$. 
\end{Definition}
A point process $\xi$ is defined by a sequence of positive real-valued random variables $(T_{k})_{k\in\N}$. With abuse of notation, we can denote the density of the point process $\xi$ by 
$$
\xi(dt) = \sum_{k\in\N}\delta_{T_k}(dt),
$$
where $\delta_x(dt)$ is the delta measure in $x$, and the random measure $\xi$ evaluated on $[0,t]$ as
$$
	\xi(t) = \xi([0,t]) = \sum_{k\in\N}\I_{[0,t]}(T_k).
$$
We suppose throughout the paper that $T_k<T_{k+1}$ with probability 1 for every $k\in\N$. 
\begin{Remark}
\label{rem-pointproc}
Equivalently, considering a sequence $(T_k)_{k\in\N}$ (where $T_0=0$) of positive real-valued random variables, such that $T_k< T_{k+1}$ with probability $1$, we can define 
$$
	\xi(t) = \xi([0,t]) =  k \quad\quad\mbox{when}\quad\quad t\in[T_k,T_{k+1}).
$$
We will often define a point process from the jump-times sequence of an integer-valued process $(V_t)_{t\geq 0}$. For instance, consider $(V_t)_{t\geq0}$ as a Poisson process, and denote $T_k =\inf\{t>0\mbox{ : }V_t\geq k\}$. Then we can use the sequence $(T_k)_{k\in\N}$ to define a point process $\xi$. The point process defined from the jump times of a process $(V_t)_{t\geq0}$ will be denoted by $\xi_V$. 
\end{Remark}

We now introduce some notation before giving the definition of $\CTBP$. We denote the set of individuals in the population using Ulam-Harris notation for trees. The set of individuals is 
$$
	\mathcal{N} = \bigcup_{n\in\N}\N^n.
$$
For $x\in\N^n$ and $k\in\N$ we denote the $k$-th child of $x$ by $xk\in\N^{n+1}$. 
This construction is well known, and has been used in other works on branching processes (see \cite{Jagers}, \cite{Nerman}, \cite{RudValko} for more details).

We now are ready to define our branching process:
\begin{Definition}[Continuous-time branching process]
\label{def-brproc}
Given a point process $\xi$, we define the $\CTBP$ associated to $\xi$ as the pair of a probability space 
$$
	(\Omega,\mathcal{A},\pr) = \prod_{x\in\mathcal{N}}\left(\Omega_x,\mathcal{A}_x,\pr_x\right),
$$
and an infinite set $(\xi^x)_{x\in\mathcal{N}}$ of i.i.d. copies of the process $\xi$. We will denote the branching process by $\sub{\xi}$.
\end{Definition}

\begin{Remark}[Point processes and their jump times]
Throughout the paper, we will define point processes in terms of jump times of processes $(V_t)_{t\geq0}$. In order to keep the notation light, we will denote branching processes defined by point processes given by jump times of the process $V_t$ by $\sub{V}$. To make it more clear, by $\sub{V}$ we denote a probability space as in Definition \ref{def-brproc} and an infinite set of measures $(\xi_V^x)_{x\in\N}$, where $\xi_V$ is the point process defined by the process $V$.
\end{Remark}
According to Definition \ref{def-brproc}, a branching process is a pair of a probability space and a sequence of random measures. It is possible though to define an {\itshape evolution} of the branching population. 
At time $t=0$, our population consists only of the root, denoted by $\emp$. Every time $t$ an individual $x$ gives birth to its $k$-th child, i.e., $\xi^x(t)=k+1$, assuming that $\xi^x(t-)=k$, we start the process $\xi^{xk}$. Formally:

\begin{Definition}[Population birth times]
	We define the sequence of birth times for the process $\sub{\xi}$ as $\tau^\xi_\emp=0$, and for $x\in\mathcal{N}$, 
	$$
		\tau^\xi_{xk} = \tau^\xi_x+\inf\left\{s\geq 0\mbox{ : } \xi^x(s)\geq k\right\}.
	$$
\end{Definition}
In this way we have defined the set of individuals, their birth times and the processes according to which they reproduce. We still need a way to count how many individuals are alive at a certain time $t$. 

\begin{Definition}[Random characteristic]
\label{charact}
	A {\em random characteristic} is a real-valued process $\Phi\colon \Omega\times\R\rightarrow\R$ such that $\Phi(\omega,s)=0$ for any $s<0$, and $\Phi(\omega,s) = \Phi(s)$ is a deterministic bounded function for every $s\geq 0$ that only depends on $\omega$ through the birth time of the individual, as well as the birth process of its children.
\end{Definition}
An important example of a random characteristic is obtained by the function $\I_{\R^+}(s)$, which measures whether the individual has been born at time $s$. 
Another example is $\I_{\R^+}(s)\I_{\{k\}}(\xi)$, which measures whether the individual has been born or not at time $s$ and whether it has $k$ children presently.
\medskip

For each individual $x\in\mathcal{N}$, $\Phi_x(\omega,s)$
denotes the value of $\Phi$ evaluated on the progeny of $x$, regarding $x$ as ancestor, when the age of $x$ is $s$. In other words,  $\Phi_x(\omega,s)$ is the evaluation of $\Phi$ on the tree rooted at $x$, ignoring the rest of the  population. If we do not specify the individual $x$, then we assume that $\Phi = \Phi_\emp$. We use random characteristics to describe the properties of the branching population.

\begin{Definition}[Evaluated branching processes]
	Consider a random characteristic $\Phi$ as in Definition \ref{charact}. We define the evaluated branching processes with respect to $\Phi$ at time $t\in\R^+$ as
	$$
		\sub{\xi}_t^\Phi = \sum_{x\in\mathcal{N}}\Phi_x(t-\tau^\xi_x).
	$$
\end{Definition}
The meaning of the evaluated branching process is clear when we consider the random characteristic $\Phi(t) = \I_{\R^+}(t)$, for which
$$
	\sub{\xi}_t^{\I_{\R^+}} = \sum_{x\in\mathcal{N}}(\I_{\R^+})_x(t-\tau^\xi_x),
$$
which is the number of $x\in\mathcal{N}$ such that $t-\tau^\xi_x\geq 0$, i.e., the total number of individuals already born up to time $t$. Another characteristic that we consider in this paper is, for $k\in\N$, $\Phi_k(t) = \I_{\{k\}}(\xi_{t})$, for which
$$
	\sub{\xi}_t^{\Phi_k} = \sum_{x\in\mathcal{N}}\I_{\{k\}}\left(\xi^x_{t-\tau^\xi_x}\right)
$$
is the number of individuals with $k$ children at time $t$.

As known from the literature, the properties of the branching process are determined by the behavior of the point process $\xi$. First of all, we need to introduce some notation. Consider a function $f:\R^+\rightarrow\R$. We denote the Laplace transform of $f$ by
$$
	\mathcal{L}(f(\cdot))(\alpha) = \int_0^\infty \e^{-\alpha t}f(t)dt.
$$
With a slight abuse of notation, if $\mu$ is a positive measure on $\R^+$, then we denote 
$$
	\mathcal{L}(\mu(d\cdot))(\alpha) = \int_0^\infty \e^{-\alpha t}\mu(dt).
$$
We use the Laplace transform to analyze the point process $\xi$:
\begin{Definition}[Supercritical property]
\label{def-supercr}
Consider a point process $\xi$ on $\R^+$. We say $\xi$ is {\em supercritical} when there exists $\alpha^*>0$ such that
$$
	\mathcal{L}(\E\xi(d\cdot))(\alpha^*) = \int_0^\infty \e^{-\alpha^* t}\E\xi(dt) =\sum_{k\in\N}\E\left[\int_0^\infty \e^{-\alpha^* t}\delta_{T_k}(dt)\right] = \sum_{k\in\N}\E\left[\e^{-\alpha^* T_k}\right]=1.
$$
We call $\alpha^*$ the {\em Malthusian parameter} of the process $\xi$.
\end{Definition}
We point out that $\E\xi(d\cdot)$ is an abuse of notation to denote the density of the {\itshape averaged} measure $\E[\xi([0,t])]$. A second fundamental property for the analysis of branching processes is the following:
\begin{Definition}[Malthusian property]
\label{def-malthus}
Consider a supercritical point process $\xi$, with Malthusian parameter $\alpha^*$. The process $\xi$ is {\em Malthusian} when 
$$
	\left.-\frac{d}{d\alpha}\left(\mathcal{L}(\E\xi(dt))\right)(\alpha)\right|_{\alpha^*} = \int_0^\infty t\e^{-\alpha^* t}\E\xi(d\cdot)
		= \sum_{k\in\N}\E\left[T_k\e^{-\alpha^* T_k}\right]<\infty.
$$
\end{Definition}
We denote
\eqn{
\label{for-tildealpha}
	\tilde{\alpha} = \inf\left\{\alpha>0 ~:~ \mathcal{L}\left(\E\xi(d\cdot)\right)(\alpha)<\infty\right\},
}
and we will also assume that the process satisfies the condition
\eqn{
\label{for-larger1}
	\lim_{\alpha\searrow\tilde{\alpha}}\mathcal{L}\left(\E\xi(d\cdot)\right)(\alpha)>1.
}
Integrating by parts, it is possible to show that, for a point process $\xi$, 
$$
	\mathcal{L}\left(\E\xi(d\cdot)\right)(\alpha) = \E\left[V_{T_\alpha}\right],
$$
where $T_\alpha$ is an exponentially distributed random variable independent of the process $(V_t)_{t\geq0}$. Heuristically, the Laplace transform of a point process $\xi_V$ is the expected number of children born at exponentially distributed time $T_\alpha$. In this case the Malthusian parameter is the exponential rate $\alpha^*$ such that at time $T_{\alpha^*}$ exactly one children has been born.

These two conditions are required to prove the main result on branching processes that we rely upon:
\begin{Theorem}[Population exponential growth]
\label{th-expogrowth}
Consider the point process $\xi$, and the corresponding branching process $\sub{\xi}$. Assume that $\xi$ is supercritical and Malthusian with parameter $\alpha^*$, and suppose that there exists $\bar{\alpha}<\alpha^*$ such that
$$
	\int_0^\infty \e^{-\bar{\alpha}t}\E\xi(dt)<\infty.
$$
Then
\begin{enumerate}
\item there exists a random variable $\Theta$ such that as $t\rightarrow\infty$,
	\eqn{
	\label{th-expogrowth-f1}
		\e^{-\alpha^*t}\sub{\xi}^{\I_{\R^+}}_t\stackrel{\pr-as}{\longrightarrow}\Theta;
	}
\item for any two random characteristics $\Phi$ and $\Psi$, 
	\eqn{
	\label{th-expogrowth-f2}
		\frac{\sub{\xi}^{\Phi}_t}{\sub{\xi}^{\Psi}_t}\stackrel{\pr-as}{\longrightarrow}\frac{\mathcal{L}(\E[\Phi(\cdot)])(\alpha^*)}{\mathcal{L}(\E[\Psi(\cdot)])(\alpha^*)}.
	}
\end{enumerate}
\end{Theorem}
This result is stated in \cite[Theorem A]{RudValko}, which is a weaker version of \cite[Theorem 6.3]{Nerman}.
Formula \eqref{th-expogrowth-f1} implies that, $\pr$-a.s., the population size grows exponentially with time. It is relevant though to give 
a description of the distribution of the random variable $\Theta$:
\begin{Theorem}[Positivity of $\Theta$]
\label{th-W}
Under the hypothesis of Theorem \ref{th-expogrowth}, if
\eqn{
\label{for-xlogx}
	\E\left[\mathcal{L}(\xi(d\cdot))(\alpha^*)\log^+\left(\mathcal{L}(\xi(d\cdot))(\alpha^*)\right)\right]<\infty,
}
then, on the event $\{\sub{\xi}^{\I_{\R^+}}_t\rightarrow\infty\}$, i.e., on the event that the branching population keeps growing in time, the random variable $\Theta$ in \eqref{th-expogrowth-f1} is positive with probability 1, and $\E[\Theta]=1$. Otherwise, $\Theta=0$ with probability 1. Condition \eqref{for-xlogx} is called the $(\mathrm{xlogx})$ condition.
\end{Theorem}
This result is proven in \cite[Theorem 5.3]{Jagers}, and it is the CTBPs equivalent of the Kesten-Stigum theorem for Galton-Watson processes (\cite[Theorem 1.1]{kes}).  

Formula \eqref{th-expogrowth-f2} says that the ratio between the evaluation of the branching process with two different characteristics converges $\pr$-a.s. to a constant that depends only on the two characteristics involved. In particular, if we consider, for $k\in\N$,
$$
	\begin{array}{ccc}
		\displaystyle \Phi(t) = \I_{\{k\}}(\xi_t),& \mbox{ and }& \displaystyle \Psi(t) = \I_{\R^+}(t),
	\end{array}
$$
then Theorem \ref{th-expogrowth} gives
\eqn{
\label{rem-ratiochar}
	\frac{\sub{\xi}^{\Phi}_t}{\sub{\xi}^{\I_{\R^+}}_t}\stackrel{\pr-as}{\longrightarrow}\alpha^*\mathcal{L}(\pr\left(\xi(\cdot)=k\right))(\alpha^*),
}
since $\mathcal{L}(\E[\I_{\R^+}(\cdot)])(\alpha^*) = 1/\alpha^*$. The ratio in the previous formula is the fraction of individuals with $k$ children in the whole population:
\begin{Definition}[limiting degree distribution for $\CTBP$]
\label{def-limitdistr}
The sequence $(p_k)_{k\in\N}$, where
$$
	p_k = \alpha^*\mathcal{L}(\pr\left(\xi(\cdot)=k\right))(\alpha^*) = \alpha^*\int_0^\infty \e^{-\alpha^* t}\pr\left(\xi(t)=k\right)dt
$$
is the {\em limiting degree distribution} for the branching process $\sub{\xi}$.
\end{Definition}

The aim of the following sections will be to study when point processes satisfy the conditions of Theorem \ref{th-expogrowth}, in order to analyze the limiting degree distribution in Definition \ref{def-limitdistr}. 

\subsection{Stationary birth processes with no fitness}
\label{sec-stat-nonfit}
In this section we present the theory of birth processes that are stationary and have deterministic rates. This is relevant since the definition of aging processes starts with a stationary process. In particular, we give description of the affine case, which plays a central role in the present work:

\begin{Definition}[Stationary non-fitness birth processes]
\label{def-statnonfit}
Consider a non-decreasing sequence $(f_k)_{k\in\N}$ of positive real numbers. A {\em stationary non-fitness birth process} is a stochastic process $(V_t)_{t\geq0}$ such that
\begin{enumerate}
	\item $V_0=0$, and $V_t\in\N$ for all $t\in\R^+$;
	\item $V_t\leq V_s$ for every $t\leq s$;
	\item for $h$ small enough,
	\eqn{
	\label{form-jumpprobstat}
		\pr\left(V_{t+h}=k+1 \mid V_t=k\right) = f_kh + o(h), ~\mbox{and for}~j\geq 2,~ \pr\left(V_{t+h}=k+j \mid V_t=k\right) = o(h^2).
	}
\end{enumerate}
We denote the jump times by $(T_k)_{k\in\N}$, i.e., 
$$
	T_k = \inf\left\{t\geq 0 \mbox{ : } V_t\geq k\right\}.
$$
\end{Definition}
 We denote the point process corresponding to $(V_t)_{t\geq0}$ by $\xi_V$. In this case, $(V_t)_{t\geq0}$ is an inhomogeneous Poisson process, and for every $k\in\N$,
$T_{k+1}-T_k$ has exponential law with parameter $f_k$ independent of $(T_{h+1}-T_h)_{h=0}^{k-1}$. It is possible to show the following proposition:

\begin{Proposition}[Probabilities for $(V_t)_{t\geq0}$]
\label{prop-eqdiffV}
Consider a stationary non-fitness birth process $(V_t)_{t\geq 0}$. Denote, for every $k\in\N$, $\pr(V_t=k)=P_k[V](t)$. Then
\eqn{
\label{for-probfunct1}
	P_0[V](t) = \mathrm{exp}\left(-f_0t\right), 
}
and, for $k\geq 1$,
\eqn{
\label{for-probfunct2}
	P_k[V](t) = f_{k-1}\mathrm{exp}\left(-f_k t\right)\int_0^t\mathrm{exp}\left(f_kx\right)P_{k-1}[V](x)dx.
}
\end{Proposition}
For a proof, see \cite[Chapter 3, Section 2]{athrBook}. From the jump times, it is easy to compute the explicit expression for the Laplace transform of $\xi_V$ as
$$
	\mathcal{L}(\E\xi_V(d\cdot))(\alpha)=\sum_{k\in\N}\E\left[\int_0^\infty \e^{-\alpha t}\delta_{T_k}(dt)\right] = \sum_{k\in\N}\E\left[\e^{-\alpha T_k}\right] = \sum_{k\in\N}\prod_{i=0}^{k-1}\frac{f_i}{\alpha+f_i},
$$
since every $T_k$ can be seen as sum of independent exponential random variables with parameters given by the sequence $(f_k)_{k\in\N}$.
Assuming now that $\xi_V$ is supercritical and Malthusian with parameter $\alpha^*$, we have the explicit expression for the limit distribution $(p_k)_{k\in\N}$, given by \eqref{deg-distr-PA-tree}.

An analysis of the behavior of the limit distribution of branching processes is presented in \cite{Athr} and \cite{Rudas}, where the authors prove that $(p_k)_{k\in\N}$ has a power-law tail only if the sequence of rates $(f_k)_{k\in\N}$ is asymptotically linear with respect to $k$. 
\begin{Proposition}[Characterization of stationary and linear process $V$]
\label{prop-V-linear}
Consider the sequence $f_k = ak+b$. Then:
\begin{enumerate}
	\item for every $\alpha\in\R^+$,
	$$
		\mathcal{L}(\E\xi_V(d\cdot))(\alpha)  =\frac{\Gamma(\alpha^*/a+b/a)}{\Gamma(b/a)}\sum_{k\in\N}\frac{\Gamma(k+b/a)}{\Gamma(k+b/a+\alpha/a)} =
				\frac{b}{\alpha-a}.
	$$
	\item The Malthusian parameter is $\alpha^*=a+b$, and $\tilde{\alpha} = a$, where $\tilde{\alpha}$ is defined as in \eqref{for-tildealpha}.
	\item The derivative of the Laplace transform is 
	$$
		-\frac{b}{(\alpha-a)^2},
	$$
	which is finite whenever $\alpha>a$;
	\item The process $(V_t)_{t\geq0}$ satisfies the $(\mathrm{xlogx})$ condition \eqref{for-xlogx}.
	\end{enumerate}
\end{Proposition}
\begin{proof}
The proof can be found in \cite[Theorem 2]{RudValko}, or \cite[Theorem 2.6]{Athr}.
\end{proof}
For affine PA weights $(f_k)_{k\in\N} = (ak+b)_{k\in\N}$, the Malthusian parameter $\alpha^*$ exists. Since $\alpha^* = a+b$, the limiting degree distribution of the branching process $\sub{V}$ is given by
\eqn{
\label{for-degnormale}
	p_k = (1+b/a)\frac{\Gamma(1+2b/a)}{\Gamma(b/a)}\frac{\Gamma\left(k+b/a\right)}{\Gamma\left(k+b/a+2+b/a\right)}.
}
Notice that $p_k$ has a power-law decay with exponent $\tau = 2+\frac ba$.
Branching processes of this type are related to PAM, also called the Barab\'asi-Albert model (\cite{ABrB}). This model shows the so-called {\itshape old-get-richer} effect. Clearly this is not true for real-world citation networks. In Figure \ref{fig-average_degree_increment}, we notice that, on average, the increment of the citation received by old papers is smaller than the increment of younger papers. Rephrasing it, old papers tend to be cited less and less over time. 

\subsection{The Malthusian parameter}
\label{sec-fitconditions}

The existence of the Malthusian parameter is a necessary condition to have a branching process growing at exponential rate. In particular, the Malthusian parameter does not exist in two cases: when the process is subcritical and grows slower than exponential, or when it is explosive. In the first case, the branching population might either die out or grow indefinitely with positive probability, but slower than at exponential rate. In the second case, the population size explodes in finite time with probability one. In both cases, the behavior of the branching population is different from what we observe in citation networks (Figure \ref{fig-numberpublic}). For this reason, we focus on supercritical processes, i.e., on the case where the Malthusian parameter exists. 

Denote by $(V_t)_{t\geq0}$ a stationary birth process defined by PA weights $(f_k)_{k\in\N}$. In general, we assume $f_k\rightarrow\infty$. Denote the sequence of jump times by $(T_k)_{k\in\N}$. As we quote in Section \ref{sec-stat-nonfit}, the Laplace transform of a birth process $(V_t)_{t\geq0}$ is given by
$$
	\mathcal{L}(\E V(d\cdot))(\alpha) = \E\left[\sum_{k\in\N}\e^{-\alpha T_k}\right] = \E\left[V_{T_{\alpha}}\right] = \sum_{k\in\N}\prod_{i=0}^{k-1}\frac{f_i}{\alpha+f_i}.
$$
Such expression comes from the fact that, in stationary regime, $T_k$ is the sum of $k$ independent exponential random variables. We can write
$$
	\sum_{k\in\N}\mathrm{exp}\left(-\sum_{i=0}^{k-1}\log\left(1+\frac{\alpha}{f_i}\right)\right) = 
		\sum_{k\in\N}\mathrm{exp}\left(-\alpha\sum_{i=0}^{k-1}\frac{1}{f_i}(1+o(1))\right).
$$

The behavior of the Laplace transform depends on the asymptotic behavior of the PA weights. We define now the terminology we use:
\begin{Definition}[Superlinear PA weights]
\label{def-superlin}
Consider a PA weight sequence $(f_k)_{k\in\N}$. We say that the PA weights are {\em superlinear} if $\sum_{i=0}^{\infty}1/f_i<\infty$.
\end{Definition}
As a general example, consider $f_k = ak^q+b$, where $q>0$. In this case, the sequence is affine when $q=1$, superlinear when $q>1$ and sublinear when $q<1$. 

When the weights are superlinear, since $C = \sum_{i=0}^{\infty}1/f_i<\infty$, we have
\eqn{
\label{for-superliBound}
	\sum_{k\in\N}\mathrm{exp}\left(-\alpha\sum_{i=0}^{k-1}\frac{1}{f_i}(1+o(1))\right)\geq 
		\sum_{k\in\N}\mathrm{exp}\left(-\alpha C\right) = +\infty.
}	
This holds for every $\alpha>0$. As a consequence, the Laplace transform $\La(\E V(d\cdot))(\alpha)$ is always infinite, and there exist no Malthusian parameter. In particular, if we denote by $T_\infty = \lim_{k\rightarrow\infty}T_k$, then $T_\infty<\infty$ a.s.. This means that the birth process $(V_t)_{t\geq0}$ explodes in a finite time. 

When the weights are at most linear, the bound  in \eqref{for-superliBound} does not hold anymore. In fact, consider as example affine weights $f_k = ak+b$. We have that $\sum_{i=0}^{k-1}\frac{1}{f_i} = (1/a)\log k(1+o(1))$. As a consequence, the Laplace transform can be written as
\eqn{
\label{for-linBound}
 \sum_{k\in\N}\mathrm{exp}\left(-\frac{\alpha}{a}\log k(1+o(1))\right) = \sum_{k\in\N} k^{-\frac{\alpha}{a}}(1+o(1)).
}
In this case, the Laplace transform is finite for $\alpha>a$. For the sublinear case, for which  
$\sum_{i=0}^{k-1}1/f_i = Ck^{(1-q)}(1+o(1))$, we obtain
$$
	\sum_{k\in\N}\mathrm{exp}\left(-C\alpha k^{1-q}\right).
$$
This sum is finite for any $\alpha>0$. 

We can now introduce fitness in the stationary process: 
\begin{Remark}
\label{rem-laplacemonotone}
Consider the process $(V_t)_{t\geq0}$ defined by the sequence of PA weights $(f_k)_{k\in\N}$ as in Section \ref{sec-stat-nonfit}. For $u\in\R^+$ we denote by $(V^u_t)_{t\geq0}$ the process defined by the sequence $(uf_k)_{k\in\N}$. It is easy to show that 
$$
	\mathcal{L}(\E\xi_{V^u}(d\cdot))(\alpha) = \mathcal{L}(\E\xi_V(d\cdot))(\alpha/u).
$$
The behavior of the degree sequence of  $(V^u_t)_{t\geq0}$ is the same of the process $V_t$.
\end{Remark}
Remark \ref{rem-laplacemonotone} shows a sort of monotonicity of the Laplace transform with respect to the sequence $(f_k)_{k\in\N}$. This is very useful to describe the Laplace transform of a birth process with fitness, which we define now:

\begin{Definition}[Stationary fitness birth processes]
\label{def-birthfitness}
Consider a birth process $(V_t)_{t\geq0}$ defined by a sequence of weights $(f_k)_{k\in\N}$.  Let $Y$ be a positive random variable. We call stationary fitness birth processes the process $(V^Y_t)_{t\geq0}$, defined by the random sequence of weights $(Y f_k)_{k\in\N}$, i.e., conditionally on $Y$,
$$
	\pr\left(V^Y_{t+h} = k+1 \mid V^Y_t = k, Y\right) = Y f_k h+ o(h).
$$
\end{Definition}
By Definition \ref{def-birthfitness}, it is obvious that the properties of the process $(V^Y_t)_{t\geq0}$ are related to the properties of $(V_t)_{t\geq0}$. 
Since we consider a random fitness $Y$ independent of the process $(V_t)_{t\geq0}$, from Remark \ref{rem-laplacemonotone} it follows that
\eqn{
\label{for-condFIT-1}
	\La(\E V^Y(d\cdot))(\alpha) = \E\left[\mathcal{L}(\E\xi_{V^u}(d\cdot))(\alpha)_{u=Y}\right] = \E\left[\sum_{k\in\N}\prod_{i=0}^{k-1}\frac{Yf_i}{\alpha+Yf_i}\right].
}
For affine weights the fitness distribution needs to be bounded, as discussed in Section \ref{sec-res-aging-fitness-exp}. In this section we give a qualitative explanation of this fact. 
Consider the sum in the expectation in the right hand term of \eqref{for-condFIT-1}. We can rewrite the sum as
\eqn{
\label{for-condFIT-2}
	\sum_{k\in\N}\prod_{i=0}^{k-1}\frac{Yf_i}{\alpha+Yf_i} = \sum_{k\in\N}\mathrm{exp}\left(-\sum_{i=0}^{k-1}\log\left(1+\frac{\alpha}{Yf_i}\right)\right)=\sum_{k\in\N}\mathrm{exp}\left(-\frac{\alpha}{Y}\sum_{i=0}^{k-1}\frac{1}{f_i}(1+o(1))\right).
} 
The behavior depends sensitively on the asymptotic behavior of the PA weights. In particular, a necessary condition for the existence of the Malthusian parameter is that the sum in \eqref{for-condFIT-1} is finite on an interval of the type $(\tilde{\alpha},+\infty)$. In other words, since the Laplace transform is a decreasing function (when finite), we need to prove the existence of a minimum value $\tilde{\alpha}$ such that it is finite for every $\alpha>\tilde{\alpha}$. Using \eqref{for-condFIT-2} in \eqref{for-condFIT-1}, we just need to find a value $\alpha$ such that the right hand side of \eqref{for-condFIT-2} equals 1.

In the case of affine weights $f_k = ak+b$, we have $\sum_{i=0}^{k-1}\frac{1}{f_i} = C\log k(1+o(1))$, for a constant $C$. As a consequence, \eqref{for-condFIT-2} is equal to 
\eqn{
\label{for-condFIT-4}
	\E\left[\sum_{k\in\N}\mathrm{exp}\left(-C\frac{\alpha}{Y}\log k\right)\right] = \E\left[\sum_{k\in\N} k^{-C\alpha/Y}\right].
}
The sum inside the last expectation is finite only on the event $\{Y<C\alpha\}$. If $Y$ has an unbounded distribution, then for every value of $\alpha>0$ we have that $\{Y\geq C\alpha\}$ is an event of positive probability. As a consequence, for every $\alpha>0$, the Laplace transform of the birth process $(V^Y_y)_{t\geq0}$ is infinite, which means there exists no Malthusian parameter.

This is why a bounded fitness distribution is necessary to have a Malthusian parameter using affine PA weights. The situation is different in the case of sublinear weights. For example, consider $f_k = (1+k)^q$, where $q\in(0,1)$. Then, the difference to affine weights is that now 
$\sum_{i=0}^{k-1}1/f_i = Ck^{1-q}(1+o(1))$. Using this in \eqref{for-condFIT-2}, we obtain
$$
	\E\left[\sum_{k\in\N}\mathrm{exp}\left(-C\frac{\alpha}{Y}k^{(1-q)}\right)\right].
$$
In this case, since both $\alpha$ and $Y$ are always positive, the last sum is finite with probability $1$, and the expectation might be finite under appropriate moment assumptions on $Y$.

Assume now that the fitness $Y$ satisfies the necessary conditions, so that the process $(V_t^Y)_{t\geq0}$ is supercritical and Malthusian with parameter $\alpha^*$. 
We can evaluate the limiting degree distribution. Conditioning on $Y$, the Laplace transform of $\E\xi_{V^Y}(dx)$ is
$$
	\sum_{k\in\N}\prod_{i=0}^{k-1}\frac{Y f_i}{\alpha+Y f_i}, 
$$
so, as a consequence, the limiting degree distribution of the branching processes is
\eqn{
\label{for-pkfitstat}
	p_k = \E\left[\frac{\alpha^*}{\alpha^*+Yf_k}\prod_{i=0}^{k-1}\frac{Y f_i}{\alpha^*+Y f_i}\right].
}
It is possible to see that the right-hand side of \eqref{for-pkfitstat} is similar to the distribution of the simpler case with no fitness given by \eqref{deg-distr-PA-tree}. We still  have a product structure for the limit distribution, but in the fitness case it has to be averaged over the fitness distribution. This result is similar to \cite[Theorem 2.7, Corollary 2.8]{der2014}.

Considering affine weights $f_k = ak+b$, we can rewrite \eqref{for-pkfitstat} as
$$
	p_k = \E\left[\frac{\Gamma((\alpha^*+b)/(aY))}{\Gamma(b/(aY))}
					\frac{\Gamma(k+b/(aY))}{\Gamma(k+b/(aY)+1 +\alpha^*/(aY))}\right].
$$
Asymptotically in $k$, the argument of the expectation in the previous expression is random with a power-law exponent $\tau(Y) = 1+\alpha^*/(aY)$. For example, in this case averaging over the fitness distribution, it is possible to obtain power-laws with logarithmic corrections (see eg \cite[Corollary 32]{Bhamidi}).

\section{Existence of limiting distributions}
\label{sec-existence}
In this section, we give the proof of Theorems \ref{th-limitdist-nonstat}, \ref{th-explosive} and \ref{th-degagefit}, proving that the branching processes defined in Section \ref{sec-mainres} do have a limiting degree distribution. As mentioned, we start by proving Theorem \ref{th-degagefit}, and then explain how Theorem \ref{th-limitdist-nonstat} follows as special case.
 
Before proving the result, we do need some remarks on the processes we consider. Birth process with aging alone and aging with fitness are defined respectively in Definition \ref{def-nonstatbirth} and \ref{def-nonstatfit}. Consider then a process with aging and fitness $(M_t)_{t\geq0}$ as in  Definition \ref{def-nonstatfit}. Let $(T_k)_{k\in\N}$ denote the sequence of birth times, i.e., 
$$
	T_k = \inf\left\{t\geq 0 \colon M_t\geq k\right\}.
$$
It is an immediate consequence of the definition that, for every $k\in\N$,
\eqn{
\label{for-scaledtime}
	\pr\left(T_k\leq t\right) = \pr\left(\bar{T}_k\leq YG(t)\right),
}
where $(\bar{T}_k)_{k\in\N}$ is the sequence of birth times of a stationary birth process $(V_t)_{t\geq0}$ defined bu the same PA function $f$. 
Consider then the sequence of functions $(P_k[V](t))_{k\in\N}$  associated with the stationary process $(V_t)_{t\geq0}$ defined by the same sequence of weights $(f_k)_{k\in\N}$ (see Proposition \ref{prop-eqdiffV}). As a consequence,  for every $k\in\N$, $\pr(M_t=k)=\E[P_k[V](YG(t))]$, and the same holds for an aging process just considering $Y\equiv 1$.
Formula \eqref{for-scaledtime} implies that the aging process is the stationary process with a deterministic time-change given by $G(t)$. A process with aging and fitness  is the stationary process  with a random time-change given by $YG(t)$.

Assume now that $g$ is integrable, i.e. $\lim_{t\rightarrow\infty}G(t)=G(\infty)<\infty$. Using \eqref{for-scaledtime}  we can describe the limiting degree distribution  $(q_k)_{k\in\N}$ of a fixed individual in the branching population, i.e., the distribution $N_\infty$ (or $M_\infty$) of the total number of children an individual will generate in its entire lifetime. In fact, for every $k\in\N$, 
\eqn{
\label{for-limitG}
	\lim_{t\rightarrow\infty}\pr\left(N_t = k\right) =\lim_{t\rightarrow\infty}P[V](G(t)) = \pr\left(V_{G(\infty)}=k\right),
}
which means that $N_\infty$ has the same distribution as $V_{G(\infty)}$. With fitness,
$$
	\lim_{t\rightarrow\infty}\pr\left(M_t = k\right) =\lim_{t\rightarrow\infty}\E[P[V](YG(t))] = \pr\left(V_{YG(\infty)}=k\right).
$$
For example, in the case of aging only, this is rather different from the stationary case, where the number of children of a fixed individual diverges as the individual gets old (see e.g \cite[Theorem 2.6]{Athr}).

\subsection{Proof of Theorem \ref{th-degagefit}}
\label{sec-pf-aging-fitness-gen}
Birth processes with continuous aging effect and fitness are defined in Definition \ref{def-nonstatfit}. We now identify conditions on the fitness distribution to have a Malthusian parameter:

\begin{Lemma}[Condition  \eqref{cond-eg_inf2}]
\label{lem-lapAGEFIT}
Consider a stationary process $(V_t)_{t\geq0}$, an integrable aging function $g$ and a random fitness $Y$. Assume that $\E[V_t]<\infty$ for every $t\geq0$. Then the process $(V_{YG(t)})_{t\geq0}$ is supercritical if and only if  Condition \eqref{cond-eg_inf2} holds, i.e.,
$$
	\E\left[V_{YG(t)}\right]<\infty \quad\mbox{for every }t\geq0 \quad \quad \mbox{and}
	\quad \lim_{t\rightarrow\infty}\E\left[V_{YG(t)}\right]>1.
$$
\end{Lemma}
\begin{proof}
For the if part, we need to prove that
$$
	\lim_{\alpha\rightarrow0^+}\E\left[V_{YG(T_{\alpha^*})}\right]>1 \quad \quad \mbox{and}\quad\quad \lim_{\alpha\rightarrow\infty}\E\left[V_{YG(T_{\alpha^*})}\right]=0.
$$
As before, $(\bar{T}_k)_{k\in\N}$ are the jump times of the process $(V_{G(t)})_{t\geq0}$. Then
$$
	\E\left[V_{YG(T_{\alpha^*})}\right] = \sum_{k\in\N}\E\left[\e^{-\alpha \bar{T}_k/Y}\right].
$$
When $\alpha\rightarrow0$, we have $\E\left[\e^{-\alpha \bar{T}_k}\right] \rightarrow \pr\left(\bar{T}_k/Y<\infty\right)$. Now,
$$
	\sum_{k\in\N}\pr\left(\bar{T}_k<\infty\right) = \lim_{t\rightarrow\infty}\sum_{k\in\N}\pr\left(\bar{T}_k/Y\leq t\right) = \lim_{t\rightarrow\infty}\E\left[V_{YG(t)}\right]>1.
$$
For $\alpha\rightarrow\infty$,
$$
	\int_0^\infty \alpha \e^{-\alpha t}\E\left[V_{YG(t)}\right]dt = \int_0^\infty \e^{-u}\E\left[V_{YG(u/\alpha)}\right]du.
$$
When $\alpha\rightarrow\infty$ we have $\E\left[V_{YG(u/\alpha)}\right]\rightarrow0$. Then, fix $\alpha_0>0$ such that $\E\left[V_{YG(u/\alpha)}\right]<1$ for every $\alpha>\alpha_0$. As a consequence, 
$\e^{-u}\E\left[V_{YG(u/\alpha)}\right]du\leq \e^{-u}$ for any $\alpha>\alpha_0$. By dominated convergence, 
$$
	\lim_{\alpha\rightarrow\infty}\int_0^\infty \alpha \e^{-\alpha t}\E\left[V_{YG(t)}\right]dt=0.
$$
Now suppose Condition \eqref{cond-eg_inf2} does not hold. This means that $\E[V_{YG(t_0)}]= +\infty$ for some $t_0\in[0,G(\infty))$ or $\lim_{t\rightarrow\infty}\E[V_{YG(t_0)}]\leq 1$. 

If the first condition holds, then there exists $t_0\in(0,aG(\infty))$ such that $\E\left[V_{YG(t)}\right]=+\infty$ for every $t\geq t_0$ (recall that $\E\left[V_{YG(t)}\right]$ in an increasing function in $t$).  As a consequence, for every $\alpha>0$, we have $\E\left[V_{YG(T_{\alpha})}\right]=+\infty$, which means that the process is explosive.

If the second condition holds, then for every $\alpha>0$ the Laplace transform of the process is strictly less than $1$, which means there exists no Malthusian parameter. 
\end{proof}

Lemma \ref{lem-lapAGEFIT} gives a weaker condition on the distribution $Y$ than requiring it to be bounded. Now, we want to investigate the degree distribution of the branching process, assuming that the process $(M_t)_{t\geq0}$ is supercritical and Malthusian. Denote the Malthusian parameter by $\alpha^*$. The above allows us to complete the proof of Theorem \ref{th-degagefit}:

\begin{proof}[Proof of Theorem \ref{th-degagefit}] We start from
\eqn{
\label{for-start-pkfit}
	p_k = \E\left[P_k[V](YG(T_{\alpha^*}))\right].
}
Conditioning on $Y$ and integrating by parts in the integral given by the expectation in \eqref{for-start-pkfit}, gives
$$
	-f_kY\int_0^\infty\e^{-\alpha^*t}P_k[V](YG(t))g(t)dt + f_{k-1}Y\int_0^\infty\e^{-\alpha^*t}P_{k-1}[V](YG(t))g(t)dt.
$$
Now, we define
\eqn{
\label{for-hatL-fit}
\hat{\mathcal{L}}(k,\alpha^*,Y) = \left(\frac{\mathcal{L}(\pr\left(V_{uG(\cdot)}=k\right)g(\cdot))(\alpha^*)}{\mathcal{L}(\pr\left(V_{uG(\cdot)}=k\right))(\alpha^*)}\right)_{u=Y}
}
Notice that the sequence $(\hat{\mathcal{L}}(k,\alpha^*,Y))_{k\in\N}$ is a sequence of random variables. Multiplying both sides of the equation by $\alpha^*$, on the right hand side we have
$$
	-f_kY\hat{\mathcal{L}}(k,\alpha^*,Y)\E\left[P_k[V](uG(T_{\alpha^*}))\right]_{u=Y}+
	f_{k-1}Y\hat{\mathcal{L}}(k-1,\alpha^*,Y)\E\left[P_{k-1}[V](uG(T_{\alpha^*}))\right]_{u=Y},
$$
while on the left hand side we have
$$
	\alpha^* \E\left[P_k[V](uG(T_{\alpha^*}))\right]_{u=Y}.
$$
As a consequence, 
\eqn{
\label{for-fitrecur}
	\E\left[P_k[V](uG(T_{\alpha^*}))\right]_{u=Y} = \frac{f_{k-1}Y\hat{\mathcal{L}}(k-1,\alpha^*,Y)}{\alpha^*+f_kY\hat{\mathcal{L}}(k,\alpha^*,Y)}\E\left[P_{k-1}[V](uG(T_{\alpha^*}))\right]_{u=Y}.
}
We start from $p_0$, that is given by
$$
	\E\left[P_0[V](uG(T_{\alpha^*}))\right]_{u=Y} = \frac{\alpha^*}{\alpha^*+f_{0}Y\hat{\mathcal{L}}(0,\alpha^*,Y)}.
$$
Recursively using \eqref{for-fitrecur}, gives
$$
	\E\left[P_k[V](uG(T_{\alpha^*}))\right]_{u=Y} = \frac{\alpha^*}{\alpha^*+f_kY\hat{\mathcal{L}}(k,\alpha^*,Y)}
			\prod_{i=0}^{k-1}
			\frac{f_{i}Y\hat{\mathcal{L}}(i,\alpha^*,Y)}{\alpha^*+f_iY\hat{\mathcal{L}}(i,\alpha^*,Y)}.
$$
Taking expectation on both sides gives
$$
	p_k = \E\left[\frac{\alpha^*}{\alpha^*+f_kY\hat{\mathcal{L}}(k,\alpha^*,Y)}\prod_{i=0}^{k-1}
			\frac{f_{i}Y\hat{\mathcal{L}}(i,\alpha^*,Y)}{\alpha^*+f_iY\hat{\mathcal{L}}(i,\alpha^*,Y)}\right].
$$
\end{proof}
Now the sequence $(\hat{\mathcal{L}}(k,\alpha^*,Y))_{k\in\N}$ creates a relation among the sequence of weights, the aging function and the fitness distribution, so that these three ingredients are deeply related.

\subsection{Proof of Theorems \ref{th-limitdist-nonstat} and \ref{th-explosive}}
\label{sec-aging-gen-PA}
As mentioned, Theorem \ref{th-limitdist-nonstat} follows immediately by considering $Y\equiv 1$. The proof in fact is the same, since we can express the probabilities $\pr(N_t=k)$ as function of the stationary process $(V_t)_{t\geq0}$ defined by the same PA function $f$. 

Condition \eqref{cond-eg_inf1} immediately follows from Condition \eqref{cond-eg_inf2}. In fact, considering $Y\equiv 1$, Condition  \eqref{cond-eg_inf2} becomes
\eqn{
\label{cond-eg_inf2-NEW}
 \E\left[V_{G(t)}\right]<\infty \quad\mbox{for every }t\geq0 \quad \quad \mbox{and}
	\quad \lim_{t\rightarrow\infty}\E\left[V_{G(t)}\right]>1.
}
The first inequality in general true for the type of stationary process we consider (for instance with $f$ affine). The second inequality is exaclty Condition \eqref{cond-eg_inf1}.

The expression of the sequence $(\hat{\La}^g(k,\alpha^*))_{k\in\N}$ is simplier than the general case given in \eqref{for-hatL-fit}. In fact, in \eqref{for-hatL-fit}, the sequence $(\hat{\mathcal{L}}(k,\alpha^*,Y))_{k\in\N}$ is actually a squence of random variables. In the case of aging alone, 
$$
	\hat{\La}^g(k,\alpha^*) = \frac{\mathcal{L}(\pr\left(V_{G(\cdot)}=k\right)g(\cdot))(\alpha^*)}{\mathcal{L}(\pr\left(V_{G(\cdot)}=k\right))(\alpha^*)},
$$
which is a deterministic sequence. 
\begin{Remark}
Notice that $\hat{\La}^g(k,\alpha^*)=1$ when $g(t)\equiv 1$, so that $G(t)=t$ for every $t\in\R^+$ and there is no aging, and we retrieve the stationary process $(V_t)_{t\geq0}$. 
\end{Remark}
Unfortunately, the explicit expression of the coefficients $(\hat{\mathcal{L}}^g(k,\alpha^*))_{k\in\N}$ is not easy to find, even though they are deterministic. 

Theorem  \ref{th-explosive}, which states that even if $g$ is integrable, the aging does not affect the explosive behavior of a birth process with superlinear weights, is a direct consequence of \eqref{for-limitG}:

\begin{proof}[Proof of Theorem  \ref{th-explosive}]
Consider a birth process $(V_t)_{t\geq0}$, defined by a sequence of superlinear weights $(f_k)_{k\in\N}$ (in the sense of Definition \ref{def-superlin}), and an integrable aging function $g$. Then, for every $t>0$, 
$$
	\pr\left(N_t=\infty\right) = \pr\left(V_{G(t)}=\infty\right)>0. 
$$
Since this holds for every $t>0$, the process $(N_t)_{t\geq0}$ is explosive. As a consequence, for any $\alpha>0$, $\E\left[N_{T_\alpha}\right]=\infty$, which means that there exists no Malthusian parameter.
\end{proof}

\section{Affine weights and adapted Laplace method}
\label{sec-laplaceSection}

\subsection{Aging and no fitness}
\label{sec-adaptedLap}
In this section, we consider affine PA weights, i.e., we consider $f_k = ak+b$. The main aim is to identify the asymptotic behavior of the limiting degree distribution of the branching process with aging. Consider a stationary process $(V_t)_{t\geq0}$, where $f_k = ak+b$. Then, for any $t\geq0$, it is possible to show by induction and the recursions in \eqref{for-probfunct1} and \eqref{for-probfunct2}  that
\eqn{
\label{for-Pktstat}
	P_k[V](t) = \pr\left(V_t = k\right) = \frac{1}{\Gamma(b/a)}\frac{\Gamma(k+b/a)}{\Gamma(k+1)}\e^{-bt}\left(1-\e^{-at}\right)^k.
}
 We omit the proof of  \eqref{for-Pktstat}. 
As a consequence, since the corresponding aging process is $(V_{G(t)})_{t\geq0}$, the limiting degree distribution is given by
\eqn{
\label{for-pkAdap}
	p_k = \pr\left(V_t = k\right) = \frac{\Gamma(k+b/a)}{\Gamma(b/a)\Gamma(k+1)}\int_0^\infty \alpha^*\e^{-\alpha^* t}\e^{-bG(t)}\left(1-\e^{-aG(t)}\right)^kdt.
}
We can obtain an immediate upper bound for $p_k$, in fact
$$
	p_k = \frac{\Gamma(k+b/a)}{\Gamma(b/a)\Gamma(k+1)}\int_0^\infty \alpha^*\e^{-\alpha^* t}\e^{-bG(t)}\left(1-\e^{-aG(t)}\right)^kdt\leq \frac{\Gamma(k+b/a)}{\Gamma(b/a)\Gamma(k+1)}(1-\e^{-aG(\infty)})^k,
$$
which implies that the distribution $(p_k)_{k\in\N}$ has at most an exponential tail. A more precise analysis is hard. Instead we will give an asymptotic approximation, by adapting  the Laplace method for integrals to our case. 

The Laplace method states that, for a function $f$ that is twice differentiable and with a unique absolute minimum $x_0\in(a,b)$, as $k\rightarrow\infty$,
\eqn{
\label{for-laplaceTheor}
	\int_a^b \e^{-k\Psi(x)}dx=\sqrt{\frac{2\pi}{k\Psi''(x_0)}}\e^{-k\Psi(x_0)}(1+o(1)).
}
In this situation, the interval $[a,b]$ can be infinite. The idea behind this result is that, when $k\gg 1$, the major contribution to the integral comes from a neighborhood of $x_0$ where $\e^{-k\Psi(x)}$ is maximized. In the integral in \eqref{for-pkAdap}, we do not have this situation, since we do not have an integral of the type \eqref{for-laplaceTheor}. Defining
\eqn{
\label{for-Psidefinition}
	\Psi_k(t) := \frac{\alpha^*}{k}t+\frac{b}{k}G(t)-\log\left(1-\e^{-aG(t)}\right),
}
we can rewrite the integral in \eqref{for-pkAdap} as
\eqn{
\label{forIkdef}
	I(k):=\int_0^\infty \alpha^* \e^{-k\Psi_k(t)}dt.
}
The derivative of the function $\Psi_k(t)$ is 
\eqn{
\label{for-PSider}
	\Psi_k'(t) = \frac{\alpha^*}{k}+\frac{b}{k}g(t)-\frac{ag(t)\e^{-aG(t)}}{1-\e^{-aG(t)}}.
}
In particular, if there exists a minimum $t_k$, then it depends on $k$. In this framework, we cannot directly apply the Laplace method. We now show that we can apply a result similar to \eqref{for-laplaceTheor} even to our case: 

\begin{Lemma}[Adapted Laplace method 1]
\label{Lem-adaptLap-age}
Consider $\alpha,a,b>0$. Let the integrable aging function $g$ be such that
\begin{enumerate}
	\item for every $t\geq0$, $0<g(t)\leq A<\infty$;
	\item $g$ is differentiable on $\R^+$, and $g'$ is finite almost everywhere;
	\item there exists a positive constant $B<\infty$ such that $g(t)$ is decreasing for $t\geq B$;
	\item assume that the solution $t_k$ of $\Psi_k'(t)=0$, for $\Psi_k'(t)$ as in \eqref{for-PSider}, is unique, and $g'(t_k)<0$.
\end{enumerate}
Then, for $\sigma_k^2 = (k\Psi_k''(t_k))^{-1}$, there exists a constant $C$ such that, as $k\rightarrow\infty$, 
$$
	I(k) = C\sqrt{2\pi\sigma_k^2}\e^{-k\Psi_k(t_k)}\left(\frac{1}{2}+\pr\left(\mathcal{N}(0,\sigma_k^2)\geq t_k\right)\right)(1+o(1)),
$$
where $\mathcal{N}(0,\sigma_k^2)$ denotes a normal distribution with zero mean and variance $\sigma_k^2$.
\end{Lemma}
Since Lemma \ref{Lem-adaptLap-age} is an adapted version of the classical Laplace method, we move the proof to Appendix \ref{sec-appendix}. 
We can use the result of Lemma \ref{Lem-adaptLap-age} to prove: 
\begin{Proposition}[Asymptotics - affine weights, aging, no fitness]
\label{prop-pkage_asym}
Consider the affine PA weights $f_k=ak+b$, an integrable aging function $g$, and denote the limiting degree distribution of the corresponding branching process by $(p_k)_{k\in\N}$. Then, under the hypotheses of Lemma \ref{Lem-adaptLap-age}, there exists a constant $C>0$ such that, as $k\rightarrow\infty$, 
\eqn{
\label{for-pkage_asym}
	p_k = \frac{\Gamma(k+b/a)}{\Gamma(k+1)}\left(Cg(t_k)-\frac{g'(t_k)}{g(t_k)}\right)^{1/2}\e^{-\alpha^* t_k}(1-\e^{-aG(\infty)})^kD_k(g)(1+o(1)),
} 
where 
$$
	D_k(g) = \frac{1}{2}+\frac{1}{2\sqrt{\pi}}\int_{-C_k(g)}^{C_k(g)}\e^{-\frac{u^2}{2}}du,
$$
and $C_k(g) = t_k\left(Cg(t_k)-\frac{g'(t_k)}{g(t_k)}\right)^{1/2}$.
\end{Proposition}

\subsection{Aging and fitness case}
\label{sec-adapLapFit}
In this section, we investigate the asymptotic behavior of the limiting degree distribution of a CTBP, in the case of affine PA weights. The method we use is analogous to that in Section \ref{sec-adaptedLap}. 

We assume that the fitness $Y$ is absolutely continuous with respect to the Lebesgue measure, and we denote its density function by $\mu$.  The limiting degree distribution of this type of branching process is given by
\eqn{
\label{for-pkagefit_integ}
	p_k = \pr\left(V_{YG(T_{\alpha^*})}=k\right) = \frac{\Gamma(k+b/a)}{\Gamma(b/a)\Gamma(k+1)}\int_{\R^+\times \R^+}\alpha^*\e^{-\alpha^*t}\mu(s)\e^{-bsG(t)}\left(1-\e^{-asG(t)}\right)^k dsdt.
}
We immediately see that the degree distribution has exponential tails when the fitness distribution is bounded:

\begin{Lemma}[Exponential tails for integrable aging and bounded fitnesses]
\label{lem-exp-tails-aging-bd-fitness}
When there exists $\gamma$ such that $\mu([0,\gamma])=1$, i.e., the fitness has a bounded support, then
	\eqn{
	p_k\leq \frac{\Gamma(k+b/a)}{\Gamma(b/a)\Gamma(k+1)} \left(1-\e^{-a \gamma G(\infty)}\right)^k.
	}
In particular, $p_k$ has exponential tails.
\end{Lemma}

\begin{proof} Obvious.
\end{proof}

Like in the situation with only aging, the explicit solution of the integral in \eqref{for-pkagefit_integ} may be hard to find. We again have to adapt the Laplace method to estimate the asymptotic behavior of the integral. We  write
\eqn{
\label{for-psi_k_bivar}
	I(k) := \int_{\R^+\times \R^+}\e^{-k\Psi_k(t,s)}dsdt,
}
where
\eqn{
\label{forPSi2def}
	\Psi_k(t,s) := \frac{\alpha^*}{k}t+\frac{b}{k}sG(t)-\frac{1}{k}\log\mu(s)-\log(1-\e^{-saG(t)}).
}
As before, we want to minimize the function $\Psi_k$.  We state here the lemma:

\begin{Lemma}[Adapted Laplace method 2]
\label{lem-adapLapfitAge}
Let $\Psi_k(t,s)$ as in \eqref{forPSi2def}. Assume that
\begin{enumerate}
	\item $g$ satisfies the assumptions of Lemma \ref{Lem-adaptLap-age};
	\item $\mu$ is twice differentiable on $\R^+$;
	\item there exists a constant $B'>0$ such that, for every $s\geq B'$, $\mu$ is monotonically decreasing;
	\item $(t_k,s_k)$ is the unique point where both partial derivatives are zero;
	\item $(t_k,s_k)$ is the absolute minimum for $\Psi_k(t,s)$;
	\item  the hessian matrix $H_k(t_k,s_k)$ of $\Psi_k(t,s)$ evaluated in $(t_k,s_k)$ is positive definite.
\end{enumerate} 
Then,
$$
	I(k) = \e^{-k\Psi_k(t_k,s_k)}\frac{2\pi}{\sqrt{\mathrm{det}(kH_k(t_k,s_k))}}\pr\left(\mathcal{N}_1(k)\geq -t_k,\mathcal{N}_2(k)\geq -s_k\right)(1+o(1)),
$$
where $(\mathcal{N}_1(k),\mathcal{N}_2(k)) := \mathcal{N}(\sub{0},(kH_k(t_k,s_k))^{-1})$ is a bivariate normal distributed vector and  $\sub{0} = (0,0)$.
\end{Lemma}
The proof  of Lemma \ref{lem-adapLapfitAge} can be found in Appendix \ref{sec-app-agefit}. Using Lemma \ref{lem-adapLapfitAge} we can describe the limiting degree distribution $(p_k)_{k\in\N}$: 
\begin{Proposition}[Asymptotics - affine weights, aging, fitness]
\label{prop-pkasym_fitage}
Consider affine PA weights $f_k = ak+b$, an integrable aging function $g$ and a fitness distribution density $\mu$. Assume that the corresponding branching process is supercritical and Malthusian. Under the hypotheses of Lemma \ref{lem-adapLapfitAge}, the limiting degree distribution $(p_k)_{k\in\N}$ of the corresponding $\CTBP$ satisfies
$$
	p_k = \frac{k^{b/a-1}}{\Gamma(b/a)}\frac{2\pi}{\sqrt{\mathrm{det}(kH_k(t_k,s_k))}}\e^{-k\Psi_k(t_k,s_k)}\pr\left(\mathcal{N}_1\geq -t_k,\mathcal{N}_2\geq -s_k\right)(1+o(1)).
$$ 
\end{Proposition}

\subsection{Three classes of fitness distributions}
\label{sec-fitexamples}

Proposition \ref{prop-pkasym_fitage} in Section \ref{sec-adapLapFit} gives the asymptotic behavior of the limiting degree distribution of a $\CTBP$ with integrable aging and fitness. Lemma \ref{lem-adapLapfitAge} requires conditions under which the function $\Psi_k(t,s)$ as in \eqref{forPSi2def} has a unique minimum point denoted by $(t_k,s_k)$. In this section we consider the three different classes of fitness distributions that we have introduced in Section \ref{sec-res-aging-fitness-exp}.

For the heavy-tailed class, i.e., for distributions with  tail thicker than exponential, there is nothing to prove. In fact, \eqref{for-AgeFitLap} immediately implies that such distributions are explosive.

For the other two cases,
we apply Proposition \ref{prop-pkasym_fitage}, giving the precise asymptotic behavior of the limiting degree distributions of the correponding $\CTBPs$. Propositions \ref{prop-expfit_general} and \ref{prop-subexpfit} contain the results respectively on the general-exponential and sub-exponential classes. The proof of these propositions are moved to Appendix \ref{sec-proof-3class}.

\begin{Proposition}
\label{prop-expfit_general}
Consider a general exponential fitness distribution as in \eqref{def-generalexpfit}. Let $(M_t)_{t\geq0}$ be the corresponding birth process.
Denote the unique minimum point of $\Psi_k(t,s)$ as in \eqref{forPSi2def} by $(t_k,s_k)$. Then
\begin{enumerate}
	\item for every $t\geq0$, $M_t$ has a dynamical power law with exponent $
\tau(t) = 1+\frac{\theta}{aG(t)}$;
	\item the asymptotic behavior of the limiting degree distribution $(p_k)_{k\in\N}$ is given by
	$$
		p_k= \e^{-\alpha^* t_k}h(s_k)\left(\tilde{C}-\alpha^*\frac{g'(t_k)}{g(t_k)}\right)^{-1/2}k^{-(1+\theta/(aG(\infty)))}(1+o(1)),
	$$
	where the power law term has exponent $\tau = 1+\theta/aG(\infty)$;
	\item the distribution $(q_k)_{k\in\N}$ of the total number of children of a fixed individual has a power law behavior with exponent $\tau = 1+\theta/aG(\infty)$.
\end{enumerate}
\end{Proposition}
By \eqref{for-varpYcond} it is necessary to consider the exponential rate $\theta>aG(\infty)$ to obtain a non-explosive process. In particular, this implies that, for every $t\geq0$, $\tau(t)$,  as well as $\tau$, are  strictly larger than $2$. As a consequence, the three distributions $(P_k[M](t))_{k\in\N}$, $(p_k)_{k\in\N}$ and $(q_k)_{k\in\N}$ have finite first moment. Increasing the value of $\theta$ leads to power-law distributions with exponent larger than $3$, so with finite variance. 

A second observation is that, independently of the aging function $g$, the point $s_k$ is of order $\log k$. In particular, this has two consequences. First the correction to the power law given by $h(s_k)$ is a power of $\log k$. Since $h'(s)/h(s)\rightarrow0$ as $s\rightarrow\infty$.  Second the power-law term $k^{-(1+\theta/(aG(\infty)))}$ arises from $\mu(s_k)$. This means that the exponential term in the fitness distribution $\mu$ not only is necessary to obtain a non-explosive process, but also generates the power law.

The third observation is that the behavior of the three distributions $(P_k[M](t))_{k\in\N}$, $(p_k)_{k\in\N}$ and $(q_k)_{k\in\N}$ depends on the integrability of the aging function, {\em but does only marginally depends on its precise shape}. 
The contribution of the aging function $g$ to the exponent of the power law in fact is given only by the value $G(\infty)$. The other terms that depend directly on the shape of $g$ are $\e^{-\alpha ^* t_k}$ and the ratio $g'(t_k)/g(t_k)$. 
The ratio $g'/g$ does not contribute for any function $g$ whose decay is in between power law and exponential. The term $\e^{-\alpha ^* t_k}$ depends on the behavior of $t_k$, that can be seen as roughly  $g^{-1}(1/\log k)$. For any function between power law and exponential, $\e^{-\alpha^* t_k}$ is asymptotic to a power of $\log k$.

The last observation is that every distribution in the general exponential class shows a {\em dynamical power law} as for the pure exponential distribution, as shown in Section \ref{sec-expfitness}. The pure exponential distribution is a special case where we consider $h(s)\equiv 1$. Interesting is the fact that $\tau$ actually does not depend on the choice of $h(s)$, but only on the exponential rate $\theta>aG(\infty)$. In particular, Proposition  \ref{prop-expfit_general} proves that the limiting degree distribution of the two examples in Figure \ref{fig-pwrlwtime} have power-law decay.

We move to the class of sub-exponential fitness. We show that the power law is lost due to the absence of a pure exponential term. We prove the result using densities of the form
\eqn{
\label{def-subfit-2}
	\mu(s) = C\e^{-s^{1+\varepsilon}},
}
for $\varepsilon>0$ and $C$ the normalization constant.
 The result is the following:

\begin{Proposition}
\label{prop-subexpfit}
Consider a sub-exponential fitness distribution as in \eqref{def-subfit-2}. Let $(M_t)_{t\geq0}$ be the corresponding birth process. 
Denote the minimum point of $\Psi_k(t,s)$ as in \eqref{forPSi2def} by $(t_k,s_k)$. Then
\begin{enumerate}
	\item for every $t\geq0$, $M_t$ satisfies
	$$
		\pr\left(M_t=k\right) = k^{-1}(\log k)^{-\varepsilon/2}\e^{-\frac{\theta}{(aG(t))^{1+\varepsilon}} (\log k)^{1+\varepsilon}}(1+o(1));
	$$
	\item the limiting degree distribution $(p_k)_{k\in\N}$ of the $\CTBP$ has asymptotic behavior given by
	$$
		p_k=  \e^{-\alpha^* t_k}k^{-1}\left(C_1-s_k^\varepsilon\frac{g'(t_k)}{g(t_k)}\right)\e^{-\frac{\theta}{(aG(\infty))^{1+\varepsilon}} (\log k)^{1+\varepsilon}}(1+o(1));
	$$
	\item the distribution $(q_k)_{k\in\N}$ of the total number of children of a fixed individual satisfies
	$$
			q_k = k^{-1}(\log k)^{-\varepsilon/2}\e^{-\frac{\theta}{(aG(\infty))^{1+\varepsilon}} (\log k)^{1+\varepsilon}}(1+o(1)).
	$$
\end{enumerate}
\end{Proposition}
In Proposition \ref{prop-subexpfit} the distributions $(P_k[M](t))_{k\in\N}$, $(p_k)_{k\in\N}$ and $(q_k)_{k\in\N}$ decay faster than a power law. This is due to the fact that a sub-exponential tail for the fitness distribution does not allow the presence of sufficiently many individuals in the branching population whose fitness value is sufficiently high to restore the power law.

In this case, we have that $s_k$ is roughly $c_1\log k-c_2\log\log k$. Hence, as first approximation, $s_k$ is still of logarithmic order. 
The power-law term is lost because there is no pure exponential term in the distribution $\mu$. In fact, in this case $\mu(s_k)$ generates the dominant term $\e^{-\theta (\log k)^{1+\varepsilon}}$.

\subsection{The case of exponentially distributed fitness: Proof of Corollary \ref{th-degexpfitness}}
\label{sec-expfitness}
The case when the fitness $Y$ is exponentially distributed turns out to be simpler. In this section, denote the fitness by $T_\theta$, where $\theta$ is the parameter of the exponential distribution. First of all, we investigate the Laplace transform of the process. In fact, we can write
$$
	\E\left[M_{T_\alpha}\right] = \int_0^\infty\theta\e^{-\theta s}\E\left[V_{sG(T_\alpha)}\right]ds,
$$
which is the Laplace transform of the stationary process $(V_{sG(T_\alpha)})_{s\geq0}$ with bounded fitness $G(T_\alpha)$ in $\theta$. As a consequence,
$$
	\E\left[M_{T_\alpha}\right] = \sum_{k\in\N}\E\left[\prod_{i=0}^{k-1}\frac{f_iG(T_\alpha)}{\theta+f_iG(T_\alpha)}\right].
$$
Suppose that there exists a Malthusian parameter $\alpha^*$. This means that, for fixed $(f_k)_{k\in\N}$, $g$ and $\theta$, $\alpha^*$ is the unique value such that $\E\left[M_{T_{\alpha^*}}\right]=1$. As a consequence, if we fix $(f_k)_{k\in\N}$, $g$ and $\alpha^*$, $\theta$ is the unique value such that 
$$
	\sum_{k\in\N}\E\left[\prod_{i=0}^{k-1}\frac{f_iG(T_\alpha)}{\theta+f_iG(T_\alpha)}\right]=1.
$$
Therefore $\theta$ is the Malthusian parameter of the process $(V_{sG(T_\alpha)})_{s\geq0}$. We are now ready to prove Corollary \ref{th-degexpfitness}:
\medskip


\begin{proof}[Proof of Corollary \ref{th-degexpfitness}]
We can write $\pr\left(M_t = k\right) = \pr\left(V_{T_\theta G(t)}=k\right)$, which means that we have to evaluate the Laplace transform of $\pr\left(V_{sG(t)}=k\right)$ in $\theta$. Using \eqref{for-Pktstat} the first part follows immediately by simple calculations. For the second part, we just need to take the limit as $t\rightarrow\infty$. For the sequence $(p_k)_{k\in\N}$, the result is immediate since $p_k = \E[P_k[M](T_{\alpha^*})]$.
\end{proof}
\medskip

The case of affine PA weights $f_k = ak+b$ is particularly nice. As already mentioned in Section \ref{sec-mainres}, the process  $(M_t)_{t\geq0}$ has a power-law distribution at every $t\in\R^+$ and \eqref{pkt-formula} follows immediately.
Further, \eqref{for-expfit-pk} and \eqref{exp-fitness-entire} follow directly.
\qed

%

\appendix

\section{Limiting distribution with aging effect, no fitness}
In this section, we analyze the limiting degree distribution $(p_k)_{k\in\N}$ of CTBPs with aging but no fitness. In Section \ref{sec-app-age} we prove the adapted Laplace method for the general asymptotic behavior of $p_k$. In Section \ref{sec-examples-age} we consider some examples of aging function $g$, giving the asymptotics for the corresponding distributions.

\subsection{Proofs of Lemma \ref{Lem-adaptLap-age} and Proposition \ref{prop-pkage_asym}}
\label{sec-app-age}

\begin{proof}[Proof of Lemma \ref{Lem-adaptLap-age}]
First of all, we show that $t_k$ is actually a minimum. In fact,
$$
	\lim_{t\rightarrow 0}\frac{d}{dt}\Psi_k(t) = -\infty,\quad\quad \mbox{and}\quad \lim_{t\rightarrow\infty}\frac{d}{dt}\Psi_k(t) = \frac{\alpha}{k}>0.
$$
As a consequence, $t_k$ is a minimum. Then,
\eqn{
\label{for-gtkasymp}
	\lim_{k\rightarrow\infty}g(t_k)\left(\frac{\alpha^*}{k}\frac{1-\e^{-aG(\infty)}}{a\e^{-aG(\infty)}}\right)^{-1} = \lim_{k\rightarrow\infty}g(t_k)\frac{ak}{\alpha^*(\e^{aG(\infty)}-1)}= 1.
}
In particular, $g(t_k)$ is of order $1/k$. Then, since $t_k$ is the actual minimum,  and $g$ is monotonically decreasing for $t\geq B$,
\eqn{
\label{for-Phi2}
	\Psi_k''(t_k)  = \frac{b}{k}g'(t_k)+g(t_k)^2 \frac{a^2\e^{-aG(t_k)}(2-\e^{-aG(t_k)})}{(1-\e^{-aG(t_k)})^2}-g'(t_k)\frac{a\e^{-aG(t_k)}}{1-\e^{-aG(t_k)}}>0.
}
We use the fact that we are evaluating the second derivative in the point $t_k$ where the first derivative is zero. This means
$$
	g(t_k)\frac{a\e^{-aG(t_k)})}{1-\e^{-aG(t_k)}}= \frac{\alpha}{k}+\frac{b}{k}g(t_k).
$$
We use this in \eqref{for-Phi2} to obtain
\eqn{
\label{for-Phi3}
	\begin{split}
		k\Psi_k''(t_k)  & = bg'(t_k)+g(t_k)\frac{a(2-\e^{-aG(t_k)})}{1-\e^{-aG(t_k)}}\left(\alpha+bg(t_k)\right)-\frac{g'(t_k)}{g(t_k)}\left(\alpha+bg(t_k)\right)\\
		& = g(t_k)\frac{a(2-\e^{-aG(t_k)})}{1-\e^{-aG(t_k)}}\left(\alpha+bg(t_k)\right)-\alpha\frac{g'(t_k)}{g(t_k)}.
	\end{split}
}
Now, we use Taylor expansion around $t_k$ of $\Psi_k(t)$ in the integral in \eqref{forIkdef}. Since we use the expansion around $t_k$, which is the minimum of $\Psi_k(t)$, the first derivative of $\Psi_k$ is zero. As a consequence, we have
$$
	I(k) = \int_0^\infty \e^{-k\left(\Psi_k(t_k)+\frac{1}{2}\Psi_k''(t_k)(t-t_k)^2+o((t-t_k)^2)\right)}dt.
$$
 First of all, notice that the contribution of the terms with $|t-t_k|\gg 1$ is negligible. In fact, we have
\eqn{
	\e^{-k\Psi_k(t)}\leq \e^{-\alpha^* t}(1-\e^{-aG(\infty)})^k,
}
which means that such terms are exponentially small, so we can ignore them. 
Now we make a change of variable $u = t-t_k$. Then
$$
	I(k) = \int_{-t_k}^\infty \e^{-k\left(\Psi_k(t_k)+\frac{1}{2}\Psi_k''(t_k)u^2+o(u^2)\right)}du.
$$
In particular, since the term $ \e^{-k\Psi_k(t_k)}$ does not depend on $u$, we can write
$$
	I(k) = \e^{-k\Psi_k(t_k)}\int_{-t_k}^\infty \e^{-k\left(\frac{1}{2}\Psi_k''(t_k)u^2+o(u^2)\right)}du.
$$
We use the notation $k\Psi_k(t_k)= \frac{1}{\sigma_k^2}$, which means we can rewrite the integral as
$$
	\e^{-k \Psi_k(t_k)}\sqrt{2\pi\sigma_k^2}\int_{-\infty}^{t_k}\frac{1}{\sqrt{2\pi\sigma_k^2}}\e^{-\frac{u^2}{2\sigma_k^2}}dt = \e^{-k\Psi_k(t_k)}\sqrt{2\pi\sigma_k^2}\pr\left(\mathcal{N}(0,\sigma_k^2)\leq t_k\right).
$$
Since the distribution $\mathcal{N}(0,\sigma_k^2)$ is symmetric with respect to $0$, for every $k\in\N$,
\eqn{
\label{for-prNterm}
	\pr\left(\mathcal{N}(0,\sigma_k^2)\leq t_k\right) = \frac{1}{2}\left[1+\frac{1}{\sqrt{\pi}}\int_{-t_k/\sigma_k}^{t_k/\sigma_k}\e^{-\frac{u^2}{2}}du\right].
}
The behavior of the above integral depends on the ratio $t_k/\sigma_k$, which is bounded between $0$ and $1$. As a consequence, the term $\pr\left(\mathcal{N}(0,\sigma_k^2)\leq t_k\right)$ is bounded between $1/2$ and $1$. 
\end{proof}

Using Lemma \ref{Lem-adaptLap-age}, we can prove Proposition \ref{prop-pkage_asym}:
\begin{proof}[Proof of Proposition \ref{prop-pkage_asym}]
Recall that $\sigma^2_k = (k\Psi(t_k)'')^{-1}$. Using \eqref{for-Phi3}, the fact that $g$ is bounded almost everywhere, and $g'(t_k)<0$, we can write
\eqn{
\label{for-var}
	k\Psi(t_k)''= \alpha\left(\frac{a(2-\e^{-aG(\infty)})}{1-\e^{-aG(\infty)}}g(t_k)-\frac{g'(t_k)}{g(t_k)}\right)(1+o(1)).
}
Notice that in \eqref{for-var} the terms $g(t_k)-\frac{g'(t_k)}{g(t_k)}$ are always strictly positive, since $g(t)$ is decreasing and $t_k\rightarrow\infty$ as $k\rightarrow\infty$.
As a consequence, we can replace the term $\sqrt{2\pi/\sigma^2_k}$ by $\left(Cg(t_k)-\frac{g'(t_k)}{g(t_k)}\right)^{1/2}$, for $C=\frac{a(2-\e^{-aG(\infty)})}{1-\e^{-aG(\infty)}}$. We also have that
$$
	\e^{-k\Psi_k(t_k)} = \mathrm{exp}\left[-\alpha^*t_k -bG(t_k)+k\log\left(1-\e^{-aG(t_k)}\right)\right]= \e^{-\alpha^*t_k}(1-\e^{-aG(\infty)})^k(1+o(1)),
$$
since  $G(t_k)$ converges to $G(\infty)$. For the term in \eqref{for-prNterm}, it is easy to show that it is asymptotic to $D_k(g)$. This completes the proof.
\end{proof}

\subsection{Examples of aging functions}
\label{sec-examples-age}
In this section, we analyze two examples of aging functions, in order to give examples of the limiting degree distribution of the branching process. We consider affine weights $f_k = ak+b$, and three different aging functions:
$$
	g(t) = \e^{-\lambda t},  \quad\quad  g(t) = (1+t)^{-\lambda}, \quad \quad\mbox{and}\quad \quad
	g(t) = \lambda_1\e^{-\lambda_2(\log (t+1)-\lambda_3)^2}.
$$
We assume that in every case the aging function $g$ is integrable, so we consider $\lambda>0$ for the exponential case, $\lambda>1$ for the power-law case and $\lambda_1,\lambda_2,\lambda_3>0$ for the lognormal case. We assume that $g$ satisfies Condition \eqref{cond-eg_inf1} in order to have a supercritical process. 

We now apply \eqref{for-pkage_asym} to these three examples, giving their asymptotics. In general, we approximate $t_k$ with the solution of, for $c_1 = \frac{a\e^{-aG(\infty)}}{1-\e^{-aG(\infty)}}$,
\eqn{
\label{for-tkappr}
	\frac{\alpha^*}{k}+\frac{b}{k}g(t)-c_1g(t) =0.
}
We start considering the exponential case $g(t) = \e^{-\lambda t}$. In this case, from \eqref{for-tkappr} we obtain that, ignoring constants,
\eqn{
\label{for-expage-1}
	t_k = \log k(1+o(1)).
}
As we expected, $t_k\rightarrow\infty$. We now use \eqref{for-var}, which gives a bound on $\sigma_k^2$ in \eqref{for-gtkasymp} in terms of $g$ and its derivatives. As a consequence,
$$
	\left(g(t_k)-\frac{g'(t_k)}{g(t_k)}\right)^{-1/2} = \left(\e^{-\lambda t_k}+\lambda\right)^{1/2}\sim \lambda^{1/2}(1+o(1)).
$$
Looking at $\e^{-k\Psi_k(t_k)}$, it is easy to compute that, with $t_k$ as in \eqref{for-expage-1},
$$
	\mathrm{exp}\left[-\alpha^* \log k+ bG(t_k)+k\log (1-\e^{-aG(t_k)})\right]= k^{-\alpha^*}(1+o(1)).
$$
Since $t_k/\sigma_k\rightarrow\infty$, then $\pr\left(\mathcal{N}(0,\sigma_k^2)\leq t_k\right)\rightarrow 1$, so that
$$
	p_k =\frac{\Gamma(k+b/a)}{\Gamma(b/a)}\frac{1}{\Gamma(k+1)}C_1 k^{-\alpha^*}\e^{-C_2 k}(1+o(1)),
$$
which means that $p_k$ has an exponential tail with power-law corrections.

We now apply the same result to the power-law aging function, so $g(t) = (1+t)^{-\lambda}$, and $G(t) = \frac{1}{\lambda-1}(1+t)^{1-\lambda}$. In this case
$$
	(1+t_k) = \left(\frac{\alpha^*}{c_1k}\right)^{-1/\lambda}(1+o(1)).
$$
We use again \eqref{for-gtkasymp}, so
$$
	\left(g(t_k)-\frac{g'(t_k)}{g(t_k)}\right) = \left(\frac{\alpha^*}{c_1k}+\lambda\left(\frac{c_1k}{\alpha^*}\right)^{1/\lambda}\right)^{1/2}\sim k^{\alpha^*/2\lambda}(1+o(1)).
$$
In conclusion,
$$
	p_k = \frac{\Gamma(k+b/a)}{\Gamma(b/a)}\frac{1}{\Gamma(k+1)}k^{\alpha^*/2\lambda}\e^{-\alpha^*\left(\frac{\alpha^*}{c_1k}\right)^{-1/\lambda}-C_2k}(1+o(1)),
$$
which means that also in this case we have a power-law with exponential truncation.

In the case of the lognormal aging function, \eqref{for-tkappr} implies that
$$
	[\log(t_k+1)-\lambda_3]^2 \approx  +\frac{1}{\lambda_2}\log\left(\frac{c_1}{\alpha^*}k\right).
$$
By \eqref{for-gtkasymp} we can say that
$$
	\left(g(t_k)-\frac{g'(t_k)}{g(t_k)}\right)=\left(\lambda_1\log\left(\frac{c_1}{\alpha^*}k\right)+
		2\lambda_2\frac{\log(t_k+1)}{t_k+1}\right)(1+o(1))= \lambda_1\log\left(\frac{c_1}{\alpha^*}k\right)(1+o(1)).
$$
We conclude then, for some constant $C_3>0$, 
$$
	p_k =  \frac{\Gamma(k+b/a)}{\Gamma(b/a)}\frac{1}{\Gamma(k+1)}\left(\lambda_1\log\left(\frac{c_1}{\alpha^*}k\right)\right)^{1/2}\e^{-\alpha^* \e^{(\log (\frac{c_1}{\alpha^*}k))^{1/2}}} \e^{-C_3 k}(1+o(1)).
$$

\section{Limiting distribution with aging and fitness}
\label{sec-appendix}
In this section, we consider birth processes with aging and fitness. We prove Lemma \ref{lem-adapLapfitAge}, used in the proof of Proposition \ref{prop-pkasym_fitage}. Then we give examples of limiting degree distributions for different aging functions and exponentially distributed fitness. 

\subsection{Proofs of Lemma \ref{lem-adapLapfitAge} and Proposition \ref{prop-pkasym_fitage}}
\label{sec-app-agefit}

\begin{proof}[Proof of Lemma \ref{lem-adapLapfitAge}]
We use again second order Taylor expansion of the function $\Psi_k(t,s)$ centered in $(t_k,s_k)$, where the first order partial derivatives are zero. As a consequence we write
$$
	\mathrm{exp}\left[-k\Psi_k(t,s)\right]  = 
\mathrm{exp}\left[- k\Psi_k(t_k,s_k)+\frac{1}{2}\sub{x}^T \left(kH_k(t_k,s_k)\right)\sub{x}+o(||\sub{x}||^2)\right] ,
$$
where 
$$
	\sub{x} = \left[\begin{array}{c} t-t_k\\ s-s_k\end{array}\right],\quad\quad \mbox{and}
	\quad \quad 
	H_k(t_k,s_k) = \left[
	\begin{array}{cc}
		\displaystyle\frac{\de^2\Psi_k}{\de t^2}(t_k,s_k) & \displaystyle\frac{\de^2\Psi_k}{\de s\de t} (t_k,s_k) \\
		& \\
		\displaystyle\frac{\de^2\Psi_k}{\de s\de t} (t_k,s_k) & \displaystyle\frac{\de^2\Psi_k}{\de s^2}(t_k,s_k)
	\end{array}\right].
$$
As for the proof of Lemma \ref{Lem-adaptLap-age}, we start by showing that we can ignore the terms where $||\sub{x}||^2\gg1$. In fact, 
$$
	\e^{-k\Psi_k(t,s)} \leq \mathrm{exp}\left(-\alpha^*t-bsG(t)+k\log(\mu(s))\right).
$$
Since $\mu$ is a probability density, $\mu(s)<1$ for $s\gg 1$. As a consequence, $\log(\mu(s))<0$, which means that the above bound is exponentially decreasing whenever $t$ and $s$ are very large. As a consequence, we can ignore the contribution given by the terms where $|t-t_k|\gg1$ and $|s-s_k|\gg1$.

The term $\e^{- k\Psi_k(t_k,s_k)}$ is independent of $t$ and $s$, so we do not consider it in the integral. Writing $u= t-t_k$ and $v= s-s_k$, we can write
$$
	\int_{\R^+\times\R^+}\e^{- \frac{1}{2}\sub{x}^T (kH_k(t_k,s_k))\sub{x}}dsdt = \int_{-t_k}^\infty\int_{-s_k}^\infty 
	\e^{- \frac{1}{2}\sub{y}^T (kH_k(t_k,s_k))\sub{y}}dudv,
$$
where this time $\sub{y}^T = [u~ v]$. As a consequence, 
\eqn{
\label{for-I_kasympFit}
	\int_{-t_k}^\infty\int_{-s_k}^\infty 
		\e^{- \frac{1}{2}\sub{y}^T (kH_k(t_k,s_k))\sub{y}}dudv = \frac{2\pi}{\sqrt{\mathrm{det}(kH_k(t_k,s_k))}}
		\pr\left(\mathcal{N}_1(k)\geq -t_k,\mathcal{N}_2(k)\geq -s_k\right),
}
provided that the covariance matrix $(kH_k(t_k,s_k))^{-1}$ is positive definite. 
\end{proof}
As a consequence, we can use \eqref{for-I_kasympFit} to obtain that, for the corresponding limiting degree distribution of the branching process $(p_k)_{k\in\N}$, as $k\rightarrow\infty$, 
$$
	p_k=\frac{\Gamma(k+b/a)}{\Gamma(b/a)}\frac{1}{\Gamma(k+1)}\e^{- k\Psi_k(t_k,s_k)}\frac{2\pi}{\sqrt{\mathrm{det}(kH_k(t_k,s_k))}}
		\pr\left(\mathcal{N}_1(k)\geq -t_k,\mathcal{N}_2(k)\geq -s_k\right)(1+o(1)).
$$
This results holds if the point $(t_k,s_k)$ is the absolute minimum of $\Psi_k$, and the Hessian matrix is positive definite at $(t_k,s_k)$. 

\subsection{The Hessian matrix of $\Psi_k(t,s)$}
\label{sec-app-hessian}
First of all, we need to find a point $(t_k,s_k)$ which is the solution of the system
\begin{eqnarray}
	&\displaystyle\frac{\de \Psi_k}{\de t} &= \displaystyle \frac{\alpha^*}{k}+\frac{b}{k}sg(t)-\frac{sag(t)\e^{-saG(t)}}{1-\e^{-saG(t)}} = 0  \label{for-psipart_t},\\
	&\displaystyle\frac{\de \Psi_k}{\de s} &=\displaystyle \frac{b}{k}G(t)-\frac{1}{k}\frac{\mu'(s)}{\mu(s)}-\frac{aG(t)\e^{-saG(t)}}{1-\e^{-saG(t)}}=0  \label{for-psipart_s}.
\end{eqnarray}
Denote the solution by $(t_k,s_k)$. Then
\eqn{
	\begin{split}
		\frac{\de^2\Psi_k}{\de t^2} & = \frac{b}{k}sg'(t_k)+g(t_k)^2 \frac{s^2a^2\e^{-asG(t_k)}}{(1-\e^{-asG(t_k)})^2}-g'(t_k)\frac{as\e^{-asG(t_k)}}{1-\e^{-asG(t_k)}},\\
		\frac{\de^2\Psi_k}{\de s^2} & = -\frac{1}{k}\frac{\mu''(s)\mu(s)-\mu'(s)^2}{\mu(s)^2}+
			\frac{a^2G(t)^2\e^{-saG(t)}}{(1-\e^{-saG(t)})^2},\\
		\frac{\de^2\Psi_k}{\de s\de t} & = \frac{b}{k}g(t_k) +\left(1-\frac{1}{as_k}\right)\left(\frac{b}{k}G(t_k)-\frac{1}{k}\frac{\mu'(s_k)}{\mu(s_k)}\right)\left(\frac{\alpha^*}{k}+\frac{b}{k}s_kg(t_k)\right).
	\end{split}
}
From \eqref{for-psipart_t} and \eqref{for-psipart_s} we know 
\eqn{
\label{for-substit}
\begin{split}
	\frac{\alpha^*}{k}+\frac{b}{k}s_kg(t_k) & = \frac{s_kag(t_k)\e^{-s_kaG(t_k)}}{1-\e^{-s_kaG(t_k)}},\\
	\frac{b}{k}G(t_k)-\frac{1}{k}\frac{\mu'(s_k)}{\mu(s_k)}& = \frac{aG(t_k)\e^{-s_kaG(t_k)}}{1-\e^{-s_kaG(t_k)}}.
\end{split}
}
Using \eqref{for-substit} in the expressions for the second derivatives,
\eqn{
	\begin{split}
		\frac{\de^2\Psi_k}{\de t^2} & = \frac{b}{k}s_kg'(t_k) + \frac{as_kg(t_k)}{(1-\e^{-as_kG(t_k)})}\left(\frac{\alpha}{k}+\frac{b}{k}s_kg(t_k)\right)-\frac{g'(t_k)}{g(t_k)}\left(\frac{\alpha}{k}+\frac{b}{k}s_kg(t_k)\right)\\
		& = \frac{as_kg(t_k)}{(1-\e^{-as_kG(t_k)})}\left(\frac{\alpha}{k}+\frac{b}{k}s_kg(t_k)\right)-\frac{\alpha}{k}\frac{g'(t_k)}{g(t_k)},
	\end{split}
}
\eqn{
	\begin{split}
		\frac{\de^2\Psi_k}{\de s^2} & =-\frac{1}{k}\frac{\mu''(s_k)\mu(s_k)-\mu'(s_k)^2}{\mu(s_k)^2}+\frac{aG(t_k)}{1-\e^{-as_kG(t_k)}}\left(\frac{b}{k}G(t_k)-\frac{1}{k}\frac{\mu'(s_k)}{\mu(s_k)}\right)\\
		& = -\frac{1}{k}\frac{\mu''(s_k)}{\mu(s_k)}+\frac{1}{k}\left(\frac{\mu'(s_k)}{\mu(s_k)}\right)^2-\frac{1}{k}\frac{\mu'(s_k)}{\mu(s_k)}\frac{aG(t_k)}{1-\e^{-as_kG(t_k)}}+\frac{1}{k}\frac{abG(t_k)^2}{1-\e^{-as_kG(t_k)}}.
	\end{split}
}
In conclusion, the matrix $kH_k(t_k,s_k)$ is given by
\eqn{
\label{for-hessian-element}
	\begin{split}
		\left(kH_k(t_k,s_k)\right)_{1,1} & = 
				\frac{as_kg(t_k)}{(1-\e^{-as_kG(t_k)})}\left(\alpha+bs_kg(t_k)\right)-\alpha\frac{g'(t_k)}{g(t_k)};\\
		\left(kH_k(t_k,s_k)\right)_{2,2} & = 
		-\frac{\mu''(s_k)}{\mu(s_k)}+\left(\frac{\mu'(s_k)}{\mu(s_k)}\right)^2-\frac{\mu'(s_k)}{\mu(s_k)}\frac{aG(t_k)}{1-\e^{-as_kG(t_k)}}+\frac{abG(t_k)^2}{1-\e^{-as_kG(t_k)}};\\
		\left(kH_k(t_k,s_k)\right)_{2,1}& = 
				bg(t_k) +\left(1-\frac{1}{as_k}\right)\left(\frac{b}{k}G(t_k)-\frac{1}{k}\frac{\mu'(s_k)}{\mu(s_k)}\right)\left(\alpha^*+bs_kg(t_k)\right).
	\end{split}
}
We point out that, solving \eqref{for-psipart_s} in terms of $s$, it follows that
\eqn{
\label{for-solution_s}
	s = \frac{1}{aG(t)}\log\left(1+k\frac{aG(t)}{bG(t)-\frac{\mu'(s)}{\mu(s)}}\right).
}
As a consequence,
\eqn{
\label{for-sgt-asym}
	s_kg(t_k) = \alpha^* G(t_k)\frac{\mu(s_k)}{\mu'(s_k)}.
}
We use \eqref{for-solution_s}, \eqref{for-sgt-asym} and the expressions for the elements of the Hessian matrix given in \eqref{for-hessian-element} for the examples in Section \ref{sec-append-exfit}. We also use the formulas of this section in the proof of Propositions \ref{prop-expfit_general} and \ref{prop-subexpfit} given in Section \ref{sec-proof-3class}.

\subsection{Examples of aging functions}
\label{sec-append-exfit}
Here we give examples of limiting degree distributions. We consider the same three examples of aging functions we considered in Section \ref{sec-examples-age}, so
$$
	g(t) = \e^{-\lambda t},  \quad\quad  g(t) = (1+t)^{-\lambda}, \quad \quad\mbox{and}\quad \quad
	g(t) = \lambda_1\e^{-\lambda_2(\log (t+1)-\lambda_3)^2}.
$$
 We consider exponentially distributed fitness, so $\mu(s) = \theta\e^{-\theta s}$. In order to have a supercritical and Malthusian process, we can rewrite Condition \eqref{for-varpYcond} for exponentially distributed fitness as
$aG(\infty)<\theta<(a+b)G(\infty)$. 

In general, we  identify the minimum point $(t_k,s_k)$, then use \eqref{for-asympt_expfit-precise}. For all three examples, replacing $G(t)$ by $G(\infty)$ and using \eqref{for-solution_s}, it holds that
$$
	s_k \approx \frac{1}{aG(\infty)}\log\left(k\frac{aG(\infty)}{bG(\infty)+\theta}\right),
$$
and $s_kg(t_k)\approx \alpha^* G(\infty)\theta$. For the exponential aging function, using \eqref{for-sgt-asym}, it follows that
$\e^{-\lambda t_k}\approx \log k$.
In this case, since $g'(t)/g(t) = -\lambda$, the conclusion is that, ignoring the constants,
$$
	p_k = k^{-(1+\lambda\theta/a)}(\log k)^{\alpha^*/\lambda}(1+o(1)).
$$
For the inverse-power aging function $t_k\approx (\log k)^{1/\lambda}$, which implies (ignoring again the constants) that
$$
	p_k= k^{-(1+(\lambda-1)\theta/a)}\e^{-\alpha^*(\log k)^{1/\lambda}}(1+o(1)),
$$
where we recall that, for $g$ being integrable, $\lambda>1$. For the lognormal case, 
$$
	t_k\approx \e^{\left(\log k\right)^{1/2}},
$$
which means that
$$
	p_k = k^{-(1+\theta/aG(\infty))}\e^{-\alpha^* \e^{\left(\log k\right)^{1/2}}}(1+o(1)).
$$

\section{Proof of propositions \ref{prop-expfit_general} and  \ref{prop-subexpfit}}
\label{sec-proof-3class}
In the present section, we prove Propositions \ref{prop-expfit_general} and  \ref{prop-subexpfit}. These proofs are applications of Proposition \ref{prop-pkasym_fitage}, and mainly consist of computations. In the proof of the two propositions, we often refer to Appendix \ref{sec-app-hessian} for expressions regarding the Hessian matrix of $\Psi_k(t,s)$ as in \eqref{forPSi2def}.

\subsection{Proof of Proposition \ref{prop-expfit_general}} 
We start by proving the existence of the dynamical power-law. We already know that
\eqn{
\label{for-prop-fitexp-1}
	\pr\left(M_t=k\right) = \frac{\Gamma(k+b/a)}{\Gamma(b/a)\Gamma(k+1)}\int_0^
\infty \mu(s)\e^{-bsG(t)}\left(1-\e^{-asG(t)}\right)^kds.
}
We write
\eqn{
\label{for-jk-1}
	J(k) = \int_0^\infty \e^{-k\psi_k(s)}ds,
}
where
\eqn{
\label{for-psik-classes}
\psi_k(s) = \frac{bG(t)}{k}s-\frac{1}{k}\log(\mu(s))-\log\left(1-\e^{-asG(t)}
\right).
}
In order to give asymptotics on $J(k)$ as in \eqref{for-jk-1}, we can use a Laplace method similar to the one used in the proof of Lemma \ref{Lem-adaptLap-age}, but the analysis is simpler since in this case $\psi_k(s)$ is a function of only one variable. The idea is again to find a minimum point $s_k$ for $\psi_k(s)$, and to use Taylor expansion inside the integral, so
$$
	\psi_k(s) = \psi_k(s_k)+\frac{1}{2}{\psi''}_k(s_k)(s-s_k)^2+o((s-s_k)^2).
$$
We can ignore the contribution of the terms where $(s-s_k)^2\gg 1$, since $\e^{-k\psi_k(s)}\leq \e^{-bsG(t)}$, so that the error is at most exponentially small. As a consequence,
\eqn{
	J(k) = \sqrt{\frac{\pi}{{\psi''}_k(s_k)}}\e^{-k\psi_k(s_k)}(1+o(1)).
}
The minimum $s_k$ is a solution of
\eqn{
	\frac{d\psi_k(s)}{ds} = \frac{bG(t)}{k}-\frac{1}{k}\frac{\mu'(s)}{\mu(s)}-
\frac{aG(t)\e^{-saG(t)}}{1-\e^{-asG(t)}}=0.
}
In particular, $s_k$ satisfies the following equality, which is similar to \eqref{for-solution_s}:
\eqn{
	s_k = \frac{1}{aG(t)}\log\left(1+k\frac{aG(t)}{bG(t)-\mu'(s_k)/\mu(s_k)}
\right).
}
When $\mu(s) = Ch(s)\e^{-\theta s}$,
\eqn{
	\frac{\mu'(s)}{\mu(s)} = \frac{h'(s)\e^{-\theta s}-\theta h(s)\e^{-\theta s
}}{h(s)\e^{-\theta s}} = -\theta\left(1-\frac{h'(s)}{\theta h(s)}\right)\approx -\theta.
}
In particular, this implies
\eqn{
	s_k= \frac{1}{aG(t)}\log\left(1+k\frac{aG(t)}{bG(t)+\theta}\right)(1+o(1)).
}
Similarly to the element $\left(kH_k(t_k,s_k)\right)_{2,2}$ in \eqref{for-hessian-element}, 
\eqn{
	k\frac{d^2\psi_k(s_k)}{ds^2} = -\frac{\mu''(s_k)}{\mu(s_k)}+\left(\frac{\mu'(s_k)}{\mu(s_k)}\right)^2-\frac
{\mu'(s_k)}{\mu(s_k)}\frac{aG(t)}{1-\e^{-as_kG(t)}}+\frac{abG(t)^2}{1-\e^{-as_kG(t)}}.
}
For the general exponential class, the ratio
$$
	\frac{\mu''(s_k)}{\mu(s_k)} = \frac{h''(s)}{h(s)}-2\theta+\theta^2.
$$
As a consequence, $k\frac{d^2\psi_k(s)}{ds^2}$ converges to a positive constant, which means that $s_k$ is an actual minimum. Then $J(k) = c_1\e^{-k\psi_k(s_k)}(1+o(1))$. Using this in \eqref{for-prop-fitexp-1} and ignoring the constants, 
\eqn{
\begin{split}
	\pr\left(M_t=k\right) & = \frac{\Gamma(k+b/a)}{\Gamma(b/a)\Gamma(k+1)}\e^{-s_kbG(t)}\mu(s_k)(1+o(1))\\
		& =k^{-1}k^{-b/a} k^{b/a}h(s_k)k^{-\theta/aG(t)}(1+o(1))= h(s_k)k^{-(1+\theta/aG(t))}(1+o(1)),
\end{split}
}
which is a power-law distribution with exponent $\tau(t) = 1+\theta/aG(t)$, and minor corrections given by $h(s_k)$. This holds for every $t\geq0$. In particular, considering $G(\infty)$ instead of $G(t)$, with the same argument we can also prove  that the distribution of the total number of children obeys a power-law tail with exponent $\tau(\infty) = 1+\theta/aG(\infty)$.

We now prove the result on the limiting distribution $(p_k)_{k\in\N}$ of the $\CTBP$, for which we apply directly Proposition  \ref{prop-pkasym_fitage}, using the analysis on the Hessian matrix given in Section \ref{sec-app-hessian}. First of all, from \eqref{for-solution_s} it follows that
\eqn{
\label{for-proof-genexp1}
	s_k= \frac{1}{aG(t_k)}\log\left(1+k\frac{aG(t_k)}{bG(t_k)+\theta}\right)(1+o(1)),
}
and by \eqref{for-sgt-asym} 
\eqn{
\label{for-proof-genexp2}
	s_kg(t_k)\stackrel{k\rightarrow\infty}{\longrightarrow}\alpha\frac{G(\infty)}{\theta}.
}
For the Hessian matrix, using \eqref{for-proof-genexp1} and \eqref{for-proof-genexp1} in \eqref{for-hessian-element}, for any integrable aging function $g$ we have 
$$
		\left(kH_k(t_k,s_k)\right)_{2,2} = C_2+o(1)>0,\quad \quad \mbox{and}\quad \quad 
		\left(kH_k(t_k,s_k)\right)_{2,1}= o(1),
$$
	but $\left(kH_k(t_k,s_k)\right)_{1,1}$ behaves according to  $g'(t_k)/g(t_k)$. If this ratio is bounded, then $\left(kH_k(t_k,s_k)\right)_{1,1}= C_1+o(1)>0$, while $\left(kH_k(t_k,s_k)\right)_{1,1}\rightarrow\infty$ whenever $g'(t_k)/g(t_k)$ diverges. In both cases, 
 $(t_k,s_k)$ is a minimum. In particular, again ignoring the multiplicative constants and using \eqref{for-proof-genexp1} and \eqref{for-proof-genexp2} in the definition of $\Psi_k(t,s)$, the limiting degree distribution of the $\CTBP$ is asymptotic to
\eqn{
\label{for-asympt_expfit-precise}
	k^{-(1+\theta/(aG(t_k)))}h(s_k)\e^{-\alpha^* t_k}\left(\tilde{C}-\alpha^*\frac{g'(t_k)}{g(t_k)}\right)^{-1/2},
}
where the term $\left(\tilde{C}-\alpha^*\frac{g'(t_k)}{g(t_k)}\right)^{-1/2}$, which comes from the determinant of the Hessian matrix, behaves differently according to the aging function. With this, the proof of Proposition \ref{prop-expfit_general} is complete.
\qed

\subsection{Proof of Proposition \ref{prop-subexpfit}}
This proof is identical to the proof of Proposition \ref{prop-expfit_general}, but this time we consider a sub-exponential distribution. First, we start looking at the distribution of the birth process at a fixed time $t\geq0$. We define $\psi_k(s)$ and $J(k)$ as in \eqref{for-psik-classes} and \eqref{for-jk-1}. We use again \eqref{for-prop-fitexp-1}, so
$$
	s_k = \frac{1}{aG(t)}\log\left(1+k\frac{aG(t)}{bG(t)-\mu'(s_k)/\mu(s_k)}
\right).
$$
In this case, we have
\eqn{
\label{for-subexp-proof-1}
	\frac{\mu'(s)}{\mu(s)} = -\theta(1+\varepsilon)s^\varepsilon.
}
Then $s_k$ satisfies 
\eqn{
	s_k = \frac{1}{aG(t)}\log\left(1+k\frac{aG(t)}{bG(t)+\theta(1+\varepsilon)s_k^\varepsilon}\right).
}
By substitution, it is easy to check that $s_k$ is approximatively  $c_1\log k-c_2\log\log k = \log k(1-\frac{\log\log k}{\log k})$, for some positive constants $c_1$ and $c_2$. This means that as first order approximation, $s_k$ is still of logarithmic order. Then, 
\eqn{
\label{for-subexp-proof-2}
	\frac{\mu''(s)}{\mu(s)} = \theta^2(1+\varepsilon)^2s^{2\varepsilon}-\theta(1+\varepsilon)\varepsilon s^{\varepsilon-1}.
}
Using \eqref{for-subexp-proof-1} and \eqref{for-subexp-proof-2}, we can write
\eqn{
\begin{split}
	k\frac{d^2\psi_k(s)}{ds^2}  = &\theta(1+\varepsilon)\varepsilon s_k^{\varepsilon-1}+\theta(1+\varepsilon)s_k^\varepsilon\frac{aG(t)}{1-\e^{-as_kG(t)}}+\frac{abG(t)^2}{1-\e^{-as_kG(t)}}\\
	=& \theta(1+\varepsilon)\varepsilon s_k^{\varepsilon-1}+\theta(1+\varepsilon)\frac{s_k^{\varepsilon}}{k}
		(bG(t)+\theta(1+\varepsilon)s_k^\varepsilon)\\
		&+ \frac{bG(t)}{k}(bG(t)+\theta(1+\varepsilon)s_k^\varepsilon).
\end{split}
}
The dominant term is $c_1s_k^{\varepsilon}$, for some constant $c_1$. This means $k\frac{d^2\psi_k(s)}{ds^2}$ is of order $(\log k)^{\varepsilon}$. Now, 
\eqn{
\begin{split}
	J(k) &= \left(k\frac{d^2\psi_k(s)}{ds^2}\right)^{-1/2}C\e^{-bG(t)s_k-\theta s_k^{1+\varepsilon}+k\log(1-\e^{-aG(t)s_k})}(1+o(1)) \\
	& =  (\log k)^{-\varepsilon/2}k^{-b/a}\e^{-\theta (\log k)^{1+\varepsilon}}(1+o(1)).
\end{split}
}
As a consequence,
\eqn{
	\pr\left(M_t = k\right) = k^{-1}(\log k)^{-\varepsilon/2}\e^{-\theta (\log k)^{1+\varepsilon}}(1+o(1)),
}
which is not a power-law distribution. Again using similar arguments, we show that the limiting degree distribution of the $\CTBP$ does not show a power-law tail. In this case
$$
s_k = \frac{1}{aG(t_k)}\log\left(1+k\frac{aG(t_k)}{bG(t_k)+\theta(1+\varepsilon)s_k^\varepsilon}\right),
$$
and 
$$
	s_kg(t_k) = \frac{\alpha G(t_k)}{\theta(1+\varepsilon)s_k^\varepsilon}= \frac{\alpha G(t_k)}{\log^\varepsilon k}(1+o(1))\rightarrow0.
$$
The Hessian matrix elements are
\eqn{
\begin{split}
	(kH_k(t_k,s_k))_{1,1} &= \frac{a\alpha^2 G(t_k)}{s_k^\varepsilon}-a\alpha \frac{g'(t_k)}{g(t_k)}+o(1),\\
	(kH_k(t_k,s_k))_{2,2} &=\theta(1+\varepsilon)\varepsilon s_k^{\varepsilon-1}+\theta s_k^\varepsilon aG(\infty)+abG(\infty)^2+o(1),\\
	(kH_k(t_k,s_k))_{1,2} &= o(1).
\end{split}
}
This implies that
\eqn{
	\mathrm{det}\left(kH_k(t_k,s_k)\right) = C_1-s_k^\varepsilon\frac{g'(t_k)}{g(t_k)}+o(1)>0.
}
As a consequence, $(t_k,s_k)$ is an actual minimum. Then using the definition of $\Psi_k(t,s)$, 
\eqn{
	p_k=\e^{-\alpha^* t_k}k^{-1+b/a}k^{-b/a}\mu(s_k)\sim \e^{-\alpha^* t_k}k^{-1}\left(C_1-s_k^\varepsilon\frac{g'(t_k)}{g(t_k)}\right)\e^{-\frac{\theta}{(aG(\infty))^{1+\varepsilon}} (\log k)^{1+\varepsilon}}(1+o(1)).
}
This completes the proof.
\qed

\bigskip

\noindent
{\bfseries Acknowledgments.}
We are grateful to Nelly Litvak and Shankar Bhamidi for discussions on preferential attachment models and their applications, and Vincent Traag and Ludo Waltman from CWTS for discussions about citation networks as well as the use of Web of Science data. 
This work is supported in part by the Netherlands Organisation for Scientific Research (NWO) through the Gravitation {\sc Networks} grant 024.002.003. The work of RvdH is further supported by the Netherlands Organisation for Scientific Research (NWO) through VICI grant 639.033.806.

\printbibliography[title=References, heading = bibintoc]
\end{document}